\documentclass[11pt,a4paper,oneside,centertags,reqno]{amsart}

\usepackage[title]{appendix}

\usepackage{etoolbox}
\robustify{\footnote}

\usepackage[multiple]{footmisc}
\usepackage{txfonts}
\usepackage{exscale}
\usepackage{amsbsy}
\usepackage[colorlinks=true,citecolor=red,pagebackref,hypertexnames=false]{hyperref}
\usepackage{lmodern}
\usepackage{endnotes}
\usepackage{latexsym}
\usepackage{amsmath}
\usepackage{amsfonts}
\usepackage{amssymb}
\usepackage{graphpap}
\usepackage{epsfig}
\usepackage{amscd}
\usepackage{amsthm}
\usepackage{enumerate}
\usepackage{eucal}
\usepackage[francais]{[babel}
\usepackage[deutsch]{[babel}
\usepackage[ansinew]{inputenc}
\usepackage{graphics}
\usepackage{dsfont}

\usepackage[backrefs,non-compressed-cites,initials,nobysame,]{amsrefs}
\usepackage{pdfsync}
\usepackage{color}
\usepackage{mathrsfs}

\usepackage{bbm}

\font\got=eufm10 at 11pt

\font\posebni=msam10

\renewcommand{\Re}[0]{{\rm Re}\,}
\renewcommand{\Im}[0]{{\rm Im}\,}

\newcommand{\C}[0]{{\mathbb C}}

\newcommand{\f}[0]{\varphi}

\newcommand{\N}[0]{{\mathbb N}}

\newcommand{\R}[0]{\mathbb{R}}

\newcommand{\BS}[0]{\mathbb{S}}

\newcommand{\leqsim}[0]{\,\text{\posebni \char46}\,}
\newcommand{\geqsim}[0]{\,\text{\posebni \char38}\,}

\newcommand{\cA}[0]{{\mathcal A}}
\newcommand{\cB}[0]{{\mathcal B}}

\newcommand{\cE}[0]{{\mathcal E}}

\newcommand{\cH}[0]{{\mathcal H}}
\newcommand{\cI}[0]{{\mathcal I}}
\newcommand{\cJ}[0]{{\mathcal J}}

\newcommand{\cM}[0]{{\mathcal M}}

\newcommand{\cO}[0]{{\mathcal O}}

\newcommand{\cQ}[0]{{\mathcal Q}}
\newcommand{\cR}[0]{{\mathcal R}}
\newcommand{\cS}[0]{{\mathcal S}}

\newcommand{\cV}[0]{{\mathcal V}}
\newcommand{\cW}[0]{{\mathcal W}}

\newcommand{\oA}[0]{{\mathscr A}}
\newcommand{\oB}[0]{{\mathscr B}}
\newcommand{\oC}[0]{{\mathscr C}}
\newcommand{\oD}[0]{{\mathscr D}}

\newcommand{\oL}[0]{{\mathscr L}}

\newcommand{\oV}[0]{{\mathscr V}}
\newcommand{\oW}[0]{{\mathscr W}}

\newcommand{\ovR}[0]{\overline{\rm R}}

\newcommand{\gota}[0]{{\text{\got a}}}
\newcommand{\gotb}[0]{{\text{\got b}}}

\newcommand\bS{\mathbf{S}}
\newcommand\bP{\mathbf{P}}

\newcommand{\mn}[2]{\{ #1 : #2 \}}
\newcommand{\Mn}[2]{\left\{ #1 : #2 \right\}}
\newcommand{\sk}[2]{\left\langle #1 , #2\right\rangle}

\renewcommand{\div}[0]{{\rm div}\,}

\newcommand{\Dom}[0]{{\rm D}}

\newtheorem{theorem}{Theorem}
 \newtheorem{defi}[theorem]{Definition}
\newtheorem{lemma}[theorem]{Lemma}
\newtheorem{proposition}[theorem]{Proposition}
\newtheorem{corollary}[theorem]{Corollary}

\theoremstyle{definition}

\renewcommand\leq[0]{\leqslant}
\renewcommand\geq[0]{\geqslant}
\renewcommand\epsilon[0]{\varepsilon}
\renewcommand\theta[0]{\vartheta}

\newcommand\wrt{\,\text{\rm d}}
\renewcommand\mod[1]{\left\vert{#1}\right\vert}

\newcommand\norm[2]{{\left\Vert{#1}\right\Vert_{#2}}}

\newtheorem{preremark}[theorem]{Remark}  \newenvironment{remark}%
{\begin{preremark}\rm}{\end{preremark}}

\begin{document}

\bibliographystyle{plain}

\title[Bilinear embedding on domains]{Bilinear embedding for perturbed divergence-form operator with complex coefficients on irregular domains}
\author[Poggio]{Andrea Poggio}

\date{\today}



\address{Andrea Poggio \\Universit\`a degli Studi di Genova\\ Dipartimento di Matematica\\ Via Dodecaneso\\ 35 16146 Genova\\ Italy }
\email{poggio@dima.unige.it}

\begin{abstract}
Let $\Omega\subseteq\R^{d}$ be open, $A$ a complex uniformly strictly accretive $d\times d$ matrix-valued function on $\Omega$ with $L^{\infty}$ coefficients, $b$ and $c$ two $d$-dimensional vector-valued functions on $\Omega$ with $L^{\infty}$ coefficients and $V$ a locally integrable nonegative function on $\Omega$.   Consider the operator $\oL^{A,b,c,V}=-\div(A\nabla) + \sk{\nabla}{b} - \div(c \, \cdot) + V $ with mixed boundary conditions on $\Omega$. We extend the bilinear inequality that Carbonaro and Dragi\v{c}evi\'c proved in \cite{CD-Potentials} in the special cases when $b=c = 0$,  previously proved  in \cite{CD-Mixed} when $V=0$ as well. As a consequence, we obtain that the solution to the parabolic problem $u^{\prime}(t)+\oL^{A,b,c,V}u(t)=f(t)$, $u(0)=0$, has maximal regularity in $L^{p}(\Omega)$, for all $p>1$ such that $A$ satisfies the $p$-ellipticity condition that Carbonaro and Dragi\v{c}evi\'c introduced in \cite{CD-DivForm} and $b,c,V$ satisfy another condition that we introduce in this paper. Roughly speaking, $V$ has to be ``big'' with respect to $b$ and $c$.
We do not impose any conditions on $\Omega$,  in particular, we do not assume any regularity of $\partial\Omega$, nor the existence of a Sobolev embedding. 
\end{abstract}

\maketitle

\section{Introduction}
\label{s: Neumann introduction}
Let $\Omega\subseteq\R^{d}$ be a nonempty open set. Denote by $\cA(\Omega)$ the class of all complex uniformly strictly elliptic $d\times d$ matrix-valued functions on $\Omega$ with $L^{\infty}$ coefficients (in short, elliptic matrices). That is to say, $\cA(\Omega)$ is the class of all measurable $A:\Omega\rightarrow \C^{d\times d}$ for which there exist $\lambda$, $\Lambda>0$ such that for almost all $x\in\Omega$ we have
\begin{eqnarray*}
\label{eq: N ellipticity}
\Re\sk{A(x)\xi}{\xi}
&\hskip -19pt\geq \lambda|\xi|^2\,,
&\quad\forall\xi\in\C^{d};
\\
\label{eq: Neumann bounded}
\mod{\sk{A(x)\xi}{\sigma}}
&\hskip-6pt\leq \Lambda \mod{\xi}\mod{\sigma}\,,
&\quad\forall\xi,\sigma\in\C^{d}.
\end{eqnarray*}
We denote by $\lambda(A)$ and $\Lambda(A)$ the optimal $\lambda$ and $\Lambda$, respectively.

Suppose that $A\in\cA(\Omega)$. Let $b,\,c$ be $d$-dimensional vector-valued functions on $\Omega$ with $L^{\infty}$ coefficients and $V \in L^1_{\text{loc}}(\Omega)$ a nonnegative function. Fix a closed subspace $\oV$ of the Sobolev space $H^1(\Omega)$ containing $H^1_{0}(\Omega)$.  Consider the sesquilinear form $\gota = \gota_{A,b,c,V,\oV}$ defined by
\begin{eqnarray}
\label{e : form sesq}
\Dom(\gota)&\hskip -121pt=\Mn{u\in\oV}
 {\int_\Omega V|u|^2 < \infty}, \nonumber 
\\
\gota(u,v)&= \displaystyle \int_{\Omega}\sk{A\nabla u}{\nabla v}_{\C^{d}} + \sk{\nabla u}{b}_{\C^d}\overline{v} + u\sk{c}{\nabla v}_{\C^d} + V u \overline{v}.
\end{eqnarray}
Clearly, $\gota$ is densely defined in $L^2(\Omega)$. Using the abbreviation $\gota(u)=\gota(u,u)$ and the {\it ad hoc} notation $\xi = \nabla u / u$, we find that
\begin{eqnarray*}
\aligned
\Re \gota (u) &= \int_\Omega |u|^2 \Bigl[\Re\sk{A\xi}{\xi} + \Re \sk{b+c}{\xi} +  V \Bigr], \\
\Im \gota(u) &= \int_\Omega |u|^2 \Bigl[\Im\sk{A\xi}{\xi} + \Im \sk{b-c}{\xi} \Bigr].
\endaligned
\end{eqnarray*}
Consequently, if we suppose that for almost all $x \in \Omega$
\begin{equation*}
\label{eq: accr and cont}
\Re\sk{A(x)\xi}{\xi} + \Re \sk{b(x)+c(x)}{\xi} +  V(x) \geq 0,  \quad \forall\xi\in\C^d,
\end{equation*}
then $\gota$ is accretive. Moreover, assume that there exist $\mu \in (0,1]$ and $M >0$ such that for almost all $x\in\Omega$
\begin{eqnarray}\label{eq: sect form below}
&\Re\sk{A(x)\xi}{\xi} + \Re \sk{b(x)+c(x)}{\xi} + V(x) \geq \mu (|\xi|^2 +  V(x)),  \quad \forall\xi\in\C^d, \\
\label{eq: sect form above}
&  |b(x)-c(x)| \leq M\sqrt{V}.
\end{eqnarray}
Then, following Ouhabaz \cite[the proof of Proposition~4.30]{O}, we get that $\gota$ is also closed. We denote by $\cB_{\mu,M}(\Omega)$ the class of all $(A,b,c,V) \in \cA(\Omega) \times (L^\infty(\Omega,\C^d))^2 \times L_{\text{loc}}^1(\Omega, \R_+)$ for which \eqref{eq: sect form below} and \eqref{eq: sect form above} hold for some fixed $\mu,M >0$. Moreover, we set
$$
\cB(\Omega) = \bigcup_{\mu,M>0} \cB_{\mu,M}(\Omega).
$$
For $\oA=(A,b,c,V)$, we denote by $\mu(\oA)$ and $M(\oA)$ the optimal $\mu$ and $M$, respectively.

Given $\phi \in (0,\pi)$ define the sector
$$
\bS_{\phi}=\mn{z\in\C\setminus\{0\}}{|\arg (z)|<\phi}.
$$
Also set $\bS_0=(0,\infty)$.
Boundedness of $A$, \eqref{eq: sect form below} and \eqref{eq: sect form above} imply that $\gota$ is sectorial of some angle $\theta_0 =\theta_0(\mu,M,\Lambda) \in(0, \pi/2)$ in the sense of \cite{Kat}, meaning that its numerical range $\text{Nr}(\gota) = \{\gota(u): u\in \Dom(\gota), \, \|u\|_2=1\}$ satisfies
\begin{equation}
\label{eq: sect numer range}
\text{Nr}(\gota) \subseteq \overline{\bS}_{\theta_0}.
\end{equation}
Denote by $\oL=\oL^{A,b,c,V,\oV}_{2}$ the unbounded operator on $L^{2}(\Omega)$ associated with the sesquilinear form $\gota$.
That is,
 $$
 \Dom(\oL):=\Mn{u\in\Dom(\gota)}
 {\exists w\in L^2(\Omega):\ 
 \gota(u,v)=\sk{w}{v}_{L^2(\Omega)}\ \forall v\in \Dom(\gota)}
$$
and 
\begin{equation}
\label{eq: ibp}
\sk{\oL u}{v}_{L^2(\Omega)} = \gota(u,v), \quad \forall u\in\Dom(\oL),\quad \forall v\in\Dom(\gota)\,.
\end{equation}
Formally, $\oL$ is given by the expression
$$
 \oL u=-\div(A\nabla u) + \sk{\nabla u}{b} -\div(u c) + V u.
$$
It follows from \eqref{eq: sect numer range} that $-\oL$ is the generator of a strongly continuous semigroup on $L^{2}(\Omega)$
$$
T_t =T^{A,b,c,V,\oV}_{t}, \quad t>0,
$$
 which is analytic and contractive in the cone $\bS_{\pi/2-\theta_0}$. Moreover, $\oL$ is sectorial of angle $\omega(\oL) \leq \theta_0$.
For details and proofs see \cite[Chapter VI]{Kat}, \cite[Chapters I and IV]{O} and \cite[Sections~2.1 and 3.4]{Haase}.

\subsection{Mixed boundary conditions}\label{s: boundary}
Given a closed set $\Gamma\subseteq \partial\Omega$ we define $H_\Gamma^1(\Omega)$ to be the closure in $H^1(\Omega)$ of the set
$$
\{u_{\vert_{\Omega}}: u\in C^{\infty}_{c}(\R^{d}\setminus \Gamma)\}.
$$

We shall always assume that $\oV$ is one of the following closed subspaces of $H^{1}(\Omega)$:
\begin{enumerate}
\item $\oV=H^1(\Omega)$, corresponding to {\it Neumann boundary} conditions for $\oL$, or
\item $\oV=H^1_{\Gamma}(\Omega)$, corresponding to {\it Dirichlet boundary} conditions in $\Gamma$ and Neumann conditions in $\partial\Omega\setminus \Gamma$ for $\oL$.
\end{enumerate}
The latter case includes 
Dirichlet boundary conditions ($\Gamma=\partial\Omega$) and  good-Neumann boundary conditions ($\Gamma=\emptyset$); see \cite[Section~4.1]{O}.

\subsection{Standard assumptions}\label{s : stand ass}
Unless otherwise specified, henceforth we assume that 
\begin{itemize}
\item $\Omega \subseteq \R^d$ is a nonempty open set;
\item $\oV$ and $\oW$ are two closed subspaces of $H^1(\Omega)$ of the type described in Section~\ref{s: boundary};
\item $A, B \in \cA(\Omega)$;
\item $b,c, \beta, \gamma \in L^\infty(\Omega, \C^d)$;
\item $V, W \in L^1_{\text{loc}}(\Omega, \R_+)$.
\end{itemize}
These assumptions shall be called the {\it standard assumptions}.

\subsection{Notation}
Given two quantities $X$ and $Y$, we adopt the convention whereby $X \leqsim Y$ means that there exists an absolute constant $C>0$ such that $X \leq C Y$. If both $X \leqsim Y$ and $Y \leqsim X$, then we write $X \sim Y$. If  $\{\alpha_1, \dots, \alpha_n\}$ is a set of parameters, then $C(\alpha_1, \dots, \alpha_n)$ denotes a constant depending only on $\alpha_1,\dots,\alpha_n$. When $X \leq C(\alpha_1, \dots, \alpha_n) Y$, we will often write $X \leqsim_{\alpha_1, \dots, \alpha_n} Y$.

Unless specified otherwise, for every $p \in (1,\infty)$ we will denote by $q$ its conjugate exponent, that is, $1/p+1/q =1$.

\subsection{The $p$-ellipticity condition} 
We summarize the following notion, which Carbonaro and Dragi\v{c}evi\'c introduced in \cite{CD-DivForm}.

Given $A \in \cA(\Omega) $ and $p \in [1, \infty]$, we say that $A$ is  {\it $p$-elliptic} if there exists $C=C(A,p) >0$ such that for a.e. $x \in \Omega$,
\begin{equation}
\label{eq: Delta>0}
\Re\sk{A(x)\xi}{\xi+|1-2/p| \overline{\xi}}_{\C^{d}}
\geq C |\xi|^2\,,
\quad \forall\xi\in\C^{d}.
\end{equation}
Equivalently, $A$ is $p$-elliptic if $\Delta_p(A) >0$, where
\begin{equation*}
\label{eq: N Deltap}
\Delta_p(A):=
\underset{x\in\Omega}{{\rm ess}\inf}\min_{|\xi|=1}\, \Re\sk{A(x)\xi}{\xi+|1-2/p| \overline{\xi}}_{\C^{d}}.
\end{equation*}

Denote by $\cA_p(\Omega)$ the class of all $p$-elliptic matrix functions on $\Omega$.
Clearly, $\cA(\Omega)=\cA_2(\Omega)$.
A bounded matrix function $A$ is real and elliptic if and only if it is $p$-elliptic for all $p>1$ \cite{CD-DivForm}. For further properties of the function $p\mapsto \Delta_{p}(A)$ we also refer the reader to \cite{CD-DivForm}.

Dindo\v s and Pipher in \cite{Dindos-Pipher} found a sharp condition which permits proving reverse H\"older inequalities for weak solutions to ${\rm div}(A\nabla u)=0$ with complex $A$. It turned out that this condition was precisely a reformulation of $p$-ellipticity \eqref{eq: Delta>0}.

A condition similar to \eqref{eq: Delta>0}, namely $\Delta_p(A)\geq0$, was formulated in a different manner by Cialdea and Maz'ya in \cite[(2.25)]{CiaMaz}. See \cite[Remark 5.14]{CD-DivForm}.

\subsection{The classes $\cS_p$ and $\cB_p$}
Let $p >1$ and $ \oA=(A, b, c, V)$ satisfy the standard assumptions of Section \ref{s : stand ass}.  For every $p \in [1, +\infty]$,  we consider the operator
\begin{equation}
\label{e : operator Ip}
\cJ_{p}(\xi)=\frac{p}{2}\Biggl(\xi+\biggl(1-\frac{2}{p}\biggr)\overline{\xi}\,\Biggr),\quad \xi\in\C^{d}.
\end{equation}
Observe that $\cJ_p =\left (p/2 \right) \cI_p$, where $\cI_p$ is the operator introduced by Carbonaro and  Dragi\v{c}evi\'c in \cite[(1.4)]{CD-DivForm}. Moreover, we have:
\begin{enumerate}[(a)]\label{p : properties Jp}
\item \label{i : R-lin} $\cJ_p$ is $\R$-linear;
\item \label{i : form for Jp} $\cJ_p \xi = p \,\Re \xi - \overline{\xi} \,$ for all $\xi \in \C^d$;
\item \label{i : sym real Jp} $\Re \sk{\cJ_p\xi}{ \sigma} = \Re \sk{\xi}{\cJ_p \sigma} \, $ for all $\xi, \sigma \in \C^d$;
\item \label{i : eq def delta} $\left(p/2\right) \Delta_p(A) = \underset{x\in\Omega}{{\rm ess}\inf} \,\underset{|\xi|=1}{\min}\,\Re\sk{A(x)\xi}{\cJ_p\xi}_{\C^{d}}$ \cite[(1.6)]{CD-DivForm}.
\end{enumerate}
If $1/p + 1/q=1$, then also 
\begin{enumerate}[(a)]
\setcounter{enumi}{4}
\item \label{i : inv Jp} $\cJ_p \cJ_q = I_{\C^d}$  \cite[p. 3201]{CD-DivForm}; 
\item \label{i : Jp i z} $\cJ_p (i \xi) = i (p-1) \cJ_q \xi \,$ for all $\xi \in \C^d$;
\item \label{i : norm Jp} $ \|\cJ_p\| = p^*-1$, where $p^* = \max\{p, q\}$.
\end{enumerate}
We define the function $\Gamma_p^\oA=\Gamma_p^{A,b,c,V} : \Omega \times \C^d \rightarrow \R$ by 
\begin{equation}
\label{e : def Gammap}
\Gamma_p^\oA(x, \xi) =  \Re\sk{A(x)\xi}{\cJ_p \xi}_{\C^d} + \Re\sk{b(x)+ (\cJ_pc)(x)}{\xi}_{\C^d} +  V(x).
\end{equation}
By using dilations $\xi \leadsto t\xi$, $t \rightarrow \infty$, we see that if for almost all $x \in \Omega$
\begin{equation}\label{e : weak Gamma cond}
\Gamma_p^\oA(x,\xi) \geq 0, \quad \forall \xi \in \C^d,
\end{equation}
then $\Delta_p(A) \geq 0$. Moreover, when $b=c=0$ the two conditions coincide.
A condition slightly weaker than \eqref{e : weak Gamma cond} was formulated in a different manner by Cialdea and Maz'ya in \cite[(2.25)]{CiaMaz} (see Section \ref{ss : compar CM}).
\smallskip

 We denote by $\cS_p(\Omega)$ the class of all $\oA=(A,b,c,V) \in \cA_p(\Omega) \times (L^\infty(\Omega,\C^d))^2 \times L_{\text{loc}}^1(\Omega, \R_+)$ for which 
\begin{itemize}
\item there exists $\mu>0$ such that for a.e. $x \in \Omega$
\begin{equation}
\label{e : equiv B}
\Gamma_p^\oA(x, \xi) \geq \mu (|\xi|^2 +  V(x)),  \quad \forall \xi \in \C^d,
\end{equation}
\item the condition \eqref{eq: sect form above} holds.
\end{itemize}
We denote by $\mu_p(\oA)$, or just $\mu_p$, the largest admissible $\mu$ in \eqref{e : equiv B}. We list a few observations on $\cS_p$ and $\mu_p$:
\begin{itemize}
\item By applying the dilation argument again, we get
\begin{equation}\label{e : mu e delta}
 \mu_p(\oA) \leq \min\left\{ \frac{p}{2} \Delta_p(A), 1 \right\};
\end{equation}
\item $(A,b,c,V) \in \cS_2(\Omega)$ if and only if \eqref{eq: sect form below} and \eqref{eq: sect form above} hold, that is, $\cS_2(\Omega)=\cB(\Omega)$;
\item when $b=c=0$, this new condition coincides with $p$-ellipticity, in the sense that $(A,0,0,V) \in \cS_p(\Omega)$ if and only if $A \in \cA_p(\Omega)$ \cite{CD-DivForm, CD-Potentials}.
\end{itemize}
Finally, we define the class
$$
\cB_p(\Omega)=  \cS_p(\Omega)\cap \cS_q(\Omega), \quad \frac{1}{p}+\frac{1}{q}=1,
$$
which also coincides with $\cA_p(\Omega)$ when $b=c=0$, in the sense that $(A,0,0,V) \in \cB_p(\Omega)$ if and only if $A \in \cA_p(\Omega)$, by means of the invariance of $p$-ellipticity under conjugation \cite[Corollary 5.16]{CD-DivForm}. It also coincides with $\cB(\Omega)$ when $p=2$.
\smallskip

Our motivation for introducing $\cB_p$ on top of $\cS_p$ was to obtain a class which, other than generalising $p$-ellipticity, retains the following properties that the classes $\cA_p$ possess:
\begin{enumerate}[(i)]
\item invariance under conjugation, in the sense that $\cA_p(\Omega) = \cA_q(\Omega)$ when $1/p +1/q =1$ \cite[Proposition 5.8]{CD-DivForm};
\item a decrease with respect to $p$, in the sense that $\{ \cA_p(\Omega) \, : \, p \in [2, \infty) \}$ is a decreasing chain of matrix classes  \cite[Corollary 5.16]{CD-DivForm}.
\item invariance under adjointness, in the sense that $ A \in \cA_p(\Omega) \iff A^* \in \cA_q(\Omega)$, $1/p +1/q =1$ \cite[Corollary 5.17]{CD-DivForm}.
\end{enumerate}
We shall see that $\cS_p$ satisfies the analogue of $(iii)$ (see Proposition \ref{p: more on B}\eqref{i : invarianza per coniugazione}), but not $(i)$ and $(ii)$ (see Proposition \ref{p : neg asp Bp}). On the other hand,  $\cB_p$ is invariant under conjugation by definition and it turns out that it is also decreasing with respect to $p$; see Proposition \ref{p: more on B}\eqref{i : interpolation for B}. Therefore, the class $\cB_p$ seems a more adequate generalisation of $p$-ellipticity.

\subsection{Semigroup properties on $L^p$}\label{s: sem emb}
As our first result, we want to generalise \cite[Theorem~1.2]{CD-Potentials} through Theorem~\ref{t : contract}. See also the implication $(a) \Rightarrow (b)$ of \cite[Theorem~1.3]{CD-DivForm} and \cite[Proposition~1]{CD-Mixed}, where $\phi=0$, $b=c=0$, $V=0$. Carbonaro and Dragi\v{c}evi\'c proved each result combining a theorem of Nittka \cite[Theorem~4.1]{Nittka} either with  \cite[Theorem~4.7]{O} for \cite[Theorem~1.3]{CD-DivForm} and \cite[Proposition~1]{CD-Mixed}, or with \cite[Theorem~4.31]{O} for \cite[Theorem~1.2]{CD-Potentials}. The proof of Theorem~\ref{t : contract} is a modification of that from \cite[Theorem~1.2]{CD-Potentials}, the main difference being that we need a new condition which generalises that in \cite[Theorem~1.2]{CD-Potentials} (namely, $\Delta_p( e^{i \phi} A) \geq 0$) in order to get the $L^p$-dissipativity of the form. See Section~\ref{ss : Lp contr} for the explanation of terminology and the proof.
\begin{theorem}
\label{t : contract}
Let $\Omega, \oA=(A, b, c, V), \oV$ satisfy the standard assumptions of Section \ref{s : stand ass}. Suppose that $\oA \in \cB(\Omega)$ and choose $p>1$ and $\phi \in \R$ such that $|\phi| < \pi/2 - \theta_0$ and  \\$\Gamma_p^{(e^{i\phi} A, \, e^{i\phi} b, \, e^{i\phi}c, \, (\cos\phi) V)}(x,\xi) \geq 0$ for almost all $x \in \Omega$ and all $\xi \in \C^d$. Then
$$
\left( e^{-t e^{i\phi}\oL} \right)_{t >0}
$$ 
extends to a strongly continuous semigroup of contractions on $L^p(\Omega)$.
\end{theorem}
The next corollary, applied with $r$ equal to the conjugate exponent of $p$, extends \cite[Corollary 1.3]{CD-Potentials}.
\begin{corollary}\label{c: N analytic sem}
Let $\Omega, \oA=(A, b, c, V), \oV$ satisfy the standard assumptions of Section \ref{s : stand ass}. Suppose that $\oA \in \cB(\Omega)$ and choose $r \leq 2 \leq p$ such that $\oA \in (\cS_r \cap \cS_p)(\Omega)$. Then there exists $\theta = \theta(p,r,\oA) >0$ such that  $\{ T_z : z \in \bS_\theta \}$ is analytic and contractive in $L^s(\Omega)$ for all $ s \in [r,p]$.
\end{corollary}

\subsection{Bilinear embeddings for perturbed operators}
In case when $b=c=0$, in \cite[Theorem 1.4]{CD-Potentials} Carbonaro and Dragi\v{c}evi\'c proved that there exists $C>0$  independent of the dimension $d$ such that
\begin{equation}\label{eq: N bil no lower terms}
\aligned
\hskip -15pt \int^{\infty}_{0}\!\int_{\Omega}\sqrt{\mod{\nabla T^{A,V,\oV}_{t}f}^2 + V \mod{T^{A,V,\oV}_{t}f}^2}\sqrt{\mod{\nabla T^{B,W,\oW}_{t}g}^2 +  W \mod{T^{B,W,\oW}_{t}g}^2} \leq C \norm{f}{p}\norm{g}{q}, 
\endaligned
\end{equation}
for all $A,B \in \cA_p(\Omega)$, $V,W \in L_\text{loc}^1 (\Omega, \R_+)$ and all $f,g \in (L^p \cap L^q)(\Omega)$, where $\oV$ and $\oW$ are two closed subspaces of $H^{1}(\Omega)$ of the type described in Section~\ref{s: boundary} and $q= p/(p-1)$ is the conjugate exponent of $p$.

Given $b,c, \beta, \gamma \in L^\infty(\Omega, \C^d)$, we extend the bilinear embedding in \eqref{eq: N bil no lower terms} to the semigroups $(T^{\oA,\oV}_{t})_{t>0}$ and  $(T^{\oB,\oW}_{t})_{t>0}$, where $\oA=(A,b,c,V), \oB=(B,\beta,\gamma,W) \in \cB(\Omega)$.
In accordance with \cite{CD-DivForm, CD-Mixed, CD-Potentials}, we need a stronger condition than the one which implies the $L^p$ contractivity of such semigroups. Consistent with \cite{CD-DivForm, CD-Mixed, CD-Potentials}, it is convenient to introduce further notation:
$$
\aligned
\lambda(A,B) &= \min\{\lambda(A),\lambda(B)\} \\
\Lambda(A,B) &= \max\{\Lambda(A),\Lambda(B)\}\\
\mu(\oA,\oB) &= \min\{\mu(\oA), \mu(\oB)\}\\
M(\oA,\oB) &= \max\{M(\oA),M(\oB)\}
\endaligned
$$
and
$$
\mu_r(\oA,\oB) = \min\{\mu_r(\oA), \mu_{r^\prime}(\oA), \mu_r(\oB),  \mu_{r^\prime}(\oB)\},
$$
whenever $\oA, \oB \in \cB_r(\Omega)$ for some $r \in (1, \infty)$, where $r^\prime = r/(r-1)$ is its conjugate exponent.
In Section \ref{s : proof bil emb} we shall prove the following result.
\begin{theorem}\label{t: N bil}
Let $\Omega, \oA=(A,b,c,V), \oB=(B,\beta,\gamma,W), \oV, \oW$ satisfy the standard assumptions of Section \ref{s : stand ass}. Suppose that $\oA, \oB \in \cB(\Omega)$ and choose $p>1$, $q=p/(p-1)$, such that $\oA, \, \oB \in \cB_p(\Omega)$. Then there exists $C>0$ independent of the dimension $d$ such that
\begin{eqnarray}\label{eq: N bil}
\displaystyle\int^{\infty}_{0}\!\int_{\Omega}\sqrt{\mod{\nabla T^{\oA,\oV}_{t}f}^2 + V \mod{T^{\oA,\oV}_{t}f}^2}\sqrt{\mod{\nabla T^{\oB,\oW}_{t}g}^2 +  W \mod{T^{\oB,W,\oW}_{t}g}^2} \leq C \norm{f}{p}\norm{g}{q}, 
\end{eqnarray}
for all $f,g\in (L^{p}\cap L^{q})(\Omega)$. The constant $C>0$ can be chosen so as to (continuously) depend only on $p, \lambda(A,B),\Lambda(A,B), \mu(\oA,\oB), M(\oA,\oB)$ and $\mu_p(\oA,\oB)$.
\end{theorem}

This result incorporates several earlier theorems as special cases, including:
\begin{itemize}
\item  $b=c=0$, $V=W$, $\Omega = \R^d$, $A, B$ equal and real \cite[Theorem 1]{Dv-kato}
\item $b=c=0$, $V=W=0$, $\Omega = \R^d$ \cite[Theorem 1.1]{CD-DivForm}
\item $b=c=0$, $V=W=0$ \cite[Theorem 2]{CD-Mixed}
\item $b=c=0$ \cite[Theorem 1.4]{CD-Potentials}.
\end{itemize}

\subsection{Bilinear embeddings with complex potentials} \label{s : compl pot}
Let $\Omega, A, b, c, V, \oV$ satisfy the standard assumptions of Section \ref{s : stand ass}. Let $\varrho \in \C$ such that $\Re(\varrho) >0$. 
Let assume that $(A,b,c,\Re(\varrho) V) \in \cB(\Omega)$. Then, it can be shown in the same way as above that the sesquilinear form $\gota$, defined by
$$
\aligned
\label{e : form sesq compl}
\Dom(\gota)&=\Mn{u\in\oV}
 {\int_\Omega  V|u|^2 < \infty}, \nonumber 
\\
\gota(u,v)&= \displaystyle \int_{\Omega}\sk{A\nabla u}{\nabla v}_{\C^{d}} + \sk{\nabla u}{b}_{\C^d}\overline{v} + u\sk{c}{\nabla v}_{\C^d} + \varrho V u \overline{v},
\endaligned
$$
is densely defined, closed and sectorial. Therefore, $-\oL$, the operator on $L^2(\Omega)$ associated with $\gota$, generates an analytic and contractive semigroup $(T_t^{\oA,\oV})_{t>0}$ on $L^2(\Omega)$, where $\oA = (A,b,c, \varrho V)$. More generally, all the results mentioned above hold also for these types of complex potentials and are proved as in the case where the potentials are real, provided that their real part belongs to the new class introduced before. For instance, Theorem \ref{t: N bil} for complex potentials now reads as follows:
\begin{theorem}\label{t: N bil compl}
Let $\Omega, \oV, \oW$ satisfy the standard assumptions of Section \ref{s : stand ass}. Let $\oA=(A,b,c,\varrho V)$ and $\oB=(B,\beta,\gamma, \sigma W)$ be of the type described above. Choose $p>1$, $q=p/(p-1)$, such that $\tilde{\oA}=(A,b,c,\Re(\varrho) V), \, \tilde{\oB}=(B,\beta,\gamma, \Re(\sigma) W) \in \cB_p(\Omega)$. Then there exists $C>0$ independent of the dimension $d$ such that
\begin{equation}\label{eq: N bil compl}
\aligned
 \displaystyle\int^{\infty}_{0}\!&\int_{\Omega}\sqrt{\mod{\nabla T^{\oA,\oV}_{t}f}^2 + \Re(\varrho) V \mod{T^{\oA,\oV}_{t}f}^2}\sqrt{\mod{\nabla T^{\oB,\oW}_{t}g}^2 +  \Re(\sigma) W \mod{T^{\oB,\oW}_{t}g}^2} \\
& \leq C \norm{f}{p}\norm{g}{q}, 
\endaligned
\end{equation}
for all $f,g\in (L^{p}\cap L^{q})(\Omega)$. The constant $C>0$ can be chosen so as to (continuously) depend only on $p, \lambda(A,B),\Lambda(A,B), \mu(\tilde{\oA},\tilde{\oB}), M(\tilde{\oA},\tilde{\oB})$ and $\mu_p(\tilde{\oA},\tilde{\oB})$.
\end{theorem}
To avoid burdening the notation, we will treat in detail only the case with real potentials.

\subsection{Maximal regularity and functional calculus on domains}  
In case when $b=c=V=0$, let $A \in \cA_p(\Omega)$ and let $-\oL_p^A$ be the generator of $(T^{A,\oV}_{t})_{t>0}$ on $L^{p}(\Omega)$. Then $\oL_p^A$ admits a bounded holomorphic functional calculus of angle $\theta<\pi/2$ and has parabolic maximal regularity \cite[Theorem 3]{CD-Mixed}. 

Following the same argument of \cite[Theorem 3]{CD-Mixed}, by means of 
\begin{itemize}
\item elementary properties of the function $(p,A)\mapsto \Delta_{p}(A)$ (see \cite[Corollary~5.17 and p. 3204]{CD-DivForm})
\item elementary properties of the classes $\cS_p(\Omega)$ (see Proposition~\ref{p : property of B} \eqref{i : cont B in p}, \eqref{i : cont B in phi})
\item a well-known sufficient condition for bounded holomorphic functional calculus \cite[Theorem~4.6 and Example~4.8]{CDMY}
\item the Dore-Venni theorem \cite{DoreVenni,PrussSohr}
\item  Theorem~\ref{t: N bil} applied with $B=A^{*}$, $\beta=c$, $\gamma=b$, $W=V$ and $\oW=\oV$
\end{itemize} 
we can deduce the following result; see Section~\ref{s: max funct} for the explanation of terminology and the proof.
\begin{theorem}\label{t: N principal}
Let $\Omega, A, b, c, V, \oV$ satisfy the standard assumptions of Section \ref{s : stand ass}. Suppose that $p>1$ and $(A,b,c,V) \in \cB_p(\Omega)$. Let $-\oL_p$ be the generator of $(T_{t})_{t>0}$ on $L^{p}(\Omega)$. Then $\oL_p$ admits a bounded holomorphic functional calculus of angle $\theta <\pi/2$. As a consequence, $\oL_{p}$ has parabolic maximal regularity.
\end{theorem}

Recent results regarding the holomorphic functional calculus for the operator $\oL_p$ have been obtained by Egert \cite{Egert}. He considered elliptic systems of second order in divergence form with bounded coefficients and subject to mixed boundary conditions on bounded and connected open sets $\Omega$ whose boundary is Lipschitz regular around the Neumann part $\overline{\partial\Omega \setminus \Gamma}$. We recall that $\Gamma$ is a closed subset of $\partial\Omega$; see Section~\ref{s: boundary}. In \cite[Theorem 1.3]{Egert} he provided the optimal interval of $p$'s for the bounded $H^\infty$-calculus on $L^p$. More precisely, after defining the interval $\cJ(\oL)=\cJ(\oL^\oA)$ by
$$
\cJ(\oL) := \{p \in (1,\infty) : \sup_{t >0} \| T_t^\oA\|_{\cB(L^p)} < \infty \},
$$
he proved that, given $p_0 \in \cJ(\oL)$, the operator $\oL_p$ admits a bounded $H^\infty$-calculus on $L^p$ for all $p \in (p_0, 2) \cup (2,p_0)$. He used the regularity assumption on the domain $\Omega$ to prove that the semigroup satisfies specific off-diagonal estimates \cite[Definition~1.6 and Proposition~4.4]{Egert}, which imply the boundedness of the $H^\infty$-calculus \cite[Lemma~5.3]{Egert}. For this last implication, he did not need the previous further assumptions on $\Omega$. In particular, he exploited the existence of a bounded linear Sobolev extension operator that extends 
$$
\begin{array}{rcl}
H^1_\Gamma (\Omega) & \rightarrow & H^1_\Gamma(\R^d), \\
L^p(\Omega) &\rightarrow &L^p(\R^d),
\end{array}
$$
for every $p \in (1, \infty)$. Here, $H^1_\Gamma(\R^d)$ is the closure in $H^1(\R^d)$ of the set $C^{\infty}_{c}(\R^{d} \setminus \Gamma)$.

Since in generality that we consider, the domain $\Omega$ may be completely irregular and/or unbounded, we only deduce that our interval of $p$'s for the bounded $H^\infty$-calculus on $L^p$ is contained in Egert's. In fact, Theorem~\ref{t : contract} and Proposition~\ref{p : property of B}\eqref{i : cont B in p} imply that
$$
 \{p \in (1,\infty) : \oA \in \cB_p(\Omega) \} \subseteq {\rm int}\left(\cJ(\oL)\right).
$$
See also \cite[Section~1.5]{CD-Mixed} for a similar discussion regarding unperturbed divergence-form operators.

\subsection{Organization of the paper}
Here is the summary of each section.
\begin{itemize}
\item In Section \ref{s: Neumann gen conv} we summarize some of the main notions needed in the paper, extending in particular the generalised convexity introduced by Carbonaro and Dragi\v{c}evi\'c in \cite{CD-DivForm}.
\item In Section \ref{s : prop of Bp} we summarize some properties of the classes $\cS_p$ and $\cB_p$.
\item In Section \ref{s: conv of pow fun and Bell fun} we prove the desired convexity of the power functions in one complex variable and of the Bellman function.
\item In Section \ref{s : Lp contr} we prove the results on contractivity and analyticity of semigroups announced in Section \ref{s: sem emb}.
\item In Section \ref{s : proof bil emb} we prove the bilinear embedding.
\item In Section \ref{s: max funct} we prove Theorem \ref{t: N principal}.
\end{itemize}

\section{Heat-flow monotonicity and generalised convexity}\label{s: Neumann gen conv} 
\subsection{The Bellman function of Nazarov and Treil}
Let $\Omega$, $\oA=(A,b,c,V)$, $\oB=(B,\beta,\gamma,W), \oV, \oW$ satisfy the standard assumptions of Section \ref{s : stand ass}. Suppose that $\oA,\oB \in \cB(\Omega)$.
We want to study the monotonicity of the flow
\begin{equation*}\label{eq: N flow}
\cE(t)=\int_{\Omega}\cQ(T^{\oA,\oV}_{t}f,T^{\oB,\oW}_{t}g)
\end{equation*}
associated with a particular explicit {\it Bellman function} $\cQ$ invented by Nazarov and Treil \cite{NT}. Here we use a simplified variant introduced in \cite{Dv-kato} which comprises only two variables:
\begin{equation}\label{eq: N Bellman}
\cQ(\zeta,\eta)=
|\zeta|^p+|\eta|^{q}+\delta
\begin{cases}
 |\zeta|^2|\eta|^{2-q},& |\zeta|^p\leqslant |\eta|^q;\\
 (2/p)\,|\zeta|^{p}+\left(
 2/q-1\right)|\eta|^{q},&|\zeta|^p\geqslant |\eta|^q\,,
\end{cases}
\end{equation}
where $p>2$, $q=p/(p-1)$, $\zeta,\eta\in\C$ and $\delta>0$ is a positive parameter that will be fixed later. It was noted in \cite[p. 3195]{CD-DivForm} that $\cQ\in C^{1}(\C^{2})\cap C^{2}(\C^{2}\setminus\Upsilon)$, where
$$
\Upsilon=\{\eta=0\}\cup\{|\zeta|^p=|\eta|^q\}\,,
$$
and that for $(\zeta,\eta)\in\C\times\C$ we have
\begin{equation}
\label{eq: N 5}
\aligned
0\leqslant \cQ(\zeta,\eta) & \leqsim_{p,\delta}\,\left(|\zeta|^p+|\eta|^q\right),\\
|(\partial_{\zeta}\cQ)(\zeta,\eta)| & \leqsim_{p,\delta}\, \max\{|\zeta|^{p-1},|\eta|\},\\
|(\partial_{\eta}\cQ)(\zeta,\eta)| & \leqsim_{p,\delta}\, |\eta|^{q-1}\,,
\endaligned
\end{equation}
where $\partial_\zeta=\left(\partial_{\zeta_1}-i\partial_{\zeta_2}\right)/2$ and $\partial_\eta=\left(\partial_{\eta_1}-i\partial_{\eta_2}\right)/2$.
\medskip


\subsection{Real form of complex operators}We explicitly identify $\C^{d}$ with $\R^{2d}$ as follows. 
For each $d\in\N_{+}$ consider the operator
$\cV_{d}:\C^{d}\rightarrow\R^{d}\times\R^{d}$, defined by 
$$
\cV_{d}(\xi_{1}+i\xi_{2})=
(\xi_{1},\xi_{2}),\quad \xi_{1},\xi_{2}\in\R^{d}.
$$
One has 
\begin{equation}
\label{e : realis compl prod}
\Re \sk{\xi}{\xi^{\prime}}_{\C^{d}}=\sk{\cV_{d}(\xi)}{\cV_{d}(\xi^{\prime})}_{\R^{2d}},\quad \xi,\xi^{\prime}\in\C^{d}.
\end{equation}
If $A\in\C^{d\times d}$ we shall frequently use its real form:
$$
\cM(A)=\cV_{d}A\cV_{d}^{-1}=\left[
\begin{array}{rr}
\Re A  & -\Im A\\
\Im A  & \Re A
\end{array}
\right]\,.
$$
For $\xi,\sigma\in\C^{d}$ we have the following extension of \eqref{e : realis compl prod}:
\begin{equation}\label{eq: real form}
\aligned
\Re\sk{A\xi}{\sigma}_{\C^{d}}=
\sk{\cM(A)\cV_{d}(\xi)}{\cV_{d}(\sigma)}_{\R^{2d}}.
\endaligned
\end{equation}
Let $k,d\in\N_{+}$. We define another identification operator
$$
\cW_{k,d}:\underbrace{\C^{d}\times\cdots\times\C^{d}}_{k-{\rm times}}\longrightarrow \underbrace{\R^{2d}\times\cdots\times\R^{2d}}_{k-{\rm times}}
$$
 by the rule
$$
\cW_{k,d}(\xi^{1},\dots,\xi^{k})
=\left(\cV_{d}(\xi^{1}),\dots,\cV_{d}(\xi^{k})\right),\quad \xi^{j}\in\C^{d},\ j=1,\dots,k.
$$

\subsection{Convexity with respect to complex matrices, complex vectors and real scalars}\label{s: GeH}
In \cite{CD-DivForm}, the authors introduced the notion of {\it generalised convexity of a function with respect to a matrix} (or a collection of matrices). Here we extend this notion to $4$-tuples $(A,b,c,V)$. So the novelty is the presence of $b,c$ and $V$. Our aim is to demonstrate the value of this extended notion by proving that it implies properties of $\oL$ as shown earlier for $b=c=0$.

Due to the presence of the first- and zero-order terms, following \cite{CD-DivForm} we only treat the one-dimensional and the two-dimensional case separately, in order to make the text more readable. One could provide a single definition for the $k$-dimensional case, as Carbonaro and Dragi\v{c}evi\'c did in \cite{CD-Mixed}. 

Let $d \in \N_+$. Take
$$
\begin{array}{rcl}
A,B & \in & \C^{d \times d}, \\
b,\beta, c,\gamma  & \in & \C^d, \\
V,W  & \in & \R, \\
\omega= (\zeta, \eta) & \in &  \C \times \C,\\
\Xi = (X,Y) & \in &  \C^d \times \C^d.
\end{array}
$$
Denote 
$$
\aligned
\oA & = (A,b,c,V),\\
\oB &= (B,\beta, \gamma, W).
\endaligned
$$

\begin{subsubsection}{One-dimensional case}
Let $\phi:\C \rightarrow \R$ be of class $C^{2}$. 
We associate the function $\phi$ with the following function on $\R^{2}$:
\begin{equation*}
\phi_{\cW} := \phi \circ \cV_{1}^{-1}.
\end{equation*}

Denote, respectively, by $D^{2}\phi(\zeta)$ and $\nabla \phi (\zeta)$ the Hessian matrix and the gradient of the function $\phi_{\cW} : \R^{2} \rightarrow \R$ calculated at the point $\cV_{1}(\zeta) \in \R^2$.
In accordance with \cite{CD-DivForm,CD-Mixed} we define 
$$
H^{A}_{\phi}[\zeta;X]=
 \sk{\left[D^{2}\phi(\zeta)\otimes I_{\R^{d}}\right]\cV_{d}(X)}{\cM(A)\cV_{d}(X)}_{\R^{2d}},
$$
where $\otimes$ denotes the Kronecker product of matrices (see, for example, \cite{CD-DivForm}).

Moreover, we define
\begin{equation}\label{eq: def c-Hess 1dim} 
\aligned
H^{(b,c)}_\phi[\zeta;X] = & \sk{\left[D^{2}\phi(\zeta)\otimes I_{\R^{d}}\right] \cV_{d}(X)}{\cV_{d}(\zeta c)}_{\R^{2d}} \\
& + \,\sk{\nabla\phi(\zeta)}{\cV_{1}(\sk{X}{b})}_{\R^{2}}, \\
G^{V}_\phi[\zeta] =&   \sk{\nabla \phi(\zeta)}{\cV_{1}(V\zeta)}_{\R^{2}}. 
\endaligned
\end{equation}
Finally, we define
\begin{equation}
\label{d : Hess sup gen 1dim}
\mathbf{H}_\phi^{\oA}[\zeta; X] =H^{A}_{\phi}[\zeta;X]+H^{(b,c)}_\phi[\zeta; X]+G^{V}_\phi[\zeta].
\end{equation}
\begin{defi}
We say that  $\phi$ is $\oA$-{\it convex} in $\C$ if $\mathbf{H}^{\oA}_{\phi}[\zeta; X]$ is nonnegative for all $\zeta \in \C$, $X \in \C^{d}$. 
\end{defi}
In accordance with \cite{CD-Mixed} we say that $\phi$ is $A$-{\it convex} if it is $(A,0,0,0)$-convex.
\end{subsubsection}

\begin{subsubsection}{Two-dimensional case}
Let $\Phi:\C^{2} \rightarrow \R$ be of class $C^{2}$. 
We associate the function $\Phi$ with the following function on $\R^{4}$:
\begin{equation}
\label{e : realis of compl fun}
\Phi_{\cW} := \Phi \circ \cW_{2,1}^{-1}.
\end{equation}

Denote, respectively, by $D^{2}\Phi(\omega)$ and $\nabla \Phi (\omega)$ the Hessian matrix and the gradient of the function $\Phi_{\cW} : \R^{4} \rightarrow \R$ calculated at the point $\cW_{2,1}(\omega) \in \R^{4}$.
In accordance with \cite{CD-DivForm,CD-Mixed} we define 
$$
H^{(A,B)}_{\Phi}[\omega;\Xi]=
 \sk{\left[D^{2}\Phi(\omega)\otimes I_{\R^{d}}\right]\cW_{2,d}(\Xi)}{\left[\cM(A)\oplus\cM(B)\right]\cW_{2,d}(\Xi)}_{\R^{4d}},
$$
where $\cM(A)\oplus\cM(B)$ is the $4d \times 4d$ block diagonal real matrix with the $2d\times2d$ blocks $ \cM(A), \cM(B)$ along the main diagonal.

Moreover, we define
\begin{equation}\label{eq: def c-Hess}
\aligned
H^{(b,\beta, c,\gamma)}_\Phi[\omega;\Xi] = & \sk{\left[D^{2}\Phi(\omega)\otimes I_{\R^{d}}\right] \cW_{2,d}(\Xi)}{\cW_{2,d}(\zeta c, \eta \gamma)}_{\R^{4d}} \\
& + \, \sk{\nabla\Phi(\omega)}{\cW_{2,1}(\sk{X}{b}, \sk{Y}{\beta})}_{\R^{4}}, \\
G^{(V,W)}_\Phi[\omega] =&   \sk{\nabla \Phi(\omega)}{\cW_{2,1}(V\zeta, W\eta)}_{\R^{4}}. 
\endaligned
\end{equation}
Finally, we define
\begin{equation}
\label{d : Hess sup gen}
\mathbf{H}_\Phi^{(\oA,\oB)}[\omega; \Xi] =H^{(A,B)}_{\Phi}[\omega;\Xi]+H^{(b,\beta,c,\gamma)}_\Phi[\omega; \Xi]+G^{(V,W)}_\Phi[\omega].
\end{equation}
\begin{defi}
We say that  $\Phi$ is $(\oA,\oB)$-{\it convex} in $\C^{2}$ if $\mathbf{H}^{(\oA,\oB)}_{\Phi}[\omega; \Xi]$ is nonnegative for all $\omega\in \C^{2}$, $\Xi \in \C^{2d}$. 
\end{defi}
In accordance with \cite{CD-Mixed} we say that $\Phi$ is $(A,B)$-{\it convex} if it is $(A,0,0,0,B,0,0,0)$-convex.
\end{subsubsection}
\bigskip
\medskip

We maintain the same notation when instead of matrices we consider matrix-valued {\it functions} $A, B\in L^{\infty}(\Omega;\C^{d\times d})$, vector-valued functions $b,\beta,c,\gamma \in L^\infty(\Omega,\C^d)$ and scalar functions $V,W \in L^1_{\rm loc}(\Omega,\R_+)$; in this case however we require that all the conditions are satisfied for a.e. $x\in\Omega$. 

%

Given $f,g \in L^2(\Omega)$, define the function
$$
\cE(t)=\int_{\Omega}\Phi\left(T^\oA_{t}f,T^\oB_{t}g\right),\quad t>0.
$$
Definitions \eqref{eq: def c-Hess} and \eqref{d : Hess sup gen} are motivated by the fact that, formally,
\begin{eqnarray*}
-\cE^{\prime}(t)= \int_{\Omega}\mathbf{H}^{(\oA,\oB)}_{\Phi}\left[\left(T^\oA_{t}f,  T^\oB_{t}g\right);\left(\nabla T^\oA_{t}f, \nabla T^\oB_{t}g\right)\right]. 
\end{eqnarray*}
It follows that if $\Phi$ is $(\oA,\oB)$-convex on $\C^{2}$ then the function $\cE$ is nonincreasing on $(0,+\infty)$.
When $\Phi$ is {\it strictly} $(\oA,\oB)$-convex and satisfies a suitable size estimate, this formal method can be used for proving bilinear inequalities in the spirit of \cite{CD-DivForm,CD-mult,CD-OU, CD-Mixed}.

\section{Properties of the class $\cB_p(\Omega)$}\label{s : prop of Bp}
In this section we shall prove some elementary properties of $\cS_p(\Omega)$ and $\cB_p(\Omega)$ which we shall use in the rest of the paper. 

In particular, we shall see that the strict positivity of $\Gamma_p$ is preserved  for exponents close enough to $p$ and for small complex rotations of the coefficients (compare with \cite[Corollary 5.17 and p. 3204]{CD-DivForm}). Moreover, we shall see that the first order terms $b, c$ are appropriately controlled by the potential $V$.

\begin{proposition}\label{p : property of B}
Let $1 <p<\infty$, $q=p/(p-1)$. Suppose that $\oA= (A,b,c,V) \in \cS_p(\Omega)$. Then the following assertions hold.
\begin{enumerate}[{\rm (i)}]
\item
\label{i : V dom bc}
For every $s \in (1,\infty)$ there exists $\widetilde{C}>0$ such that for a.e. $x \in \Omega$.
$$
\left|b(x) +( \cJ_s c)(x)\right| \leq \widetilde{C} \sqrt{V(x)}.
$$
The constant $\widetilde{C}$ can be chosen so as to (continuously) depend only on $p, s, \Lambda(A)$, $M(\oA)$ and $\mu_p(\oA)$.
\item
\label{i : V dom Rec}
There exists $C>0$ such that $|b(x)|, \,| c(x)| \le C \sqrt{V(x)}$ for a.e. $x \in \Omega$. 
The constant $C$ can be chosen so as to (continuously) depend only on $p, \Lambda(A), M(\oA)$ and $\mu_p(\oA)$.
\item
\label{i : cont B in p}
There exists $\varepsilon >0$ such that $(A,b,c,V) \in \cS_{s}(\Omega)$ for all $s \in [p-\varepsilon, p+\varepsilon]$.
\item
\label{i : cont B in phi}
There exists $\vartheta \in (0,\pi/2)$ such that $(e^{i\varphi}A,e^{i\varphi}b,e^{i\varphi}c,(\cos \varphi) V) \in \cS_{p}(\Omega)$ for all $\varphi \in [-\vartheta, \vartheta]$.
\end{enumerate}
\end{proposition}
\begin{proof}
We will be using properties \eqref{i : R-lin}-\eqref{i : norm Jp} of the operator $\cJ_p$, recalled on page \pageref{p : properties Jp}.
\begin{enumerate}[{\rm (i)}]
\item
First consider the case $s=p$. Set $\xi_0 = b+ \cJ_p c \in L^\infty(\Omega,\C^d)$. Clearly, the statement holds if $\xi_0=0$. Suppose that $\xi_0 \ne 0$. By \eqref{e : mu e delta} and by using dilations $\xi \leadsto \xi/t$, $t \rightarrow \pm \infty$, in \eqref{e : equiv B}, we see that $\mu_p < 1$.  Let $\varepsilon >0$.  By ellipticity of $A$ and by choosing $\xi = - \varepsilon \, \xi_0$ in \eqref{e : equiv B}, we obtain that
$$
\varepsilon \left[ 1 - \varepsilon ( \|\cJ_p \| \Lambda - \mu_p )\right] |\xi_0|^2 \leq (1-\mu_p) V.
$$
Observe that $\|\cJ_p \| \Lambda - \mu_p \geq 0$, since $\mu_p \leq (p/2)\Delta_p(A) \leq  \|\cJ_p\|\Lambda$, by \eqref{e : mu e delta} and \eqref{i : eq def delta}. Moreover, $\|\cJ_p \| \Lambda - \mu_p > 0$. Otherwise, by sending $\varepsilon$ to $\infty$, we would have that $\xi_0=0$. Therefore, by maximising the left-hand side with respect to $\varepsilon >0$, we get 
$$
|\xi_0|^2 \leq 4 (1-\mu_p) (\|\cJ_p\|\Lambda - \mu_p)V.
$$

If $s \ne p$, the assertion follows by the previous case, \eqref{i : V dom Rec} and the identity 
\begin{equation}
\label{e : Jp e Js}
b+\cJ_sc = b+\cJ_pc + (s-p)\,  \Re c.
\end{equation}
\item
We shall prove that $|\Re c|, \, |\Re b|, \, |\Im b|, \, |\Im c| \leqsim \sqrt{V}$ in this order. 
The first estimate follows by applying \eqref{i : V dom bc} for $s =p$, \eqref{eq: sect form above} and the identity
$$
p\,\Re c = \Re(b-c) - \Re\left(b+  \cJ_p c\right).
$$
Clearly, by \eqref{eq: sect form above} we have
\begin{eqnarray}
\label{e: reb -rec}|\Re b - \Re c| &\leqsim& \sqrt{V}, \\
\label{e : imb-imc} |\Im b - \Im c| &\leqsim& \sqrt{V}.
\end{eqnarray}
Combining \eqref{e: reb -rec} with the estimate of $|\Re c|$, we get that  $|\Re b| \leqsim \sqrt{V}$. On the other hand, by \eqref{i : V dom bc} and \eqref{e : Jp e Js} we have
\begin{equation}\label{e : imb +imc}
|\Im b +\Im c| = |\Im\left(b +\cJ_p c \right)| \leqsim \sqrt{V}.
\end{equation}
Combining \eqref{e : imb-imc} and \eqref{e : imb +imc}, we obtain that $|\Im b|, |\Im c| \leqsim \sqrt{V}$. All the constants which appear in the inequalities  (continuously) depend only on $p, s, \Lambda, M, \mu_p(\oA)$.
\item 
By \eqref{e : Jp e Js}, $\cJ_{p +\delta} \xi = \cJ_p\xi + \delta \, \Re \xi$, hence \eqref{i : V dom Rec} implies
\begin{eqnarray*}
\Gamma_{p+\delta}^\oA(\cdot,\xi) & = & \Gamma_p^\oA(\cdot,\xi) + \delta \, \Re \big[ \sk{A\xi}{\Re \xi} + \sk{\xi}{\Re c}\big] \\
& \geq& \mu_p (|\xi|^2 + V ) - |\delta| \big(\Lambda |\xi|^2 + |\xi||\Re c| \big)\\
& \geq & (\mu_p - |\delta| \Lambda) |\xi|^2 - |\delta|C_p |\xi| \sqrt{V} + \mu_p V \\
& \geq& \big(\mu_p - |\delta|\Lambda - |\delta| (C_p/2) \big) |\xi|^2 + \big[\mu_p - |\delta| (C_p/2) \big] V \\
& \geqsim& |\xi|^2 + V,
\end{eqnarray*}
if $|\delta|$ is small enough.
\item
From the properties \eqref{i : R-lin}, \eqref{i : sym real Jp} of $\cJ_p$, we infer
\begin{eqnarray*}
&&\hskip -90 pt \Re\sk{e^{i\theta}A\xi}{\cJ_p \xi} + \Re\sk{e^{i\theta}b+\cJ_p( e^{i\theta}c)}{\xi} +  (\cos\theta) V \\
&=& \cos\theta \, \Re\big[\sk{A\xi}{ \cJ_p\xi} + \sk{b+\cJ_pc}{\xi} +  V  \big] \\
&&- \, \sin\theta \, \Im\big[ \sk{A\xi}{ \cJ_p\xi} + \sk{b+(p-1)\cJ_q c}{\xi} \big].
\end{eqnarray*}
By \eqref{i : V dom Rec} we get $\left|b+(p-1)\cJ_q c\right| \leq C_q \sqrt{V}$. Therefore, we may proceed as
\begin{eqnarray*}
&\geq& (\mu_p \cos\theta - \mod{\sin\theta}\|\cJ_p\|\Lambda) |\xi|^2 - C_q \mod{\sin\theta} \sqrt{V}|\xi| + \mu_p (\cos\theta) V.
\end{eqnarray*}
\begin{alignat*}{2}
\hskip -177 pt \text{We conclude choosing $\theta$ small enough.}
\tag*{\qedhere}
\end{alignat*}
\end{enumerate}
\end{proof}

In \cite[Corollary 5.17 (3)]{CD-DivForm}, the authors proved that $\cA_p$ is invariant under taking adjoints. The following proposition shows that this invariance still holds for $\cS_p$. Moreover, in \cite[Corollary 5.16]{CD-DivForm} it was proved that $\{ \cA_p(\Omega) : p \in [2, \infty) \}$ is a decreasing chain of matrix classes. In Proposition \ref{p : neg asp Bp} we shall show that $\{ \cS_p(\Omega) : p \in [2, \infty) \}$ is not. However, an interpolation property holds for $\cS_p$ which implies that $\{ \cB_p(\Omega) : p \in [2, \infty) \}$ is a decreasing chain.
\begin{proposition}\label{p: more on B}
Let $\Omega, A, b, c, V$ satisfy the standard assumptions of Section \ref{s : stand ass}. Let $ 1 < r < p < \infty$, $q=p/(p-1)$. Then the following assertions hold.
\begin{enumerate}[{\rm (i)}]
\item
\label{i : invarianza per coniugazione}
$(A,b,c,V) \in \cS_p(\Omega)$ if and only if $(A^*, c, b, V) \in \cS_q(\Omega)$.
\item
\label{i : interpolation for B}
We have
$$
\cS_r(\Omega) \cap \cS_p(\Omega) \subset \cS_s(\Omega)
$$
 for all $s \in [r,p]$.

In particular, $\{ \cB_p(\Omega) : p \in [2, \infty) \}$ is a decreasing chain.
\end{enumerate}
\end{proposition}
\begin{proof}
\begin{enumerate}[{\rm (i)}]
\item
 By  \eqref{i : inv Jp} and \eqref{i : sym real Jp}, we obtain that for all $\xi \in \C^d$
\begin{equation}\label{e : coniug delta}
\Re\sk{A^*\xi}{\cJ_q \xi} =  \Re\sk{A\left( \cJ_q \xi \right)}{\cJ_p\left( \cJ_q \xi \right)},
\end{equation}
\begin{equation}\label{e : coniug first term}
\Re\sk{c+\cJ_q b}{\xi} = \Re\sk{\cJ_q\left(b+\cJ_pc\right)}{\xi} = \Re\sk{b+\cJ_pc}{\cJ_q\xi}.
\end{equation}
Since $|\cJ_s \xi| \sim_s |\xi|$ for all $\xi \in \C^d$ and all $s \in (1,\infty)$, we conclude combining \eqref{e : equiv B}, \eqref{e : coniug delta} and \eqref{e : coniug first term}.
\item
Let $t \in (0,1)$ such that $tp +(1-t)r = s$. By assumptions we have for all $\xi \in \C^d$
\begin{eqnarray*}
t \left( \Re\sk{A\xi}{\cJ_p\xi} + \Re\sk{b+\cJ_pc}{\xi} +  V  \right) &\geqsim& t \,( |\xi|^2+ V), \\
(1-t)\left( \Re\sk{A\xi}{ \cJ_r\xi} + \Re\sk{b+\cJ_rc}{\xi} +  V \right) &\geqsim& (1-t) \,  ( |\xi|^2+V).
\end{eqnarray*}
Using the identity
\begin{eqnarray*}
t \cJ_p +(1-t)\cJ_r & = &  \cJ_s
\end{eqnarray*}
\begin{alignat*}{2}
\hskip -152 pt \text{and summing the terms above we conclude.}
\tag*{\qedhere}
\end{alignat*}
\end{enumerate}
\end{proof}

We conclude this subsection underlining that $\{ \cS_p(\Omega) : p \in [2, \infty) \}$ is not a decreasing chain. Moreover, $\cS_p(\Omega)$ is not invariant under conjugation of $p$. This further justifies our introduction of another class ($\cB_p$).
\begin{proposition}\label{p : neg asp Bp}
For all  $V \in L_{\rm{loc}}^1(\Omega, \R_+) \setminus \{0\}$, $p,r \in (1, \infty)$, $p \ne r$, $B \in \cA_r(\Omega)$, there exist $b=b(p,r,B,V), \, c= c(p,r,B,V) \in L^\infty(\Omega,\C^d)$ such that
\begin{itemize}
\item $(A,b,c,V) \in \cS_p(\Omega)$ for all $A \in \cA_p(\Omega)$,
\item $(B,b,c,V) \not\in \cS_r(\Omega)$.
\end{itemize}
In particular, for all $V \in L_{\rm{loc}}^1(\Omega, \R_+) \setminus \{0\}$, $p,r \in (1, \infty)$, $p \ne r$,  $A \in \cA_p(\Omega) \cap \cA_r(\Omega)$, there exist $b, c \in L^\infty(\Omega, \C^d)$ such that $(A,b,c,V) \in \cS_p(\Omega) \setminus \cS_r(\Omega)$.
\end{proposition}
\begin{proof}
Fix $V \in L_{\rm{loc}}^1(\Omega, \R_+) \setminus \{0\}$, $p,r \in (1, \infty)$, with $p \ne r$ and $B \in \cA_r(\Omega)$. Let $\varrho >0$ and $v = v_\varrho \in L^\infty(\Omega, \R^d)\setminus \{0\}$ such that
$$
(r-p)^2 |v|^2 = \varrho V.
$$
We define 
\begin{eqnarray*}
b=b_\varrho= -(p-1)v- i \frac{p}{2} v \qquad \text{and} \qquad c=c_\varrho= v + i \frac{p}{2} v.
\end{eqnarray*}
Then 
$$
b+\cJ_p c = 0, \quad |b-c|^2 = \frac{2p^2}{(r-p)^2} \varrho V, \qquad b+\cJ_r c = (r-p) v.
$$
Therefore $(A,b_\varrho,c_\varrho,V) \in \cS_p(\Omega)$ for all $A \in \cA_p(\Omega)$ and $\varrho >0$. Moreover, setting $\sigma := v/|v|$, we have for all $t \geq 0$ and $x \in \Omega$ such that $V(x) \ne 0$
\begin{alignat*}{2}
\Gamma_r^{B,b,c,V}(x, -t \sigma) &=t^2\,  \Re\sk{B(x) \sigma}{\cJ_r \sigma} - t (r-p) |v(x)| + V(x) \\
 &=  t^2 \, \Re\sk{B(x) \sigma}{\cJ_r \sigma} - t\sqrt{\varrho} \sqrt{V(x)} + V(x) \underset{\varrho \rightarrow +\infty}{\longrightarrow} -\infty.
\tag*{\qedhere}
\end{alignat*}
\end{proof}

\section{The power functions}\label{s: conv of pow fun and Bell fun}
Recall the definition \eqref{eq: N Bellman} of the Bellman function $\cQ$.
Owing to the tensor structure of $\cQ$, the generalised convexity of $\cQ$ is related to that of its elementary building blocks: the power functions (see \cite{CD-DivForm}).

Let $r>0$. Define $F_r : \C \rightarrow \R_+$ by
$$
F_r(\zeta) = |\zeta|^r, \quad \zeta \in \C.
$$
Let $ \textbf{1}$ denote the constant function of value $1$ on $\C$, that is, $ \textbf{1} = F_0$. 
\smallskip

A rapid calculation \cite{CD-DivForm, CD-Mixed, CD-Potentials, CD-mult} shows that
\begin{eqnarray}
\label{e : GradFp}
(\nabla F_r)(\zeta) &=& r |\zeta|^{r-2} \cV_{1}(\zeta) \quad \forall \zeta \in \C\setminus\{0\},  \\ 
\label{e : HessFp}
(D^2F_r)(\zeta) &=& r |\zeta|^{r-2} \left(I_{\R^{2}} + (r-2) \frac{ \cV_{1}(\zeta)}{|\zeta|} \otimes \frac{ \cV_{1}(\zeta)}{|\zeta|}\right)\quad \forall \zeta \in \C \setminus\{0\}.
\end{eqnarray}
Compare \eqref{e : HessFp} with \cite[(5.5)]{CD-DivForm} and \cite[(2.5)]{CD-Potentials}.
\smallskip

In \cite{CD-DivForm}, the authors proved that $F_p$ is $A$-convex if and only if $\Delta_p(A) \geq 0$, and $\cQ$ is strictly $(A,B)$-convex provided that $A, B \in \cA_p(\Omega)$. 

 In the next subsections we shall prove that $F_p$ is $\oA$-convex if and only if $\Gamma_p^{\oA} \geq 0$, and $\cQ$ is strictly $(\oA,\oB)$-convex provided that $\oA \in (\cS_p \cap \cS_2)(\Omega)$ and $\oB \in \cS_{q}(\Omega)$, generalising the previous results.

\subsection{Generalised convexity of power functions}
Let $r >0$. Before enunciating the next lemma, extend $\cJ_r$, $\Gamma_r$ and $\cS_r$ to $r>0$ by the same rules \eqref{e : operator Ip}, \eqref{e : def Gammap} and \eqref{e : equiv B}. Recall the definitions \eqref{eq: def c-Hess 1dim} and \eqref{d : Hess sup gen 1dim}. From \eqref{e : GradFp} and \eqref{e : HessFp} it follows that for all $\zeta \in \C \setminus\{0\}$ and all $X \in \C^d$
\begin{equation}
\label{e : HessGenFp}
\mathbf{H}_{F_p}^{\oA}[\zeta; X] =  |\zeta|^{p-2} \left( H_{F_p}^{A}[\zeta/|\zeta|; X] + |\zeta|  H_{F_p}^{(b,c)}[\zeta/|\zeta|; X]+ |\zeta|^2 G_{F_p}^{V}[\zeta/|\zeta|] \right). 
\end{equation}
The next result is modelled after \cite[Lemma~5.6]{CD-DivForm}.
We define 
$$
\sigma=e^{-i\arg(\zeta)}X \in \C^d.
$$
By  \eqref{e : realis compl prod}, \eqref{e : GradFp}, \eqref{e : HessFp} and by adequately modifying the proof of \cite[Lemma~5.6]{CD-DivForm} we get the following result.
\begin{lemma}
\label{l : 11}
Let $r > 0$, $b,c \in\C^d$, $V \in \R$, $\zeta \in \C$ and $X \in \C^d$. Then, for $|\zeta|=1$,
$$
\aligned
r^{-1} H_{F_r}^{(b,c)}[\zeta; X] &=  \,\Re \sk{b+\cJ_rc}{\sigma},\\
r^{-1} G_{F_p}^{V}[\zeta]  &=   V |\zeta|^2.
\endaligned
$$
\end{lemma}
\begin{proposition}\label{p: gen conv power function} 
Let $\Omega$ and $\oA=(A, b, c, V)$ satisfy the standard assumptions of Section \ref{s : stand ass}. Let $r>0$, $\zeta \in \C\setminus\{0\}$ and $X \in \C^{d}$. Suppose that  for a.e. $x \in \Omega$ 
\begin{equation*}
\Gamma_r^{\oA}(x,\xi) \geq 0, \quad \forall \xi \in \C^d.
\end{equation*}
Then 
$$
\mathbf{H}_{F_r}^{\oA(x)}[\zeta; X]  \geq 0.
$$
Moreover, if  $\oA \in \cS_r(\Omega)$, then for a.e. $x \in \Omega$ we have
\begin{eqnarray*}\label{eq: gen conv power function}
\mathbf{H}_{F_r}^{\oA(x)}[\zeta; X]  \geq r \mu_r(A)  |\zeta|^{r-2}\bigl( |X|^2 + V(x) |\zeta|^2 \bigr).
\end{eqnarray*}
\end{proposition}
\begin{proof}
By \eqref{e : HessGenFp}, Lemma \ref{l : 11} and \cite[Lemma 5.6]{CD-DivForm} we get
\begin{eqnarray*} \label{e: hess power funct}
\aligned
\mathbf{H}_{F_r}^{\oA}[\zeta; X] &= r |\zeta|^{r-2} \left(\Re \sk{A\sigma}{\cJ_r\sigma}_{\C^d} +  |\zeta|\, \Re\sk{b+\cJ_rc}{\sigma} + |\zeta|^2 V \right) \\
& =  r |\zeta|^r \Gamma_r^{\oA}(x, \sigma / |\zeta|).
\endaligned
\end{eqnarray*}
Therefore, we conclude.
\end{proof}

\subsection{Generalised convexity of the Bellman function of Nazarov and Treil}
Now we shall prove the $(\oA,\oB)$-convexity of the Bellman function provided that $\oA \in (\cS_p\cap\cS_2)(\Omega)$ and $\oB \in \cS_q(\Omega)$. This lemma is the analogue of \cite[Lemma 5.11]{CD-DivForm}.
\begin{lemma}\label{p: 1-Hess tensor product}
Let $1<q<2$ and $b,c,\beta,\gamma \in \C^d$. Take $\omega=(\zeta,\eta) \in \C\times\C$ such that $|\zeta|<|\eta|^{q-1}$ and $X, Y \in \C^{d}$. Then
$$
\aligned
 H_{F_2 \otimes F_{2-q}}^{(b,\beta,c,\gamma)} [\omega; (X,Y)]  =& \, |\eta|^{2-q} H_{F_2}^{(b,c)}[\zeta, X] + |\zeta|^2 H_{F_{2-q}}^{(\beta,\gamma)}[\eta, Y] \nonumber \\
&+ \, 2(2-q)  |\zeta|^2 |\eta|^{1-q}\, \Re\sk{c}{\Re \sigma_2} \\
& + \, 2(2-q)  |\zeta| |\eta|^{2-q} \,\Re\sk{\gamma}{\Re \sigma_1}.
\endaligned
$$
\end{lemma}
\begin{proof}
This follows from combining the definition of $ H_{F_2 \otimes F_{2-q}}^{(b,\beta,c,\gamma)} [\omega; (X,Y)]$ (see \eqref{eq: def c-Hess}) and the identity
\begin{alignat*}{2}
\partial_{\zeta_j\eta_k}^2(F_2 \otimes F_{2-q})(\zeta,\eta) = 2(2-q) \zeta_j\eta_k|\eta|^{-q} \quad \text{for } j,k=1,2.
\tag*{\qedhere}
\end{alignat*}
\end{proof}

Lemma \ref{p: 1-Hess tensor product} and \cite[Corollary~5.10, Lemma~5.11]{CD-DivForm} immediately give the following estimate.
\begin{corollary}\label{c: gen Hess tensor prod}
Let $\Omega, \oA=(A,b,c,V), \oB=(B,\beta,\gamma,W)$ satisfy the standard assumptions of Section \ref{s : stand ass}. Let $1 < q <2$. Take $\omega= (\zeta,\eta) \in \C \times \C$ such that $|\zeta| < |\eta|^{q-1}$ and $X,Y \in \C^{d}$. Then, for almost everywhere $x \in \Omega$, we have
\begin{eqnarray*}
&&\hskip -25pt \mathbf{H}_{F_2 \otimes F_{2-q}}^{(\oA(x),\oB(x))}[\omega; (X,Y)] \\
&   \geq& |\eta|^{2-q} \mathbf{H}_{F_2}^{\oA(x)}[\zeta;X] + \frac{(2-q)^2}{2}\Delta_{2-q}(B)|\eta|^{q-2}|Y|^2\\
&&-\, (2-q) \left|\beta(x)+(\cJ_{2-q} \gamma)(x)\right| |\zeta|^2 |\eta|^{1-q} |Y| + (2-q)  W(x) |\zeta|^2 |\eta|^{2-q} \\
&&-\, 4(2-q)\Lambda|X||Y| -2(2-q)|\Re c(x)| |\zeta|^2 |\eta|^{1-q} |Y| -2(2-q) |\Re \gamma (x)| |\eta| |X|.
\end{eqnarray*}
\end{corollary}

\begin{theorem}
\label{t: Conv Belman}
Let $\Omega, \oA=(A,b,c,V), \oB=(B,\beta,\gamma,W)$ satisfy the standard assumptions of Section \ref{s : stand ass}.
Let $p\ge2$ such that $\oA \in (\cS_p \cap \cS_2)(\Omega)$ and $\oB \in \cS_q(\Omega)$.
Then there exists $\delta \in (0,1)$ and $C>0$ such that for $\mathcal{Q}= \mathcal{Q}_{p,\delta}$ we have, for almost everywhere $x \in \Omega $, 
\begin{equation*}
\hskip -10pt \mathbf{H}_{\mathcal{Q}}^{(\oA(x),\oB(x))}[\omega; (X,Y)] \geq C \sqrt{|X|^2+ V(x) |\zeta|^2} \sqrt{|Y|^2+ W(x) |\eta|^2}
\end{equation*}
for any $\omega=(\zeta,\eta) \in (\C\times \C) \setminus \Upsilon$ and $X,Y \in \C^{d}$. The constant $C>0$ can be chosen so as to (continuously) depend on $p, \lambda(A,B), \Lambda(A,B), M(\oA,\oB), \mu(\oA), \mu_p(\oA)$ and $\mu_q(\oB)$, but not on the dimension $d$.
\end{theorem}
\begin{proof}
When $p=2$ the Bellman function reads $\cQ(\zeta,\eta) = (1+\delta) |\zeta|^2 +|\eta|^2$ for all $\zeta,\eta \in \C$, hence the theorem quickly follows from Proposition \ref{p: gen conv power function}. Thus form now on assume that $p>2$.

If $|\zeta|^p >|\eta|^q >0$, then by the second assertion of Proposition \ref{p: gen conv power function} we have, for almost every $x \in \Omega$,
\begin{eqnarray*}
 \mathbf{H}_{\mathcal{Q}}^{(\oA(x),\oB(x))}[\omega; (X,Y)] & = & [1+ \delta (2/p) ] \mathbf{H}_{F_p}^{\oA(x)}[\zeta;X]+ [1+ \delta(2/q  -1 )]  \mathbf{H}_{F_q}^{\oB(x)}[\eta;Y] \\
& \geq & p[1+ \delta(2/p) ] \mu_p(A)  |\zeta|^{p-2} \bigl(|X|^2 + V(x)|\zeta|^2 \bigr)\\
&&+\, q[1+ \delta(2/q  -1 )] \mu_q(B)|\eta|^{q-2}\bigl(  |Y|^2 + W(x) |\eta|^2 \bigr).
\end{eqnarray*}
By assumptions we have $2-q>0$. So whenever $\delta >0$, we may continue as
\begin{eqnarray*}
& \geq & \min\{\mu_p(A),\mu_q(B)\} \left(p|\zeta|^{p-2} \bigl(|X|^2 + V(x) |\zeta|^2 \bigr) + q|\eta|^{q-2}\bigl(  |Y|^2 + W(x) |\eta|^2 \bigr) \right) \\
& \geq & \frac{p}{\sqrt{p-1}} \min\{\mu_p(A),\mu_q(B)\} \sqrt{|X|^2 + V(x) |\zeta|^2} \sqrt{|Y|^2 + W(x) |\eta|^2}.
\end{eqnarray*}
In the last step we have used the assumption $|\zeta|^p >|\eta|^q$ and the inequality between the arithmetic and geometric mean.

What remains is the case $|\zeta|^p <|\eta|^q$. Set $\Theta := \{p, \lambda, \Lambda, \mu, M, \mu_p(\oA), \mu_q(\oB)\}$. By Proposition \ref{p : property of B} \eqref{i : V dom bc}, \eqref{i : V dom Rec} there exist $C_0=C_0(\Theta)>0$ such that for a.e. $x \in \Omega$
$$
|\Re c(x) | \leq C_0 \sqrt{V(x)}, \quad |\Re \gamma(x)| \leq C_0 \sqrt{W(x)}, \quad \left|\beta(x) + (\cJ_{2-q} \gamma)(x)\right| \leq C_0 \sqrt{W(x)}.
$$
Therefore, from Proposition \ref{p: gen conv power function} (applied with $r=q$ and $r=2$) and Corollary \ref{c: gen Hess tensor prod} we get, for almost every $x \in \Omega$,
\begin{eqnarray*}
\label{e : 8}
&&\hskip -35pt \mathbf{H}_{\mathcal{Q}}^{(\oA(x),\oB(x))}[\omega; (X,Y)] \nonumber \\ 
& \geq & \mathbf{H}_{F_q}^{\oB(x)}[\eta;Y] + \delta \mathbf{H}_{F_2 \otimes F_{2-q}}^{(\oA(x),\oB(x))}[\omega; (X,Y)] \nonumber \\
& \geq &  2\delta |\zeta|^2 |\eta|^{-q} \biggl( \frac{q \mu_q}{6 \delta} |Y|^2   -(2-q) C_0 \sqrt{V(x)} |\eta| |Y| +  \mu_2 V(x) |\eta|^{2}   \biggr) \nonumber \\
&& +\,\, (2-q) \delta |\zeta|^2 |\eta|^{-q} \biggl(\frac{q \mu_q}{3(2-q) \delta} |Y|^2  - C_0 \sqrt{W(x)} |\eta| |Y| + W(x) |\eta|^{2} \biggr)\nonumber \\
&& +\,\, 2 \delta \biggl(\mu_2|\eta|^{2-q}|X|^2 -(2-q) C_0 \sqrt{W(x)} |\eta| |X| + \frac{q \mu_q}{2 \delta} W(x)|\eta|^{q}  \biggr)    \nonumber \\
&& +\,\,  2\delta \biggl( \Gamma |\eta|^{q-2}|Y|^2 - 2(2-q)\Lambda|X||Y|\biggr),
\end{eqnarray*}
where
$$
\Gamma = \frac{q \mu_q }{6\delta} + \frac{(2-q)^2}{4}\Delta_{2-q}(B).
$$
Since $\mu_2, \mu_q >0$, we have that $ (\mu_2 \mu_q)/ \delta$ grows to infinity as $\delta \searrow 0$. Therefore, there exists $\delta = \delta(\Theta) >0$ such that
\begin{equation*}
\label{eq : first choose of delta}
\frac{2q \mu_2 \mu_q}{3 \delta} > C_0^2,
\end{equation*}
which, through \cite[Corollary~3.4]{CD-Potentials} and the facts that $q \in (1,2)$ and $\mu_2 <1$ (see \eqref{e : mu e delta}), implies  the existence of $C_1=C_1(\delta,\Theta), C_2=C_2(\delta, \Theta) >0$ such that 
$$
\aligned
 \frac{q \mu_q}{6 \delta} |Y|^2   -(2-q) C_0 \sqrt{V(x)} |\eta| |Y| +  \mu_2 V(x) |\eta|^{2} &\geq C_1 V(x) |\eta|^2, \nonumber \\
\frac{q \mu_q}{3(2-q) \delta} |Y|^2  - C_0 \sqrt{W(x)} |\eta| |Y| + W(x) |\eta|^{2} &\geq 0, \\
\mu_2|\eta|^{2-q}|X|^2 -(2-q) C_0 \sqrt{W(x)} |\eta| |X| + \frac{q \mu_q}{2\delta} W(x)|\eta|^{q} & \geq  C_2(|\eta|^{2-q}|X|^2 + W(x)|\eta|^q), \nonumber
\endaligned
$$
where 
$$
C_2=C_2(\delta,\Theta)= \left(\sqrt{\frac{q \mu_2 \mu_q}{2 \delta}}- \frac{2-q}{2} C_0(\Theta) \right) \min\left\{ \sqrt{\frac{2 \mu_2 \delta}{q \mu_q}}, \sqrt{\frac{q \mu_q}{2 \mu_2 \delta}} \right\}.
$$
Therefore, for almost every $x \in \Omega$,
\begin{eqnarray*}
&&\hskip -45pt \mathbf{H}_{\mathcal{Q}}^{(\oA(x),\oB(x))}[\omega; (X,Y)] \nonumber \\ 
& \geq & 2\delta \left(C_2 |\eta|^{2-q} |X|^2 -2(2-q)\Lambda |X| |Y| + \Gamma |\eta|^{q-2}|Y|^2 \right) \\
&& +\,\,  2 \delta C_1 V(x) |\zeta|^2 |\eta|^{2-q} + 2\delta C_2 W(x) |\eta|^q.
\end{eqnarray*}
Since $C_2$ tends to $\mu_2 >0$ and $\Gamma$ grows to infinity as $\delta \searrow 0$, we can choose $\delta>0$ sufficiently small so that 
$$
C_2 \Gamma > [(2-q)\Lambda]^2.
$$
By \cite[(5.23)]{CD-DivForm}, we may choose $\delta$ (continuously) depending only on $\Theta$. Therefore, applying \cite[Corollary~3.4]{CD-Potentials} again, we obtain that, for almost every $x \in \Omega$,
\begin{eqnarray*}
&&\hskip -45pt \mathbf{H}_{\mathcal{Q}}^{(\oA(x),\oB(x))}[\omega; (X,Y)] \nonumber \\ 
& \geqsim_\Theta &  |\eta|^{2-q}\left(|X|^2 + V(x)|\zeta|^2 \right)+ |\eta|^{q-2}\left(|Y|^2 +  W(x)|\eta|^2\right) \\
& \geqsim_\Theta &  \sqrt{|X|^2 +  V(x) |\zeta|^2} \sqrt{|Y|^2 +W(x) |\eta|^2},
\end{eqnarray*}
where in the last step we have used the inequality between the arithmetic and geometric mean.
\end{proof}
\begin{remark}
When $b=c=0$, in \cite[Theorem 3.1]{CD-Potentials} Carbonaro and Dragi\v{c}evi\'c estimated $H_{\cQ}^{(A,B)}$ and $G_{\cQ}^{(V,W)}$ separately. However, this time, in order to use the nonnegativity of the function $\Gamma_p$, we cannot separate $H_{\cQ}^{(A,B)}, H_{\cQ}^{(b,c,\beta,\gamma)}$ and $G_{\cQ}^{(V,W)}$ into individual pieces and estimate them one by one.
\end{remark}

We would like to have an analogue of \cite[Corollary 5.5]{CD-DivForm}. Fix a radial function $\varphi \in C_c^\infty(\R^4)$ such that $0\leq\f\leq1$, ${\rm supp}\, \f\subset B_{\R^{4}}(0,1)$ and $\int\f=1$.  For $\nu\in(0,1]$ define $\varphi_\nu(\omega)=\nu^{-4}\varphi(\omega/\nu)$.  Recall the notation \eqref{e : realis of compl fun}. If $\Phi : \C^2 \rightarrow \R$, define 
\begin{equation}
\label{e : compl convol}
\Phi \star \varphi_\nu = (\Phi_\cW \star \varphi_\nu) \circ \cW_{2,1} : \C^2 \rightarrow \R.
\end{equation}
\begin{lemma}
\label{l : hessiana reg Bellman}
Let $A, B \in \C^{d \times d}$, $b,\beta,c,\gamma \in \C^d$ and $V, W \in \R_+$. Set $\oA=(A,b,c,V)$ and $\oB=(B,\beta,\gamma,W)$. Then, for all $\omega=(\zeta,\eta) \in \C^2$, $X,Y \in \C^d$, $\nu \in (0,1)$, we have
\begin{equation}
\label{eq : Hess conv is conv Hess plus rems}
\aligned
\mathbf{H}^{(\oA,\oB)}_{ \cQ\star\varphi_{\nu}}[\omega;(X,Y)]&=\int_{\R^4}\mathbf{H}^{(\oA,\oB)}_{ \cQ}[\omega-\cW_{2,1}^{-1}(\omega^{\prime});(X,Y)]\varphi_\nu(\omega^{\prime})\,\wrt\omega^{\prime} \\ 
& \,\,\,\,\, + \,\,R^{(c,\gamma)}_\nu[\omega;(X,Y)] +R^{(V,W)}_\nu(\omega),
\endaligned
\end{equation}
where
\begin{equation}
\label{eq: def R1, R2}
\aligned
 R^{(c,\gamma)}_\nu & =  R^{(c,\gamma)}_\nu[\omega;(X,Y)] \hskip -10pt \\
 &= \int_{\R^4}\sk{\left[D^{2}\cQ(\omega -\cW_{2,1}^{-1}(\omega^{\prime}))\otimes I_{\R^{d}}\right] \cW_{2,d}(X,Y)}{\cW_{2,d}\left(\cV_1^{-1}(\zeta^\prime)c,\cV_1^{-1}(\eta^\prime) \gamma\right)} \times \\  
&\hskip 25pt \times  \varphi_\nu(\omega^{\prime})\,\wrt\omega^{\prime}, \\
R^{(V,W)}_\nu &= R^{(V,W)}_\nu(\omega) = \int_{\R^4}\sk{\nabla \cQ (\omega-\cW_{2,1}^{-1}(\omega^{\prime}))}{(V\zeta^\prime, W\eta^\prime)}\varphi_\nu(\omega^{\prime})\,\wrt\omega^{\prime}.
\endaligned
\end{equation}
\end{lemma}
\begin{remark}
\label{r : no inv und conv}
We would like to use \eqref{eq : Hess conv is conv Hess plus rems} in order to estimate $\mathbf{H}_{\cQ\star\varphi_\nu}^{(\oA,\oB)}$ from below. The integral $\int_{\R^4}\mathbf{H}_{\cQ}$ in \eqref{eq : Hess conv is conv Hess plus rems} will be estimated by means of Theorem~\ref{t: Conv Belman}.
Note that in the unperturbed case,  Carbonaro and Dragi\v{c}evi\'c \cite{CD-DivForm, CD-Mixed} proved that
$$
H_{\cQ\star\varphi_\nu}^{(A,B)}[\omega; (X,Y)] = \int_{\R^4}H^{(A,B)}_{ \cQ}[\omega-\cW_{2,1}^{-1}(\omega^{\prime});(X,Y)]\varphi_\nu(\omega^{\prime})\,\wrt\omega^{\prime}.
$$ 
This time, the presence of the terms $R_\nu^{(c,\gamma)}$ and $R_\nu^{(V,W)}$ in \eqref{eq : Hess conv is conv Hess plus rems} is due to the fact that in the definition \eqref{eq: def c-Hess}, the element $\omega=(\zeta,\eta)$ appears not only as a variable of the function $\cQ$, but also in the components of $(\zeta c,\eta \gamma)$ and $(V\zeta, W\eta)$. For this reason we cannot proceed exactly as in \cite{CD-Mixed}.
  The terms $R_\nu^{(c,\gamma)}$ and $R_\nu^{(V,W)}$ are remainders, in the sense that  they are going to disappear as  $\nu \rightarrow 0$. In order to prove that, we will need their upper  estimates which will be established next; see Lemma~\ref{l : lemma on reminds}.
\end{remark}
\begin{proof}[Proof of Lemma~\ref{l : hessiana reg Bellman}]
By definition \eqref{eq: def c-Hess} we have
$$
\aligned
H^{(b,\beta,c,\gamma)}_{\cQ}[\omega;(X,Y)]  =&  \sk{\left[D^{2}\cQ(\omega)\otimes I_{\R^{d}}\right] \cW_{2,d}(X,Y)}{\cW_{2,d}(\zeta c, \eta \gamma)}_{\R^{4d}} \\
&+ \sk{\nabla\cQ(\omega)}{\cW_{2,1}(\sk{X}{b}, \sk{Y}{\beta})}_{\R^{4}} \\
=:& \, G^{(b,\beta)}_{ \cQ}[\omega; (X,Y)] + H^{(c,\gamma)}_{ \cQ}[\omega; (X,Y)],
\endaligned
$$
for all $\omega=(\zeta,\eta) \in \C^2 \setminus \Upsilon$ and all $X,Y \in \C^d$.
Since $ \cQ_\cW\in C^1(\R^4)$ and its second-order partial derivatives exist on $\R^4\setminus\cW(\Upsilon)$ and are locally integrable in $\R^4$, by the ACL characterisation of Sobolev spaces (see, for example, \cite[Th\'eor\`eme~V, p. 57]{Schwartz} or \cite[Theorem~11.45]{Leoni}) we have
\begin{equation}
\label{e : hess Q reg}
\aligned
\mathbf{H}^{(\oA,\oB)}_{ \cQ\star\varphi_{\nu}}[\omega;(X,Y)]&=\int_{\R^4}H^{(A,B)}_{ \cQ}[\omega-\cW_{2,1}^{-1}(\omega^{\prime});(X,Y)]\varphi_\nu(\omega^{\prime})\,\wrt\omega^{\prime} \\ 
&\,\,\,\,\, + \, \,\int_{\R^4}G^{(b,\beta)}_{ \cQ}[\omega-\cW_{2,1}^{-1}(\omega^{\prime});(X,Y)]\varphi_\nu(\omega^{\prime})\,\wrt\omega^{\prime} \\ 
&\,\,\,\,\, + \,\, \sk{\left[D^{2}(\cQ \star \varphi)(\omega)\otimes I_{\R^{d}}\right] \cW_{2,d}(X,Y)}{\cW_{2,d}(\zeta c,\eta \gamma)} \\
&\,\,\,\,\,+ \,\, \sk{\nabla (\cQ \star \varphi)(\omega)}{\cW_{2,1}(V\zeta, W\eta)}, 
\endaligned
\end{equation}
for all $\omega \in \C^2$ and all $X,Y \in \C^d$.
Writing the third and the fourth terms of the right-hand side of \eqref{e : hess Q reg} as
\begin{eqnarray*}
&&\hskip -50pt \sk{\left[D^{2}(\cQ \star \varphi)(\omega)\otimes I_{\R^{d}}\right] \cW_{2,d}(X,Y)}{\cW_{2,d}(\zeta c,\eta \gamma)}\\
 &=&  \int_{\R^4}H^{(c,\gamma)}_{ \cQ}[\omega-\cW_{2,1}^{-1}(\omega^{\prime});(X,Y)]\varphi_\nu(\omega^{\prime})\,\wrt\omega^{\prime} + R_\nu^{(c,\gamma)}[\omega;(X,Y)], \\
&&\hskip -50pt \sk{\nabla (\cQ \star \varphi)(\omega)}{\cW_{2,1}(V\zeta, W\eta)}\\
 &=&  \int_{\R^4}G^{(V,W)}_{ \cQ}[\omega-\cW_{2,1}^{-1}(\omega^{\prime})]\varphi_\nu(\omega^{\prime})\,\wrt\omega^{\prime} + R_\nu^{(V,W)}(\omega),
\end{eqnarray*}
we get \eqref{eq : Hess conv is conv Hess plus rems}.
\end{proof}

\begin{lemma}
\label{l : lemma on reminds}
Let $\nu \in (0,1)$, $c, \gamma \in \C^d$ and $V,W \in \R_+$. Then
\begin{enumerate}[{\rm (i)}]
\item
\label{i : reminds dominated}
for all $\omega=(\zeta,\eta) \in \C^2$ and all $X,Y \in \C^d$,
\begin{equation*}
\aligned
|R^{(c,\gamma)}_\nu[\omega; (X,Y)]| &\leqsim  \nu ^{q-1}
 \max\{|c|,|\gamma|\} \left(1+|\zeta|^{p-2}+|\eta|^{2-q} \right) |(X,Y)|,\\ 
|R^{(V,W)}_\nu(\omega)| &\leqsim  \nu
 \max\{V,W\} \left(1+|\zeta|^{p-1}+|\eta|^{q-1} + |\eta| \right);  
\endaligned
\end{equation*}
\item
\label{i: reminds tends to 0}
$R^{(c,\gamma)}_\nu[\omega;(X,Y)]$ and $R^{(V,W)}_\nu(\omega)$ converge to $0$ as $\nu \rightarrow 0_+$ for all $\omega \in\C^2$, $X,Y \in \C^d$.
\end{enumerate}
\end{lemma}
\begin{proof}
By using \cite[Lemma~14(iii)]{CD-Mixed}, the second and the third estimates of \cite[(29)]{CD-Mixed} and the fact that the support of both the integrand of $R^{(c,\gamma)}_\nu$ and $R^{(V,W)}_\nu$ is contained in  $B_{\R^4}(0,\nu)$, we get item \eqref{i : reminds dominated}.
Sending $\nu$ to $0$, we prove item \eqref{i: reminds tends to 0}. 
\end{proof}
\begin{corollary}\label{p: N gen conv reg} 
Let $\Omega, \oA=(A,b,c,V), \oB=(B,\beta,\gamma,W)$ satisfy the standard assumptions of Section \ref{s : stand ass}.
Let $p\ge2$ such that $\oA \in (\cS_p \cap \cS_2)(\Omega)$ and $\oB \in \cS_q(\Omega)$. Then for every $\omega =(\zeta,\eta) \in\C\times\C$, $X,Y\in\C^{d}$ and almost every $x\in\Omega$,
$$
\liminf_{\nu \rightarrow 0} \mathbf{H}_{\cQ\star\varphi_{\nu}}^{(\oA(x),\oB(x))}[\omega;(X,Y)]
\geq C \sqrt{|X|^2+V(x)|\zeta|^2}\sqrt{|Y|^2+ W(x)|\eta|^2}.
$$
The constant $C>0$ can be chosen so as to (continuously) depend on $p, \lambda(A,B)$, $\Lambda(A,B), M(\oA,\oB), \mu(\oA), \mu_p(\oA)$ and $\mu_q(\oB)$, but not on the dimension $d$.
\end{corollary}
\begin{proof}
By combining Theorem~\ref{t: Conv Belman} and Lemma~\ref{l : hessiana reg Bellman}, we infer that for almost every $x \in \Omega$,
\begin{equation}
\aligned
\label{eq: est Hess Gen Q*phi}
& \hskip -20pt \mathbf{H}^{(\oA(x),\oB(x))}_{ \cQ\star\varphi_{\nu}}[\omega;(X,Y)] \\
&\geqsim \int_{\R^4} \sqrt{|X|^2+V(x)|\zeta-\cV_1^{-1}(\zeta^\prime)|^2}\sqrt{|Y|^2+ W(x)|\eta - \cV_1^{-1}(\eta^\prime)|^2}\varphi_\nu(\omega^{\prime})\,\wrt\omega^{\prime} \\ 
& \,\,\,\,\, + \,\,R^{(c(x),\gamma(x))}_\nu[\omega;(X,Y)] +R^{(V(x),W(x))}_\nu(\omega), 
\endaligned
\end{equation}
for all $\omega=(\zeta,\eta) \in \C^2$ and $X,Y \in \C^d$.

Since the function $\C^2 \ni (\zeta,\eta) \mapsto  \sqrt{|X|^2+V(x)|\zeta|^2}\sqrt{|Y|^2+ W(x)|\eta|^2}$ is continuous in $\C^2$ for all $X,Y \in \C^d$ and for all $x\in \Omega$, we get
\begin{eqnarray}
\label{eq: F*phi tends to F}
&&  \hskip -50pt \lim_{\nu \rightarrow 0}  \int_{\R^4} \sqrt{|X|^2+V(x)|\zeta-\cV_1^{-1}(\zeta^\prime)|^2}\sqrt{|Y|^2+ W(x)|\eta - \cV_1^{-1}(\eta^\prime)|^2}\varphi_\nu(\omega^{\prime})\,\wrt\omega^{\prime} \nonumber \\
&=&  \sqrt{|X|^2+V(x)|\zeta|^2}\sqrt{|Y|^2+ W(x)|\eta|^2}
\end{eqnarray}
for all $\zeta, \eta \in \C$, $X,Y \in \C^d$ and $x \in \Omega$.
Combining Lemma \ref{l : lemma on reminds}\eqref{i: reminds tends to 0}, \eqref{eq: est Hess Gen Q*phi} and \eqref{eq: F*phi tends to F}, we conclude.
\end{proof}

\section{$L^p$ contractivity and analyticity of  $(T^{A,b,c,V,\oV}_{t})_{t>0}$} \label{s : Lp contr}
Let $\oA=(A,b,c,V) \in \cB(\Omega)$. Let prove now Theorem \ref{t : contract} and Corollary \ref{c: N analytic sem}.
\subsection{Proof of Theorem \ref{t : contract}}\label{ss : Lp contr}
Let $(\Omega, \mu)$ be a measure space, $\gotb$ a sesquilinear form defined on the domain $\Dom(\gotb) \subset L^2=L^2(\Omega)$ and $1<p<\infty$. Denote 
$$
\Dom_p(\gotb) := \{ u \in \Dom(\gotb) : |u|^{p-2} u \in \Dom(\gotb) \}.
$$
We say that $\gotb$ is {\it $L^p$-dissipative} if 
$$
\Re \gotb(u, |u|^{p-2} u) \geq 0 \quad \forall u \in \Dom_p(\gotb).
$$
The notion of $L^p$-dissipativity of sesquilinear forms  was introduced by Cialdea and Maz'ya in \cite{CiaMaz} for forms defined on $C_c^1(\Omega)$. Then it was extended by Carbonaro and Dragi\v{c}evi\'c in \cite[Definition 7.1]{CD-DivForm}. 

In order to prove the $L^p$-contractivity of $(T_t^{A,b,c,V,\oV})_{t>0}$, we follow the proof of the implication $(a) \Rightarrow (b)$ in \cite[Theorem~1.3]{CD-DivForm} for which the following theorem due to Nittka \cite[Theorem~4.1]{Nittka} is essential. We reproduce it in the form it appeared in \cite[Theorem~ 2.2]{CD-Potentials}.
\begin{theorem}[Nittka]\label{t: Nittka}
Let $(\Omega, \mu)$ be a measure space. Suppose that the sesquilinear form $\gota$ on $L^2=L^2(\Omega,\mu)$ is densely defined, accretive, continuous and closed. Let $\oL$ be the operator associated with $\gota$.

Take $p \in (1,\infty)$ and define $B^p := \{u \in L^2 \cap L^p : \|u\|_p \leq 1 \}$. Let $\bP_{B^p}$ be the orthogonal projection $L^2 \rightarrow B^p$. Then the following assertions are equivalent:
\begin{itemize}
\item $\| {\rm exp}(-t\oL) f\|_p \leq \|f\|_p$ for all $f \in L^2 \cap L^p$ and all $t \geq 0$;
\item $\Dom(\gota)$ is invariant under $\bP_{B^p}$ and $\gota$ is $L^p$-dissipative.
\end{itemize}
\end{theorem}
The $L^p$-dissipativity of the form \eqref{e : form sesq} is closely related to our $ \mathbf{H}_{F_p}^{\oA}$, as we show next.
\begin{proposition}\label{p : LpDiss and HessFp}
Suppose that $B \in \C^{d\times d}$, $\beta, \gamma \in \C^d$ and $W \in \R_+$, $p >1 $ and $f \in \Dom_p(\gota)$. Then, setting $\oB = (B,\beta,\gamma,W)$,
$$
  \mathbf{H}_{F_p}^\oB[f; \nabla f] = p\, \Re\left( \sk{B\nabla f}{\nabla\left(\mod{f}^{p-2}f\right)} + {\sk{\nabla f}{\beta}}\mod{f}^{p-2}\overline{f}+ f\sk{\gamma}{\nabla\left(\mod{f}^{p-2}f\right)}\right)  + W|f|^p.
$$
\end{proposition}
\begin{proof}
By \cite[(5.5)]{CD-DivForm}, see also \cite[Lemma 2.3]{CD-Potentials},
$$
\nabla\left(\mod{f}^{p-2}f\right) = \mod{f}^{p-2}\textrm{sign} f \cdot \cJ_p\left( \textrm{sign} \overline{f} \cdot \nabla f \right)
$$
for all $f \in H^1(\Omega)$ such that $|f|^{p-2}f \in H^1(\Omega)$.
In order to finish the proof it now suffices to recall \cite[Lemma 5.6]{CD-DivForm} and Lemma \ref{l : 11} applied with $k=1$.
\end{proof}

\begin{proof}[Proof of Theorem \ref{t : contract}]
We will use Nittka's invariance criterion (Theorem \ref{t: Nittka}). Under our assumptions on $\phi$, the sesquilinear form $\gotb := e^{i\phi}\gota$ is densely defined, closed and sectorial. It is well-known that a sectorial form is accretive and continuous; see for example \cite[Proposition 1.8]{O}. Therefore, it falls into the framework of Nittka's criterion. The operator associated with $\gotb$ is $e^{i\phi} \oL^{A,b,c,V}$. 

The invariance of $\Dom(\gotb)=\Dom(\gota)$ under $\bP_{B^p}$ was proved in \cite[Theorem 1.2]{CD-Potentials}. We have to prove just the $L^p$-dissipativity of $\gotb$.

Let $u \in \Dom_p(\gotb)$. By Proposition \ref{p : LpDiss and HessFp}, applied with $\oB = \oA_\phi := (e^{i\phi}A,e^{i\phi}b,e^{i\phi}c,(\cos\phi)V)$, we get 
$$
\Re \gotb(u, |u|^{p-2} u) = p^{-1} \int_{\Omega}  \mathbf{H}_{F_p}^{\oA_\phi}[u; \nabla u].
$$
Hence the theorem follows from Proposition \ref{p: gen conv power function}.
\end{proof}

Without assumptions on the lower order terms, we obtain the $L^p$-quasi-contractivity of the semigroup $(T_t)_{t>0}$ generated by $-\oL$ (i.e., the $L^p$-contractivity of the semigroup generated by $\oL + \omega$, for some $\omega \geq0$) provided that $A$ is $p$-elliptic. Compare the next result with \cite[Theorem 6]{CiaMaz}; see also \cite[Proposition 5.18]{CD-DivForm}. 

\begin{corollary}
\label{c : contract}
Let $\Omega, A, b, c, \oV$ satisfy the standard assumptions of Section \ref{s : stand ass}. Let $V \in L_{\rm{loc}}^1(\Omega,\R)$ be bounded from below and $p>1$ such that $A \in \cA_p(\Omega)$.
Then $(T^{A,b,c,V,\oV}_{t})_{t>0}$ extends to a quasi-contraction semigroup on $L^p(\Omega)$.
\end{corollary}
\begin{proof}
We want to find $\omega > |{\rm ess} \inf V|$ such that $(T_t^{A,b,c,V+\omega,\oV})_{t>0}$ extends to a strongly continuous semigroup of contractions on $L^p(\Omega)$. To this end, in light of Theorem \ref{t : contract} applied with $\phi =0$, it is sufficient to find $\omega, \mu >0$ such that
$$
\aligned
\Gamma_2^{A,b,c,V+\omega}(\xi) &\geq  \mu\left( |\xi|^2 + V + \omega\right), \\
\Gamma_p^{A,b,c,V+\omega}(\xi) &\geq  0,
\endaligned
$$
for all $\xi \in \C^d$. The first inequality would follow if we had $\mu < \min\{\lambda, 1\}$, where $\lambda =\lambda(A)$ is such that
$$
(\lambda - \mu) |\xi|^2 + {\rm ess} \inf \underset{|\sigma|=1}{\min} \, \Re\sk{b+c}{\sigma} |\xi| + (1-\mu) ({\rm ess} \inf V + \omega) \geq 0.
$$
This holds for $\omega$ large enough. In the same way, we are able to find $\omega$ such that the second inequality holds.
\end{proof}

\begin{proof}[Proof of Corollary \ref{c: N analytic sem}]
By Proposition \ref{p: more on B}\eqref{i : interpolation for B} and Proposition \ref{p : property of B}\eqref{i : cont B in phi}, there exists $\theta>0$ such that $(e^{i\phi}A,\,e^{i\phi}b,\,e^{i\phi}c,\,(\cos\phi)V) \in \cS_s(\Omega)$ for all $\phi \in [-\theta,\theta \,]$ and all $s \in [r,p]$. The contractivity part now follows from Theorem \ref{t : contract} 
 and the relation
\begin{equation}\label{e: compl times sem}
T_{t e^{i\phi}} = \textrm{exp}\left( -te^{i\phi} \oL \right),
\end{equation}
whereupon analyticity is a consequence of a standard argument \cite[Chapter~II, Theorem~4.6]{EN}.
\end{proof}

\subsection{Comparison between \eqref{e : weak Gamma cond} and the Cialdea-Maz'ya condition \cite[(2.25)]{CiaMaz}}\label{ss : compar CM}
Let $(C_c(\Omega, \C))^*$ denote the continuous dual of $C_c(\Omega, \C)$ endowed with the uniform norm. In \cite{CiaMaz} Cialdea and Maz'ya studied the $L^p$-dissipativity of the sesquilinear form $\widetilde\gota$, defined by
\begin{eqnarray}
\label{e : def sf CM}
\aligned
\Dom(\widetilde\gota)&= C_c^1(\Omega),  
\\
\widetilde\gota(u,v)&= \displaystyle \int_{\Omega}\sk{A\nabla u}{\nabla v} + \sk{\nabla u}{b}\overline{v} + u\sk{c}{\nabla v} + V u \overline{v}. \endaligned
\end{eqnarray}
Here $A$ is a $d \times d$ matrix with entries $a_{hk} \in(C_c(\Omega, \C))^*$, $b=(b_1, \dots, b_d)$ and $c=(c_1,\dots,c_d)$ stand for vectors with $b_j, c_j \in  (C_c(\Omega, \C))^*$ and $V$ is a complex-valued scalar distribution in $(C_c^1(\Omega))^*$.  By Riesz-Markov theorem every element of $(C_c(\Omega, \C))^*$ may be interpreted as a regular complex Borel measure on $\Omega$. Therefore, for instance, the meaning of the first and the last terms in \eqref{e : def sf CM} is
$$
\aligned
\displaystyle \int_{\Omega}\sk{A\nabla u}{\nabla v} &= \sum_{k,h = 1}^d \displaystyle \int_{\Omega} \partial_k u\,  \partial_h  \overline{v} \, \wrt a_{hk}, \\
\displaystyle \int_{\Omega}  V u \overline{v}&= \sk{V}{u \overline{v}}.
\endaligned
$$
They introduced the notion of $L^p$-dissipativity and, under the additional assumptions that $\Omega \subset \R^d$ is a bounded domain with sufficiently regular boundary \cite[p. 1087]{CiaMaz}, the entries of $A$ and $b$ belong to $C^1(\overline{\Omega})$, $c$ is zero and $V \in C(\overline{\Omega})$, they proved in \cite[Theorem~6]{CiaMaz} that the semigroup $T_t^{\oA,H_0^1}$ is $L^p$-quasicontractive if and only if 
$$
|p-2| | \sk{\Im A(x) \xi}{\xi}| \leq 2 \sqrt{p-1} \sk{\Re A(x) \xi}{\xi}|,
$$
for any $x\in\Omega, \xi \in \R^d$. The above inequality is equivalent is to saying that $\Delta_p(A_s)\geq0$, where $A_s = (A+ A^T)/2$ is the symmetric part of $A$; see \cite[Proposition~5.18]{CD-DivForm}.

Without these further assumptions on $\Omega$ and on the coefficients $A,b,c$ and $V$, in \cite[Corollary 4]{CiaMaz} they showed that $\widetilde\gota$ is $L^p$-dissipative provided that there exist two real constants $\theta$ and $\nu$ such that
\begin{eqnarray}
\label{e : CM suff cond}
&& \hskip -20pt\frac{4 }{pq} \sk{\Re A \alpha}{\alpha} + \sk{\Re A \beta}{\beta} + 2 \sk{\left(p^{-1} \Im A + q^{-1} \Im A^*\right)\alpha}{\beta} \nonumber  \\
&& +\,\,  \sk{\Im(b+c)}{\beta} + 2 \sk{\Re\left( \theta \frac{b}{p} + \nu \frac{c}{q} \right)}{\alpha} + \\
&& +\,\,  \Re\left[V-{\rm div}\left((1-\theta)\frac{b}{p} + (1-\nu)\frac{c}{q}\right)\right] \geq 0 \quad \forall \alpha,\beta \in \R^d, \nonumber
\end{eqnarray}
where $q$ is the conjugate exponent of $p$. Condition \eqref{e : CM suff cond} has to be understood in the sense of measures and distributions, that is, it means that, for every $\alpha, \beta \in \R^d$,
\begin{eqnarray*}
&& \hskip -20pt \displaystyle\int_\Omega \biggl[\biggl( \frac{4 }{pq}\sk{\Re A \alpha}{\alpha} + \sk{\Re A \beta}{\beta} + 2 \sk{\left(p^{-1} \Im A + q^{-1} \Im A^*\right)\alpha}{\beta} \nonumber  \\
&& +\,\,  \sk{\Im(b+c)}{\beta} + 2 \sk{\Re\left( \theta \frac{b}{p} + \nu \frac{c}{q} \right)}{\alpha} +\Re V \biggr) \varphi \biggr] \\
&&  - \, \,\sk{{\rm div}\left((1-\theta)\frac{\Re b}{p} + (1-\nu)\frac{\Re c}{q}\right)}{\varphi} \geq 0,
\end{eqnarray*}
for any nonnegative $\varphi \in C_c^1(\Omega)$, where the last term of the left-hand side of the previous inequality stands for the action of the distribution ${\rm div}\left((1-\theta)(\Re b/p) + (1-\nu)(\Re c/q)\right)$ on $\varphi$.

The distributional divergence appears after two integrations by parts. In fact, let $p \geq 2$ and $u$ be in $C_c^1(\Omega)$. Setting $v = |u|^{(p-2)/2}u$, which belongs to $C_c^1(\Omega)$, and noticing that $2 \,\Re(\overline{v}\nabla v) = \nabla (|v|^2)$, we have
$$
\aligned
\int_\Omega \Re \left(\sk{\nabla u}{b} |u|^{p-2} \overline{u}\right) =& \displaystyle \int_\Omega \frac{2}{p}\sk{\Re b}{\Re(\overline{v}\nabla v)} + \sk{\Im b}{\Im(\overline{v}\nabla v)}\\
=&  \displaystyle \int_\Omega \left( \frac{2 \theta}{p} \sk{\Re b}{\Re(\overline{v}\nabla v)}+ \sk{\Im b}{\Im(\overline{v}\nabla v)} \right) \\
& - \, \, \frac{1-\theta}{p} \sk{{\rm div}( \Re b)}{|v|^2}.
\endaligned
$$
Similarly, we get the divergence of $(1-\nu) (\Re c /q)$; see \cite[p. 1071 and p. 1078]{CiaMaz}.

Our assumptions on the domain $\Dom(\gota)$ of the form $\gota$, defined in \eqref{e : form sesq}, are weaker than those on the domain of $\widetilde\gota$, because the latter, which consists of all compactly supported functions on $\Omega$ of class $C^1$, is a proper subspace of  $\Dom(\gota)$. Moreover, since in generality that we consider the domain $\Omega$ may be completely irregular, $\Dom(\widetilde\gota)$ might not even be dense in $\Dom(\gota)$ with respect to the Sobolev norm $\| \cdot\|_{H^1(\Omega)}$, for instance if $\Dom(\gota) = H^1(\Omega)$.  In particular, for an arbitrary $u \in \Dom_p(\gota)$, we are not able to justify the distributional integration by parts
$$
\int_\Omega \sk{\Re b}{\nabla (|v|^2)} = -\sk{{\rm div}(\Re b)}{|v|^2},
$$
where $v= |u|^{(p-2)/2}u$ and $b \in L^\infty(\Omega,\C^d)$. Due to the weaker assumptions, our sufficient condition for the $L^p$-contractivity of the semigroup naturally turns out to imply \eqref{e : CM suff cond}.
\begin{proposition}
\label{l : charach Cialdea}
Let $A : \Omega  \rightarrow \C^{d,d}$, $b, \, c : \Omega \rightarrow \C^d$, $V : \Omega \rightarrow \R_+$ and $p>1$. Suppose that $\Gamma_p^{A,b,c,V}(x,\xi) \geq 0$ for almost all $x \in \Omega$ and all $\xi \in \C^d$. Then \eqref{e : CM suff cond} holds with $\theta=\nu=1$.
\end{proposition}
\begin{proof}
Recalling \eqref{eq: real form}, replacing $\alpha$ with $\alpha p /2$ and using the identity 
$$
b+\cJ_pc = \Re(b+(p-1)c) +i\Im(b+c),
$$
the condition \eqref{e : CM suff cond} with $\theta=\nu=1$ can be expressed as: for a.e. $x \in \Omega$
\begin{alignat*}{2}
 \Re \sk{A(x) (\alpha+i\beta)}{\cJ_p(\alpha +i\beta)} + \Re\sk{b+\cJ_pc}{\alpha +i\beta} + V(x) \geq 0 \quad \forall \alpha,\beta \in \R^d.
\tag*{\qedhere}
\end{alignat*}
\end{proof}

\section{Proof of the bilinear embedding}\label{s : proof bil emb}
We modify the heat-flow method of \cite{CD-Mixed} and \cite{CD-Potentials}. Observe that we cannot estimate $\mathbf{H}^{\oA,\oB}$ splitting it into the three terms of the definition \eqref{d : Hess sup gen}, as Carbonaro and Dragi\v{c}evi\'c implicitly did in \cite{CD-Potentials} when the first-order terms were zero. This is because our condition \eqref{e : equiv B} involves the matrix ($A$), as well as the first- and zero-order terms ($b, c$ and $V$, respectively) at the same time, without ``decoupling''. 
On the other hand, in \cite{CD-Mixed} the authors strongly exploited that the sequences appearing in the limit argument in the heat-flow method were $(A,B)$-convex, so that they could use Fatou's lemma. The crucial point there was that the $(A,B)$-convexity is invariant under convolution, while, as already observed in Remark~\ref{r : no inv und conv}, the $(\oA, \oB)$-convexity might be not. Therefore, we would like  to justify the passage of the limits through the integrals by means of the Lebesgue's dominated convergence theorem. However, under the general assumptions on the coefficients, it is not straightforward to uniformly dominate the integrands by integrable functions. Therefore, we will initially prove the bilinear embedding under further assumptions on the coefficients and then we will deduce the arbitrary case by a limit argument.



Without loss of generality we assume $p \geq 2$. Let $\oA=(A,b,c,V), \oB=(B,\beta,\gamma,W)$ satisfy the standard assumptions of Section \ref{s : stand ass}. Suppose that  $\oA, \oB  \in\cB_p(\Omega)$. 
 Fix two closed subspaces $\oV$ and $\oW$ of $H^{1}(\Omega)$ of the type discussed in Section~\ref{s: boundary}. In the spirit of \cite{CD-DivForm}, we will for the moment also
\begin{center}
{\it assume that} $A,B \in C^\infty(\Omega)$, $b,c,\beta,\gamma,V,W \in C_c^\infty(\Omega)$.
\end{center}
 The smoothness of the coefficients implies that $T_t^{\oA,\oV} f, T_t^{\oB,\oW} g \in C^\infty(\Omega)$ for all $f,g \in L^2(\Omega)$; see Lemma~\ref{l : Ttf smooth}. In particular, they and their respective gradients are locally bounded on $\Omega$. This fact, combined with the assumption on the supports of the first- and zero-order terms, will allow us to uniformly dominate certain intergrands by integrable functions and then apply the Lebesgue's dominated convergence theorem. Once the proof with these further assumptions is over, we will apply the regularisation argument from the Appendix to pass to the case of arbitrary coefficients. 

Now, we will prove that 
\begin{eqnarray}
\label{eq: N bil plus}
 \displaystyle\int^{\infty}_{0}\!\int_{\Omega}\sqrt{\mod{\nabla T^{\oA,\oV}_{t}f}^2 + V \mod{T^{\oA,\oV}_{t}f}^2}\sqrt{\mod{\nabla T^{\oB,\oW}_{t}g}^2 +  W \mod{T^{\oB,\oW}_{t}g}^2} \leqsim \left( \|f\|_p^p+\|g\|_q^q \right), 
\end{eqnarray}
for all $f,g\in (L^{p}\cap L^{q})(\Omega)$. Once \eqref{eq: N bil plus} is proved, \eqref{eq: N bil} follows by replacing $f$ and $g$ in \eqref{eq: N bil plus} with $sf$ and $s^{-1}g$ and minimising the right-hand side with respect to $s >0$.
\medskip

We now start the heat-flow method, but for simplicity we omit the subscript $\cW$.  Moreover, we simply denote
\begin{eqnarray*}
\aligned
T_t^{\oA} &= T_t^{\oA,\oV},& & \qquad & T_t^{\oB} &= T_t^{\oB,\oW},& \quad &\forall t >0,&\\
\oL^{\oA} &=\oL^{\oA,\oV},& & \qquad & \oL^{\oB} &= \oL^{\oB,\oW}.&
\endaligned
\end{eqnarray*}
Observe that the boundedness of the potentials $V, W$ implies that
\begin{equation}\label{e : Lrho}
\Dom(\gota_{\oA,\oV}) =\oV, \qquad \Dom(\gota_{\oB,\oW}) =\oW.
\end{equation}
We also recall that for $r \in (1,\infty)$ the operator $-\oL_r^\oA$ stands for the generator of $(T_t^\oA)$ on $L^r(\Omega)$, whenever the semigroup extends on $L^r(\Omega)$.

Let $\cQ=\cQ_{\delta}$ as in \eqref{eq: N Bellman}. Fix $\delta>0$ such that Theorem~\ref{t: Conv Belman} holds. Let us start with a reduction in the spirit of \cite[Section~6.1]{CD-Mixed}; see also \cite[Proposition~7.2]{CDKS-Tril}.

\begin{proposition}
\label{p : reduction be}
Let $p \geq 2$, $q=p/(p-1)$. Let $ \oA=(A,b,c,V), \oB=(B,\beta,\gamma,W)$, $\oV, \oW$ as above. Assume that
\begin{equation}
\label{eq : reduction be}
\hskip -10 pt  \int_\Omega \sqrt{|\nabla u|^2+V|u|^2}\sqrt{|\nabla v|^2+ W|v|^2} \leqsim 2\,\Re\int_\Omega\left( (\partial_\zeta \cQ)(u,v)\oL^{\oA}u+ (\partial_\eta \cQ)(u,v)\oL^{\oB} v\right)
\end{equation}
for $u \in \Dom(\oL^\oA)$, $v \in \Dom(\oL^\oB)$ such that $u,v \in C^\infty(\Omega)$ and $u,v,\oL^\oA u, \oL^\oB v \in (L^p \cap L^q)(\Omega)$. Then the bilinear estimate \eqref{eq: N bil plus} holds for all $f,g \in (L^p \cap L^q)(\Omega)$.
\end{proposition}
\begin{proof}
Fix $f,g\in (L^{p}\cap L^{q})(\Omega)$. Define
$$
\cE(t)=\int_{\Omega}\cQ\left(T^{\oA}_{t}f,T^{\oB}_{t}g\right),\quad t>0.
$$
The estimates \eqref{eq: N 5}, the analyticity and the contractivity of $(T^{\oA}_{t})_{t>0}$ and $(T^{\oB}_{t})_{t>0}$ both in $L^p$ and in $L^q$ (see Theorem~\ref{t : contract} and Corollary~\ref{c: N analytic sem}) and \cite[Proposition~C.1]{CD-Mixed} imply that $\cE$ is well defined, continuous on $[0,\infty)$, differentiable on $(0,\infty)$ with a continuous derivative and
\begin{equation}\label{e : der flow}
 -\cE^{\prime}(t) =2\,\Re\int_{\Omega}\left(\partial_{\zeta}\cQ\left(T^{\oA}_{t}f,T^{\oB}_{t}g\right)\oL^{\oA}T^{\oA}_{t}f  +\partial_{\eta}\cQ\left(T^{\oA}_{t}f,T^{\oB}_{t}g\right)\oL^{\oB}T^{\oB}_{t}g\right).
\end{equation}
The upper pointwise estimates on $\cQ$ (see, for example, \cite[Proposition~5.1]{CD-DivForm}) yield
\begin{equation}\label{e : upp est flow}
-\int_0^\infty \cE^\prime(t) \wrt t \leq \cE(0) \leqsim \|f\|^p_p+\|g\|^q_q.
\end{equation}  
Analyticity of the semigroups gives $T_t^\oA f \in \Dom(\oL^\oA_p) \cap \Dom(\oL^\oA_q)$ and $T_t^\oB g \in \Dom(\oL^\oB_p) \cap \Dom(\oL^\oB_q)$. By the consistency of the semigroups an H\"older's inequality, we have
$$
\Dom(\oL_p^\oA) \cap \Dom(\oL_q^\oA) \subset \Dom(\oL_2^\oA) \subset \oV
$$
and
$$
\Dom(\oL_p^\oB) \cap \Dom(\oL_q^\oB) \subset \Dom(\oL_2^\oB) \subset \oW.
$$
Moreover, the smoothness of $\oA$ and $\oB$ together with Lemma~\ref{l : Ttf smooth} implies that $T_t^\oA f$, $T_t^\oB g \in C^\infty(\Omega)$. Therefore, \eqref{eq : reduction be} applied with the couple $(u,v)=(T_t^\oA f, T_t^\oB g)$, combined with \eqref{e : der flow} and \eqref{e : upp est flow}, gives the desired estimate \eqref{eq: N bil plus}.
\end{proof}
When the first- and zero-order terms are zero, in order to establish an analogue of \eqref{eq : reduction be}, that is \cite[(36)]{CD-Mixed}, Carbonaro and Dragi\v{c}evi\'c approximated the Bellman function $\cQ$ by a sequence $(\cR_{n,\nu})_{n \in \N}$, $\nu \in (0,1)$, of smooth $(A,B)$-convex functions proving that 
\begin{equation}\label{eq: N final Unp}
\aligned
2\,\Re&\int_\Omega\left( (\partial_\zeta \cQ)(u,v)\oL^{A}u+ (\partial_\eta \cQ)(u,v)\oL^{B} v\right)\\
&= \lim_{\nu\rightarrow 0_{+}} \lim_{n\rightarrow +\infty} \int_\Omega H^{(A,B)}_{\cR_{n,\nu}}[(u,v);(\nabla u, \nabla v)]
\endaligned
\end{equation}
and
\begin{equation}\label{eq: N final2 Unp}
\aligned
\lim_{\nu\rightarrow 0_{+}} \lim_{n\rightarrow +\infty} \int_\Omega H^{(A,B)}_{\cR_{n,\nu}}[(u,v);(\nabla u, \nabla v)]\geqsim \int_\Omega |\nabla u| |\nabla v|;
\endaligned
\end{equation}
see \cite[Section~6.1]{CD-Mixed}. While the identity \eqref{eq: N final Unp} is easily adapted to our case since it follows from properties of the sequence $(\cR_{n,\nu})_{n\in\N_+}$ which do not depend on the matrices $A$ and $B$ (\eqref{eq: conv Rnnu}, \eqref{eq: est Rnnu not n}, \eqref{l: invariance} below), the estimate \eqref{eq: N final2 Unp} is not. In fact, it was proved by strongly exploiting  the $(A,B)$-convexity of $\cR_{n,\nu}$ for all $n \in \N_+$ and $\nu \in (0,1)$. The crucial point there was that the $(A,B)$-convexity is invariant under convolution, while, as already observed in Remark~\ref{r : no inv und conv}, the $(\oA, \oB)$-convexity might be not. Therefore, in this case we do not know if $\cR_{n,\nu}$ are $(\oA,\oB)$-convex for all $n \in \N_+$ and $\nu \in (0,1)$. 

\subsection{The approximation sequence $\left(\cR_{n,\nu}\right)_{n \in \N_+}$ of \cite{CD-Mixed}}
For $\nu \in (0,1)$, consider the sequence $\left(\cR_{n,\nu}\right)_{n \in \N_+}$ of smooth $(A,B)$-convex functions  defined in \cite{CD-Mixed}. In \cite{CD-Mixed} the authors proved that
\begin{itemize}
\item
\begin{equation}
\label{eq: conv Rnnu}
\aligned
D\cR_{n,\nu}  &\rightarrow D(\cQ\star\varphi_{\nu}),\\
D^2\cR_{n,\nu} &\rightarrow D^2(\cQ\star\varphi_{\nu})
\endaligned
\end{equation}
pointwise in $\C^2$ as $n\rightarrow\infty$;
\item
there exists $C=C(\nu)>0$ that does not depend on $n$ such that
\begin{equation}
\label{eq: est Rnnu not n}
\mod{(D\cR_{n,\nu})(\omega)}\leq C(\nu)\left(|\omega|^{p-1}+|\omega|^{q-1}
\right),
\end{equation}
for all $\omega\in\C^{2}$, $n\in\N_{+}$ and $\nu\in (0,1)$;
\item
for all $n\in\N_{+}$, $\nu>0$, $u \in \oV$ and $v \in \oW$,
\begin{equation}\label{l: invariance}
(\partial_{\zeta}\cR_{n,\nu})(u,v)\in\oV\quad \text{and}\quad (\partial_{\eta}\cR_{n,\nu})(u,v)\in\oW.
\end{equation}
\end{itemize}
Moreover, it is easy to see that there exists $C=C(\nu) >0$ that does not depend on $n$ such that
\begin{equation}
\label{eq: est 2Rnnu not n}
|D^2R_{n,\nu}(\omega)| \leq C(\nu) (|\omega|^{p-2} +1),
\end{equation}
for all $\omega\in\C^{2}$, $n\in\N_{+}$ and $\nu\in (0,1)$.

\subsection{Proof of \eqref{eq : reduction be}}
Let $u,v$ as in the assumptions of Proposition~\ref{p : reduction be}.


As in \cite{CD-Mixed}, by using \eqref{eq: conv Rnnu}, \eqref{eq: est Rnnu not n}, \cite[Lemma~4(ii)]{CD-Mixed}, the facts that $\cQ \in C^1(\C^2)$, $u, v$, $\oL^\oA u,\oL^\oB v \in (L^p \cap L^q)(\Omega)$ and the Lebesgue's dominated convergence theorem twice, we deduce that
\begin{equation}\label{eq: N final}
\aligned
2\,\Re&\int_\Omega\left( (\partial_\zeta \cQ)(u,v)\oL^{\oA}u+ (\partial_\eta \cQ)(u,v)\oL^{\oB} v\right)\\
&= \lim_{\nu\rightarrow 0_{+}} \lim_{n\rightarrow +\infty} 2\,\Re\int_\Omega\left( (\partial_\zeta \cR_{n,\nu})(u,v)\oL^{\oA}u+ (\partial_\eta \cR_{n,\nu})(u,v)\oL^{\oB} v\right).
\endaligned
\end{equation}
By \eqref{l: invariance} and \eqref{e : Lrho} we can integrate by parts the integral on the right-hand side of \eqref{eq: N final} and, by means of the chain-rule for the composition of smooth functions with vector-valued Sobolev functions, deduce that
\begin{equation}\label{eq: N final 2}
\aligned
2\,\Re&\int_\Omega\left( (\partial_\zeta  \cR_{n,\nu})(u,v)\oL^{\oA}u+ (\partial_\eta  \cR_{n,\nu})(u,v)\oL^{\oB} v\right)\\
&= \int_\Omega \mathbf{H}^{(\oA,\oB)}_{\cR_{n,\nu}}[(u,v);(\nabla u, \nabla v)].
\endaligned
\end{equation}
So far we have repeated the same steps of \cite[Section~6.1]{CD-Mixed}, adapting them to our case, establishing an analogue of \eqref{eq: N final Unp}. Now if we knew, as in \cite{CD-Mixed}, that
 \begin{equation}\label{eq: N final 44}
\aligned
\hskip -65pt \lim_{n \rightarrow +\infty} \int_\Omega \mathbf{H}^{(\oA,\oB)}_{\cR_{n,\nu}}[(u,v);(\nabla u, \nabla v)]\geq \int_\Omega  \mathbf{H}^{(\oA,\oB)}_{\cQ \star \varphi_\nu}[(u,v);(\nabla u, \nabla v)]
\endaligned
\end{equation}
and
\begin{equation}\label{eq: N final 45}
\aligned
\liminf_{\nu \rightarrow 0} \int_\Omega  \mathbf{H}^{(\oA,\oB)}_{\cQ \star \varphi_\nu}[(u,v);(\nabla u, \nabla v)] \geqsim \int_\Omega \sqrt{|\nabla u|^2+V|u|^2}\sqrt{|\nabla v|^2+ W|v|^2},
\endaligned
\end{equation}
we could prove \eqref{eq : reduction be} by combining these two estimates with \eqref{eq: N final} and \eqref{eq: N final 2}. 

Notice that, by \eqref{eq: conv Rnnu} and Corollary~\ref{p: N gen conv reg}, we have
$$
\aligned
\lim_{n \rightarrow +\infty}  \mathbf{H}^{(\oA,\oB)}_{\cR_{n,\nu}}[(u,v);(\nabla u, \nabla v)]&=  \mathbf{H}^{(\oA,\oB)}_{\cQ \star \varphi_\nu}[(u,v);(\nabla u, \nabla v)],\\
\liminf_{\nu \rightarrow 0}  \mathbf{H}^{(\oA,\oB)}_{\cQ \star \varphi_\nu}[(u,v);(\nabla u, \nabla v)] &\geqsim  \sqrt{|\nabla u|^2+V|u|^2}\sqrt{|\nabla v|^2+ W|v|^2}.
\endaligned
$$
Therefore, proving \eqref{eq: N final 44} and \eqref{eq: N final 45} reduces to justifying the passage of the limits through the integrals. While in \cite{CD-Mixed} it was sufficient applying Fatou's lemma since both $\cR_{n,\nu}$ and $\cQ \star \varphi_\nu$ are $(A,B)$-convex, now we will also need the Lebesgue's dominated convergence theorem. We will strongly exploit that $u,v \in C^\infty(\Omega)$.

\subsubsection{Proof of \eqref{eq: N final 44}}
We should proceed in a slightly different way than \cite{CD-Mixed}, since we do not know that $\cR_{n,\nu}$ are $(\oA,\oB)$-convex in $\C^2$  for all $n \in \N_+$ and $\nu \in (0,1)$; see Remark~\ref{r : no inv und conv}. However, by \cite[Theorem~16]{CD-Mixed} $\cR_{n,\nu}$ are $(A,B)$-convex for all $n \in \N_+$ and $\nu \in (0,1)$. Therefore, it is convenient to split the integrand of the right-hand side of \eqref{eq: N final 2} according to the definition \eqref{d : Hess sup gen}, obtaining that
\begin{equation}\label{eq: N final 3}
\aligned
&\int_\Omega \mathbf{H}^{(\oA,\oB)}_{\cR_{n,\nu}}[(u,v);(\nabla u, \nabla v)] \\
&= \int_\Omega H^{(A,B)}_{\cR_{n,\nu}}[(u,v);(\nabla u, \nabla v)] +  \int_\Omega \left(H^{(b,\beta,c,\gamma)}_{\cR_{n,\nu}}[(u,v);(\nabla u, \nabla v)] + G^{(V,W)}_{\cR_{n,\nu}}[(u,v)]\right) \\
&=: I_1^n + I_2^n.
\endaligned
\end{equation}
By using \eqref{eq: conv Rnnu}, \cite[Theorem~16]{CD-Mixed} and Fatou's lemma we deduce that
\begin{equation}
\label{e : est I1n}
\lim_{n \rightarrow +\infty} I_1^n \geq \int_\Omega H^{(A,B)}_{\cQ \star \varphi}[(u,v);(\nabla u, \nabla v)]. 
\end{equation}
Recall that $u, v \in C^\infty(\Omega)$ and $b,\beta,c,\gamma,V,W$ are compactly supported. Hence by using \eqref{eq: est Rnnu not n} and \eqref{eq: est 2Rnnu not n} we get
\begin{equation}
\label{eq: est CDL nnu}
\aligned
\left|H^{(b,\beta,c,\gamma)}_{\cR_{n,\nu}}[(u,v);(\nabla u, \nabla v)]\right| &\leqsim C(\nu)\left( |(u,v)|^{p-1} + |(u,v)|^{q-1} \right)|(\nabla u, \nabla v)| \mathds{1}_K \in L^1(\Omega), \\
\left|G^{(V,W)}_{\cR_{n,\nu}}[(u,v)]\right|& \leqsim C(\nu) \left( |(u,v)|^{p} + |(u,v)|^{q} \right)\mathds{1}_K \in L^1 (\Omega),
\endaligned
\end{equation}
where $K$ is a compact in $\Omega$ containing the supports of $b,\beta,c,\gamma, V,W$.
Therefore, by applying \eqref{eq: conv Rnnu}, \eqref{eq: est CDL nnu} and the Lebesgue's dominated convergence theorem we deduce that
\begin{equation}
\label{e : est I2n}
\lim_{n \rightarrow +\infty} I_2^n = \int_\Omega \left(H^{(b,\beta,c,\gamma)}_{\cQ \star \varphi}[(u,v);(\nabla u, \nabla v)] + G^{(V,W)}_{\cQ \star \varphi}[(u,v)]\right).
\end{equation}
Hence we obtain \eqref{eq: N final 44} by combining \eqref{eq: N final 3}, \eqref{e : est I1n} and \eqref{e : est I2n}.

\subsubsection{Proof of \eqref{eq: N final 45}}
By Corollary~\ref{p: N gen conv reg}, for almost $x\in \Omega$,
\begin{equation}
\label{eq : strict gen conv liminf Q*phi}
\liminf_{\nu \rightarrow 0} \mathbf{H}_{\cQ\star\varphi_{\nu}}^{(\oA,\oB)}[(u,v);(\nabla u,\nabla v)]
\geqsim \sqrt{|\nabla u|^2+V|u|^2}\sqrt{|\nabla v|^2+ W|v|^2}.
\end{equation}
If we knew that
\begin{equation}
\label{e : ipot est}
\liminf_{\nu \rightarrow 0} \int_\Omega H_{\cQ\star\varphi_\nu}^{(\oA,\oB)} \geq \int_\Omega \liminf_{\nu \rightarrow 0} H_{\cQ\star\varphi_\nu}^{(\oA,\oB)},
\end{equation}
we could get \eqref{eq: N final 45} as in \cite{CD-Mixed}. 
But we cannot proceed as in \cite{CD-Mixed} to justify \eqref{e : ipot est}: in fact, we cannot use Fatou's Lemma, not knowing that $\cQ \star \varphi_\nu$ are $(\oA,\oB)$-convex in $\C^2$. 
However, the identity \eqref{eq : Hess conv is conv Hess plus rems} implies that
\begin{equation}
\label{e : Hess conv is conv Hess plus rems (u,v)}
\aligned
\mathbf{H}^{(\oA,\oB)}_{\cQ \star \varphi_\nu}[(u,v);(\nabla u, \nabla v)] = \int_{\R^4}&\mathbf{H}^{(\oA,\oB)}_{ \cQ}[(u,v)-\cW_{2,1}^{-1}(\omega^{\prime});(\nabla u,\nabla v)]\varphi_\nu(\omega^{\prime})\,\wrt\omega^{\prime} \\
&+ R^{(c,\gamma)}_\nu[(u,v);(\nabla u, \nabla v)] +R^{(V,W)}_\nu(u,v) \\
=:  J_\nu &+ R^{(c,\gamma)}_\nu + R^{(V,W)}_\nu,
\endaligned
\end{equation}
for all $\nu \in (0,1)$. The integral  $J_\nu$ is nonnegative for all $\nu \in (0,1)$. Therefore, by using Fatou's lemma, Lemma~\ref{l : lemma on reminds}\eqref{i: reminds tends to 0} and \eqref{e : Hess conv is conv Hess plus rems (u,v)}, we get
\begin{equation}
\label{e : lim Jnu}
\liminf_{\nu \rightarrow 0} \int_\Omega J_\nu  \geq \int_\Omega \liminf_{\nu \rightarrow 0} J_\nu = \int_\Omega \liminf_{\nu \rightarrow 0} \mathbf{H}^{(\oA,\oB)}_{\cQ \star \varphi_\nu}[(u,v);(\nabla u, \nabla v)]
\end{equation}
 On the other hand, sending $\nu \rightarrow 1$ on the right-hand side of Lemma~\ref{l : lemma on reminds}\eqref{i : reminds dominated} implies that there exists a constant $C>0$, that does not depend on $\nu$, such that
\begin{equation*}
\label{e : est resti(u,v)}
\aligned
|R^{(c,\gamma)}_\nu| &\leq C \left(1+|u|^{p-2}+|v|^{2-q}\right) |(\nabla u, \nabla v)| \mathds{1}_K \in L^1(\Omega), \\
|R^{(V,W)}_\nu| &\leq C  \left(1+ |u|^{p-1} + |v|^{q-1} +|v|\right) \mathds{1}_K  \in L^1(\Omega).
\endaligned
\end{equation*}
All the right-hand terms of these estimates are integrable functions, since $u,v \in C^\infty(\Omega)$. Therefore, by using Lemma~\ref{l : lemma on reminds}\eqref{i: reminds tends to 0} and the Lebesgue's dominated convergence theorem, we obtain that
\begin{equation}
\label{e : int rmds vanishes}
\aligned
&\lim_{\nu \rightarrow 0} \int_\Omega R^{(c,\gamma)}_\nu[(u,v);(\nabla u, \nabla v)] =0,\\
&\lim_{\nu \rightarrow 0} \int_\Omega R^{(V,W)}_\nu(u,v) =0.
\endaligned
\end{equation}
Then, by combining \eqref{e : Hess conv is conv Hess plus rems (u,v)}, \eqref{e : lim Jnu} and \eqref{e : int rmds vanishes}, we get \eqref{e : ipot est} and thus \eqref{eq: N final 45}, as explained above.

\section{Maximal regularity and functional calculus: proof of Theorem~\ref{t: N principal}}\label{s: max funct}
The following result is modelled after \cite[Proposition 20]{CD-Mixed}. See \cite[Sections~7.1 and 7.2]{CD-Mixed} for the necessary terminology and references.
\begin{proposition}\label{p: N D-V bis} 
Let $\Omega, \oA=(A, b, c , V), \oV$ satisfy the standard assumptions of Section \ref{s : stand ass}. Suppose that  $p >1$ and $\oA \in (\cS_p \cap \cS_2)(\Omega)$. Let $-\oL_p$ be the generator of $(T_{t})_{t>0}$ on $L^{p}(\Omega)$. If $\omega_{H^{\infty}}(\oL_{p})<\pi/2$, then $\oL_{p}$ has parabolic maximal regularity.
 \end{proposition} 
\subsection{Proof of Theorem~\ref{t: N principal}}
We prove Theorem~\ref{t: N principal} as Carbonaro and  Dragi\v{c}evi\'c proved \cite[Theorem~3]{CD-Mixed} in \cite[Section~7.3]{CD-Mixed} when $b=c=V=0$.
Without loss of generality we suppose $p\ge2$. Let  $A,b,c,V$ satisfy the standard assumptions of Section \ref{s : stand ass}. In light of Proposition~\ref{p: N D-V bis} it suffices to show that 
$$
(A,b,c,V) \in \cB_p(\Omega) \implies \omega_{H^{\infty}}(\oL_{p})<\pi/2.
$$ 

Observe that $\oL_2^{A^{*},c,v,V} = \left( \oL^{A,b,c,V}\right)_2^*$, so $T^{A^{*},c,b,V}_{t}=\left(T^{A,b,c,V}_{t}\right)^{*}$ for all $t> 0$.  Set $T_t = T^{A,b,c,V}_{t}$ and $T_t^*=T^{A^{*},c,b,V}_{t}$ for all $t>0$. 

By Proposition \ref{p : property of B}\eqref{i : cont B in phi} and Proposition \ref{p: more on B}\eqref{i : invarianza per coniugazione}, there exists $\theta\in (0,\pi/2)$ such that 
$$
\aligned
(e^{\pm i\theta}A,e^{\pm i\theta}b,e^{\pm i\theta}c,(\cos\theta)V) & \in \cB_p(\Omega),\\
(e^{\mp i\theta}A^{*},e^{\mp i\theta}c,e^{\mp i\theta}b,(\cos\theta)V) &\in \cB_p(\Omega).
\endaligned
$$
Moreover, for every $r \in [q,p]$ both $(T_{t})_{t>0}$ and $(T^{*}_{t})_{t>0}$ are analytic (and contractive) in $L^r(\Omega)$ in the cone $\bS_{\theta}$; see Corollary \ref{c: N analytic sem}.
\smallskip

By Theorem~\ref{t: N bil compl} and \eqref{e: compl times sem} there exists $C>0$ such that   
\begin{equation*}
\aligned
\int^{\infty}_{0}\!&\int_{\Omega}\sqrt{\mod{\nabla T_{te^{\pm i\theta}}f}^2 +  (\cos\theta)V \mod{T_{te^{\pm i\theta}}f}^2 }\sqrt{\mod{\nabla T^{*}_{te^{\mp i\theta}}g}^2+  (\cos\theta)W \mod{T^{*}_{te^{\mp i\theta}}g}^2} \wrt x \wrt t \\
&\leq C \norm{f}{p}\norm{g}{q}, 
\endaligned
\end{equation*}
for all $f,g\in (L^{p}\cap L^{q})(\Omega)$. Therefore,
\begin{equation}\label{eq: bilineq*}
\hskip -3pt \int^{\infty}_{0}\!\int_{\Omega}\sqrt{\mod{\nabla T_{te^{\pm i\theta}}f}^2 + V \mod{T_{te^{\pm i\theta}}f}^2 }\sqrt{\mod{\nabla T^{*}_{te^{\mp i\theta}}g}^2+ W \mod{T^{*}_{te^{\mp i\theta}}g}^2} \wrt x \wrt t \leq \frac{C}{\cos\theta} \norm{f}{p}\norm{g}{q}, 
\end{equation}
for all $f,g\in (L^{p}\cap L^{q})(\Omega)$.

On the other hand, the integration by parts \eqref{eq: ibp} and Proposition~\ref{p : property of B}\eqref{i : V dom Rec} give
\begin{equation}
\label{eq: CDMY-bil}
\aligned
\biggl|\int_{\Omega}\oL_{2}&T_{te^{\pm i\theta}}f\,\,\overline{T^{*}_{te^{\mp i\theta}}g }\wrt x\biggr| \\
&\leqsim   \int_{\Omega}\sqrt{\mod{\nabla T_{te^{\pm i\theta}}f}^2 +  V \mod{T_{te^{\pm i\theta}}f}^2 }\sqrt{\mod{\nabla T^{*}_{te^{\mp i\theta}}g}^2+  V \mod{T^{*}_{te^{\mp i\theta}}g}^2} \wrt x,
\endaligned
\end{equation}
By combining \eqref{eq: bilineq*} and \eqref{eq: CDMY-bil}, we deduce that
$$
\int^{\infty}_{0}\!\mod{\int_{\Omega}\oL_{p}T_{2te^{\pm i\theta}}f\, \overline{g }\wrt x}\wrt t \, \leqsim \, \norm{f}{p}\norm{g}{q},
$$
for all $f,g\in (L^{p}\cap L^{q})(\Omega)$. Analyticity of $(T_{t})_{t>0}$ in $L^{p}(\Omega)$, Fatou's lemma and a density argument show that
\begin{equation}\label{eq: N CDMY}
\int^{\infty}_{0}\!\mod{\int_{\Omega}\oL_{p}T_{te^{\pm i\theta}}f\, \overline{g }\wrt x}\wrt t \, \leqsim \, \norm{f}{p}\norm{g}{q},
\end{equation}
for all $f\in L^{p}(\Omega)$ and all $g\in L^{q}(\Omega)$. 

We now apply \cite[Theorem~4.6 and Example~4.8]{CDMY} to the dual subpair $\sk{\ovR(\oL_{p})}{\ovR(\oL^{*}_{q})}$ and the dual operators $(\oL_{p})_{\vert\vert}$, $(\oL^{*}_{q})_{\vert\vert}$ \cite[p.~64]{CDMY}, and deduce from \eqref{eq: N CDMY} that $\omega_{H^{\infty}}(\oL_{p})\leq \pi/2-\theta$.

\numberwithin{theorem}{section}
\numberwithin{equation}{section}

\appendix  
\section{Approximation with regular-coefficient operators}\label{a: A}
We learnt the following regularisation method by Carbonaro and Dragi\v{c}evi\'c in \cite[Appendix]{CD-DivForm}.

For any $\phi \in (0,\pi/2)$ denote $\phi^*:=\pi/2-\phi$. Write $I=[0,1]$. Consider a one-parameter family $\{A_s \, : \, s \in I\}$ of $d \times d$ complex matrix functions on $\Omega$,  two one-parameter families $\{b_s \, : \, s \in I\}, \, \{c_s \, : \, s \in I\}$ of $d$-dimensional complex vector-valued functions on $\Omega$ and a one parameter family $\{V_s \, : \, s \in I\}$ of nonnegative functions on $\Omega$. Suppose that there exist $\mu,M,C >0$ such that
\begin{enumerate}[(H1)]
\item
\label{h1}
 $\oA_s:=(A_s,b_s,c_s,V_s) \in \cB_{\mu,M}(\Omega)$ for all $s\in I$;
\item
\label{h4}
$\max\{\|A_s\|_\infty,\|b_s\|_\infty,\|c_s\|_\infty,\|V_s\|_\infty\} \leq C$ for all $s \in I$.
\item
\label{h3}
For a.e. $x \in \Omega$
\begin{equation*}
\aligned
\lim_{s \rightarrow 0} \| A_s(x) -A_0(x)\|_{\cB(\C^d)}&= 0,\\
\lim_{s \rightarrow 0} \| b_s(x) -b_0(x)\|_{\cB(\C^d,\C)}&= 0,\\
\lim_{s \rightarrow 0} \| c_s(x) -c_0(x)\|_{\cB(\C,\C^d)}&= 0,\\
\lim_{s \rightarrow 0} \| V_s(x) -V_0(x)\|_{\cB(\C)}&= 0.
\endaligned
\end{equation*}
\end{enumerate}
Denote by $\gota_s$ the sesquilinear form associated with $\oA_s$ and by $\oL_s = \oL^{\oA_s}$ the operator on $L^2(\Omega)$ associated with $\gota_s$. Recall from Section~\ref{s: Neumann introduction} that each $\oL_s$ is sectorial with sectoriality angle $\omega(\oL_s) \leq \vartheta_0$, where $\vartheta_0=\theta_0(\mu,M,\Lambda)$ is a positive angle smaller than $\pi/2$. 
In particular,
\begin{equation}
\label{eq: unif sect}
\underset{\zeta \in \C\setminus \bS_\vartheta}{\sup}\|\zeta (\zeta - \oL_s)^{-1}\|_{L^2(\Omega)} < \infty, \quad \forall \vartheta \in (\vartheta_0,\pi/2).
\end{equation}
Set $\cH_0 = L^2(\Omega)$ and $\cH_1 =L^2(\Omega,\C^d)$. For $\zeta \in \C \setminus \bS_{\vartheta_0}$ define the operators \\$G_{\oL_s}(\zeta) :\cH_0 \rightarrow \cH_1$ and $S_{\oL_s}(\zeta) \, : \, C_c^\infty(\Omega,\C^d) \rightarrow \cH_0$ by
$$
G_{\oL_s}(\zeta) := \nabla (\zeta - \oL_s)^{-1}, \qquad S_{\oL_s}(\zeta) := (\zeta - \oL_s)^{-1} \div.
$$
\begin{lemma}
\label{l : est resolvents}
Assume {\rm (H1)} holds. Then for every $\vartheta \in (\vartheta_0, \pi/2)$ there exists $C=C(\mu,M,\Lambda,\vartheta)>0$ such that
\begin{equation*}
\begin{aligned}
|\zeta|^{1/2} \|G_{\oL_s}(\zeta) f \|_{\cH_1} &\leq C \|f\|_{\cH_0},\\
|\zeta|^{1/2} \|S_{\oL_s}(\zeta) F \|_{\cH_0} &\leq C \|F\|_{\cH_1}, \\
 \|\nabla S_{\oL_s}(\zeta) F \|_{\cH_1} &\leq C \|F\|_{\cH_1},
\end{aligned}
\end{equation*}
for all $s \in I$, $f \in L^2(\Omega)$, $F \in C_c^\infty(\Omega;\C^d)$ and $\zeta \in \C \setminus \bS_\vartheta$. Moreover, the very same estimates hold with $\oL_s$ replaced by $\oL_s^*$.
\end{lemma}
\begin{proof}
Since the proof is basically the same as that of \cite[Lemma~A.1]{CD-DivForm} for unperturbed divergence-form operators, let us prove only the first inequality. The condition (H1) implies that
$$
\aligned
\mu \|\nabla(\zeta -\oL_s)^{-1}f\|_{\cH_1}^2 & \leq \Re\gota_s\left((\zeta -\oL_s)^{-1}f,(\zeta -\oL_s)^{-1}f\right) \\
&= \Re \sk{\oL_s (\zeta -\oL_s)^{-1}f}{(\zeta -\oL_s)^{-1}f}_{\cH_0}\\
&= \Re \sk{\zeta(\zeta -\oL_s)^{-1}f}{(\zeta -\oL_s)^{-1}f}_{\cH_0}-\Re\sk{f}{(\zeta -\oL_s)^{-1}f}_{\cH_0}.
\endaligned
$$
The first inequality of the lemma now follows from \eqref{eq: unif sect}. Since $\oL_s^*= \oL^{A_s^*,c_s,b_s,V_s}$ and $(A_s^*,c_s,b_s,V_s) \in \cB_{\mu,M}(\Omega)$ by Proposition~\ref{p: more on B}\eqref{i : invarianza per coniugazione}, the same estimate clearly holds with $\oL_s$ replaced by $\oL_s^*$. 
\end{proof}
\begin{remark}
\label{r : div e grad res cont in L2}
The preceding lemma implies that for $\zeta \in \C \setminus \bS_{\theta_0}$ the operators $S_{\oL_s}(\zeta)$ and $\nabla S_{\oL_s}(\zeta)$ admit unique extensions to bounded operators $\cH_1 \rightarrow \cH_0$ and $\cH_1 \rightarrow \cH_1$, respectively. Moreover, $S_{\oL_s}^*(\zeta)= G_{\oL_s^*}(\overline{\zeta}).$
\end{remark}
\begin{lemma}
\label{l : diff resolv}
For every $s \in I$, $\zeta \in \C \setminus \bS_{\vartheta_0}$ and $f \in L^2(\Omega)$ we have
$$
\aligned
(\zeta-\oL_0)^{-1}f - (\zeta-\oL_s)^{-1}f = S_{\oL_s}&(\zeta) \circ M_{A_0 -A_s} \circ G_{\oL_0}(\zeta)f \\
&+ (\zeta - \oL_s)^{-1} \circ M_{b_0-b_s} \circ G_{\oL_0}(\zeta)f \\
& + S_{\oL_s}(\zeta) \circ M_{c_0 -c_s} \circ (\zeta - \oL_0)^{-1}f \\
&+  (\zeta - \oL_s)^{-1}\circ  M_{V_0-V_s} \circ (\zeta - \oL_0)^{-1}f
\endaligned
$$
where we denote
$$
\aligned
M_{A_0 -A_s} &: \cH_1 \rightarrow \cH_1 \text{ the operator of multiplication by } A_0 -A_s,\\
M_{b_0 -b_s} &: \cH_1 \rightarrow \cH_0 \text{ the operator } \sk{\cdot}{b_0 -b_s}, \\
M_{c_0 -c_s} &: \cH_0 \rightarrow \cH_1 \text{ the operator of multiplication by } c_0 -c_s, \\
M_{V_0 -V_s} &: \cH_0 \rightarrow \cH_0 \text{ the operator of multiplication by } V_0 -V_s.
\endaligned
$$
\end{lemma}
\begin{proof}
Let $\zeta \in \C \setminus \bS_{\vartheta_0}$. The identity
$$
(\zeta-\oL_0)^{-1} - (\zeta-\oL_s)^{-1} = (\zeta-\oL_s)^{-1}\oL_0(\zeta-\oL_0)^{-1} -\oL_s(\zeta-\oL_s)^{-1}(\zeta-\oL_0)^{-1}
$$
implies that, for all $f \in L^2(\Omega)$ and $g \in \Dom(\oL_s^*)$,
$$
\aligned
& \hskip-30 pt\sk{(\zeta-\oL_0)^{-1}f - (\zeta-\oL_s)^{-1}f}{g}_{\cH_0} \\
=&\sk{\oL_0(\zeta -\oL_0)^{-1}f}{(\overline{\zeta}-\oL_s^*)^{-1}g}_{\cH_0} -\overline{\sk{\oL_s^*(\overline{\zeta}-\oL_s^*)^{-1}g}{(\zeta -\oL_0)^{-1}f}}_{\cH_0} \\
=&\,\gota_0\left((\zeta -\oL_0)^{-1}f,(\overline{\zeta}-\oL_s^*)^{-1}g\right) - \gota_s\left((\zeta -\oL_0)^{-1}f,(\overline{\zeta}-\oL_s^*)^{-1}g\right) \\
= &\sk{M_{A_0-A_s} G_{\oL_0}(\zeta)f}{G_{\oL_s^*}(\overline{\zeta})g}_{\cH_1} + \sk{M_{b_0-b_s} G_{\oL_0}(\zeta)f}{(\overline{\zeta}-\oL_s^*)^{-1}g}_{\cH_0}\\
 &+ \sk{M_{c_0-c_s} (\zeta -\oL_0)^{-1}f}{G_{\oL_s^*}(\overline{\zeta})g}_{\cH_1} + \sk{M_{V_0-V_s}(\zeta -\oL_0)^{-1}f}{(\overline{\zeta} -\oL_s^*)^{-1}g}_{\cH_0}.
\endaligned
$$
We conclude by invoking Remark \ref{r : div e grad res cont in L2}.
\end{proof}
\begin{lemma}
\label{l : conv pot and grad sem}
Fix $f \in L^2(\Omega)$. Assuming {\rm (H1)}, {\rm (H2)} and {\rm (H3)}, for every $z \in \bS_{\vartheta_0^*}$ we have
\begin{eqnarray}
\label{eq: conv pot sem}
V_s^{1/2} T_z^{\oA_s} f &\hskip -5pt \rightarrow V_0^{1/2} T_z^{\oA_0} f \quad &\text{ in } \, L^2(\Omega),\\
\label{eq: conv grad sem}
\nabla T_z^{\oA_s} f  &\hskip -18pt\rightarrow\nabla T_z^{\oA_0} f \quad &\text{ in }\, L^2(\Omega;\C^d),
\end{eqnarray}
as $s \rightarrow 0$.
\end{lemma}
\begin{proof}
In order to prove \eqref{eq: conv pot sem}, by {\rm (H2)} and {\rm (H3)} it suffices to prove that 
\begin{equation}
\label{eq: conv sem}
T_z^{\oA_s} f \rightarrow T_z^{\oA_0} f \quad {\rm in} \,\, L^2(\Omega),
\end{equation}
as $s \rightarrow 0$.
Let $\vartheta^* \in (0,\pi/2)$ be such that $|\arg z|< \vartheta^* < \vartheta_0^*$. Fix $\delta >0$ and denote by $\gamma$ the positively oriented boundary of $\bS_\vartheta \cup \{\zeta \in \C : |\zeta|<\delta\}$. For $s \in I$ and $\zeta \in \gamma$ define 
$$
\aligned
U(s,\zeta) =\, &S_{\oL_s}(\zeta) \circ M_{A_0 -A_s} \circ G_{\oL_0}(\zeta) + (\zeta - \oL_s)^{-1} \circ M_{b_0-b_s} \circ G_{\oL_0}(\zeta)\\
& + S_{\oL_s}(\zeta) \circ M_{c_0 -c_s} \circ (\zeta - \oL_0)^{-1} +  (\zeta - \oL_s)^{-1}\circ  M_{V_0-V_s} \circ (\zeta - \oL_0)^{-1}.
\endaligned
$$
Then by \cite[Lemma~2.3.2]{Haase} and Lemma \ref{l : diff resolv},
$$
T_z^{\oA_0} f - T_z^{\oA_s} f = \frac{1}{2 \pi i} \int_\gamma e^{-z \zeta} U(s,\zeta)f \wrt\zeta.
$$
Therefore, by Minkowsky's integral inequality, \eqref{eq: unif sect} and the second estimate of Lemma \ref{l : est resolvents},
\begin{equation}
\label{eq : norm 2 diff res}
\aligned
\|T_z^{\oA_0} f - T_z^{\oA_s} f\|_2 \leqsim &\int_\gamma |e^{-z \zeta}| \cdot \| U(s,\zeta)f\|_2 \wrt|\zeta| \\
\leqsim&\int_\gamma  e^{- \Re(z \zeta)} |\zeta|^{-1/2} \|M_{A_0-A_s} G_{\oL_0}(\zeta) f \|_2 \wrt|\zeta| \\
& +\int_\gamma e^{- \Re(z \zeta)} |\zeta|^{-1} \|M_{b_0-b_s} G_{\oL_0}(\zeta) f \|_2 \wrt|\zeta| \\
& +\int_\gamma e^{- \Re(z \zeta)}  |\zeta|^{-1/2} \|M_{c_0-c_s}   (\zeta - \oL_0)^{-1}f \|_2 \wrt|\zeta|\\
& +\int_\gamma e^{- \Re(z \zeta)}|\zeta|^{-1} \|M_{V_0-V_s}   (\zeta - \oL_0)^{-1}f \|_2 \wrt|\zeta|.
\endaligned
\end{equation}
By {\rm (H2)}, {\rm (H3)}, Lemma~\ref{l : est resolvents} and the Lebesgue dominated convergence theorem we obtain
$$
\aligned
\lim_{s \rightarrow 0} \| M_{A_0 -A_s} G_{\oL_0}(\zeta) f \|_2 &= 0, \\
\lim_{s \rightarrow 0} \| M_{b_0 -b_s} G_{\oL_0}(\zeta) f \|_2 &= 0, \\
\lim_{s \rightarrow 0} \| M_{c_0 -c_s}  (\zeta - \oL_0)^{-1}f  \|_2 &= 0, \\
\lim_{s \rightarrow 0} \| M_{V_0 -V_s} (\zeta - \oL_0)^{-1}f  \|_2 &= 0,
\endaligned
$$
for all $\zeta \in \gamma$. 
Moreover, by Lemma~\ref{l : est resolvents} again and by \eqref{eq: unif sect}, we have
$$
\aligned
\| M_{A_0 -A_s} G_{\oL_0}(\zeta) f \|_2 &\leqsim |\zeta|^{-1/2}\|f\|_2, \\
\| M_{b_0 -b_s} G_{\oL_0}(\zeta) f \|_2 &\leqsim |\zeta|^{-1/2}\|f\|_2, \\
\| M_{c_0 -c_s}  (\zeta - \oL_0)^{-1}f  \|_2 &\leqsim |\zeta|^{-1}\|f\|_2, \\
\| M_{V_0 -V_s} (\zeta - \oL_0)^{-1}f  \|_2 &\leqsim |\zeta|^{-1}\|f\|_2,
\endaligned
$$
for all $\zeta \in \gamma$. Now \eqref{eq: conv sem} follows from the Lebesgue dominated convergence theorem.

Similarly one can prove \eqref{eq: conv grad sem}, using in this case the first and the third estimate of Lemma \ref{l : est resolvents} to get the analogue of \eqref{eq : norm 2 diff res}.
\end{proof}

\subsection{Convolution with approximate identity}\label{s : appr by conv}
Let $k\in C_c^\infty(\R^d)$ be a radial function such that $0\leq k\leq1$, ${\rm supp}\, k\subset B_{\R^{d}}(0,1)$ and $\int k=1$. For $\varepsilon >0$ define $k_\varepsilon(x) := \varepsilon^{-d} k(x/\varepsilon)$. Let $A \in \cA(\Omega)$, $b, c \in L^\infty(\Omega, \C^d)$ and $V \in L^\infty(\Omega, \R_+)$. Suppose that $b,c,V$ are compactly supported and that 
\begin{equation}
\label{e : same supp}
{\rm supp}\, b, \, {\rm supp}\, c \subset {\rm supp} \,V.
\end{equation}
 If $\oA=(A,b,c,V) \in \cB(\Omega)$ and $p \in (1,\infty)$, we define 
\begin{equation*}
\aligned
A_\varepsilon &:= {(\widetilde{A} * k_\varepsilon)|}_{\Omega}, \\
b_\varepsilon &:= {(b * k_\varepsilon)|}_{\Omega},\\
c_\varepsilon &:= {(c * k_\varepsilon)|}_{\Omega},\\
V_\varepsilon &:= {(V * k_\varepsilon)|}_{\Omega},\\
\oA_\varepsilon&:= (A_\varepsilon,b_\varepsilon,c_\varepsilon,V_\varepsilon),
\endaligned
\end{equation*}
where
\begin{equation*}
\widetilde{A}(x) :=
\begin{cases}
A(x),& {\rm if } \,\,\, x \in \Omega;\\
(p^*/p) I_{\R^d},& {\rm if } \,\,\, x \in \R^d \setminus \Omega.
\end{cases}
\end{equation*}
Here, $p^*= \max\{p, p/(p-1)\}$. The constant $p^*/p$ is a normalisation factor.
\begin{lemma}
\label{l : inv after reg}
Suppose that $\oA \in \cB_{\mu,M}(\Omega)$ and $\oA_\varepsilon$ are as above and $p \in (1,\infty)$.  Then
\begin{enumerate}[{\rm (i)}]
\item 
\label{i : conv coeff}
 For a.e. $x \in \Omega$
\begin{equation*}
\aligned
\lim_{\varepsilon \rightarrow 0} \| A_\varepsilon(x) -A(x)\|_{\cB(\C^d)}&= 0,\\
\lim_{\varepsilon \rightarrow 0} \| b_\varepsilon(x) -b(x)\|_{\cB(\C^d,\C)}&= 0,\\
\lim_{\varepsilon \rightarrow 0} \| c_\varepsilon(x) -c(x)\|_{\cB(\C,\C^d)}&= 0,\\
\lim_{\varepsilon \rightarrow 0} \| V_\varepsilon(x) -V(x)\|_{\cB(\C)}&= 0.
\endaligned
\end{equation*}
\item
\label{i: inv 2-cond after reg}
There exists $\varepsilon_0 \in (0,1]$ such that $\oA_\varepsilon \in \cB_{\mu,M}(\Omega)$ for all $\varepsilon \in (0,\varepsilon_0)$;
\item
\label{i: inv p-cond after reg}
If $\oA \in \cS_p(\Omega)$, then for all $\varepsilon \in (0,\varepsilon_0)$, 
$$
\aligned
\oA_\varepsilon &\in \cS_p(\Omega),\\
\mu_p(\oA_\varepsilon) &\geq\mu_p(\oA).
\endaligned
$$
Moreover, $\mu_p(\oA) = \lim_{\varepsilon \rightarrow 0} \mu_p(\oA_\varepsilon)$.

In particular, if $\oA \in \cB_p(\Omega)$ then $\oA_\varepsilon \in \cB_p(\Omega)$ for all  $\varepsilon \in (0,\varepsilon_0)$;
\item
\label{i: unif bound reg}
There exists $C>0$ such that $\max\{\|A_\varepsilon\|_\infty,\|b_\varepsilon\|_\infty,\|c_\varepsilon\|_\infty,\|V_\varepsilon\|_\infty\} \leq C$ for all $\varepsilon>0$.
\end{enumerate}
\end{lemma}
\begin{proof}
Let prove only the first limit in \eqref{i : conv coeff}. The others are proved in the same way. As in the proof of \cite[Lemma~A.5(i)]{CD-DivForm}, it is enough to show that for a.e. $x \in \Omega$ we have $ A_\varepsilon(x)\xi \rightarrow A(x)\xi$ for all $\xi \in \C^d$. This is true because each $a_{ij} \in L^\infty(\Omega) \subset L^1_{\rm loc}(\Omega)$ so that $(\widetilde{a_{ij}} * k_\varepsilon)(x)$ tends to $\widetilde{a_{ij}}(x) = a_{ij}(x)$ for a.e. $x \in \Omega$, as $\varepsilon \rightarrow 0$.

Let $r \in (1,\infty)$ and $\oA \in \cS_r(\Omega)$.
Then, by the Cauchy-Schwarz inequality and the fact that $\int k=1$, we get for all $x \in \Omega$ and all $\varepsilon>0$,
$$
\aligned
|b_\varepsilon(x) -c_\varepsilon(x)| &\leq \int_{\R^d} \left|b(x-y)-c(x-y) \right| k^{1/2}_\varepsilon(y) \cdot k^{1/2}_\varepsilon(y) \wrt y  \\
& \leq \left(\int_{\R^d} |b(x-y) -c(x-y)|^2 k_\varepsilon(y) \wrt y\right)^{1/2} \\
&\leq M \left(\int_{\R^d}V(x-y) k_\varepsilon(y) \wrt y\right)^{1/2} \\
&= M \sqrt{V_\varepsilon(x)}.
\endaligned
$$
It remains to estimate $\Gamma_r^{\oA_\varepsilon}$ from below. Let $\varepsilon_0 \in (0,1]$ such that 
\begin{equation}
\label{eq: transl supp still in Om}
{\rm supp V} + \overline{B(0,\varepsilon)} \subseteq \Omega, \quad \forall\epsilon \in (0,\varepsilon_0).
\end{equation}
Let $x \in \Omega$. Then, by combining \eqref{e : same supp}, \eqref{eq: transl supp still in Om} and the fact that ${\rm supp}\, k_\varepsilon \subset B_{\R^{d}}(0,\varepsilon)$, we obtain that, for all $\varepsilon \in (0,\varepsilon_0)$,
\begin{equation}
\label{e : A8}
\aligned
\Gamma_r^{\oA_\varepsilon}(x,\xi)   =&  \int_{\R^d} \Gamma_r^{\widetilde{A},b,c,V}(x-y,\xi) k_\varepsilon(y) \,\wrt y \\
 =&  \int_{\{y \in B(0,\varepsilon) \,:\, x-y \,\in \,{\rm supp}V\}} \Gamma_r^\oA(x-y,\xi) k_\varepsilon(y) \,\wrt y \\
& +  \int_{\{y\in B(0,\varepsilon) \,:\, x-y \, \not\in\, {\rm supp}V\}}  \Re\sk{\widetilde A(x-y)\xi}{\cJ_r \xi}  k_\varepsilon(y) \,\wrt y \\
=&: I_1+I_2,
\endaligned
\end{equation}
for all $\xi \in \C^d$. Since $\oA \in\cS_r(\Omega)$, we estimate $I_1$ by
\begin{equation}
\label{e : A9}
\aligned
I_1 &\geq \mu_r(\oA) \int_{\{y \in B(0,\varepsilon) \,:\, x-y \,\in\, {\rm supp}V\}}  \left(|\xi|^2 + V(x-y)\right) k_\varepsilon(y) \wrt y \\
&= \mu_r(\oA) |\xi|^2  \int_{\{y \in B(0,\varepsilon) \,:\, x-y \,\in\, {\rm supp}V\}}  k_\varepsilon(y) \wrt  y + \mu_r(\oA) V_\varepsilon(x).
\endaligned
\end{equation}
On the other hand, since $A \in \cA_r(\Omega)$ and $\Delta_r(I_{\R^d})=2/r^*$, we estimate $I_2$ by
\begin{equation}
\label{e : A10}
\aligned
I_2 &\geq \min\left\{  \frac{r}{2}\Delta_r(A), \frac{r p^*}{p r^*}\right\}|\xi|^2\int_{\{y\in B(0,\varepsilon) \,:\, x-y \,\not\in \,{\rm supp}V\}}  k_\varepsilon(y) \wrt y \\
&\geq \min\left\{ \mu_r(\oA), \frac{r p^*}{p r^*}\right\}|\xi|^2 \int_{\{y\in B(0,\varepsilon) \,:\, x-y \,\not\in \,{\rm supp}V\}}   k_\varepsilon(y) \wrt y,
\endaligned
\end{equation}
where in the last inequality we used \eqref{e : mu e delta}. Hence, by combining \eqref{e : A8}, \eqref{e : A9} and \eqref{e : A10} we get
\begin{equation*}
\label{e : A11}
\aligned
 \Gamma_r^{\oA_\varepsilon}(x,\xi) &\geq \min\left\{\mu_r(\oA), \frac{rp^*}{pr^*} \right\} (|\xi|^2+ V_\varepsilon(x)).
\endaligned
\end{equation*}
for all $x \in \Omega$ and all $\xi \in \C^d$. 
Note that $p^*/p \geq 1$ for all $p \in (1,\infty)$. On the other hand, $\mu_2(\oA) \leq 1$ by \eqref{e : mu e delta}. Therefore, by taking $r=2$,  we prove \eqref{i: inv 2-cond after reg}.

By taking $r=p$ and by \eqref{e : mu e delta} again, we prove that $\cA_\varepsilon \in \cS_p(\Omega)$ and that 
\begin{equation}
\label{e : A12}
\mu_p(\oA_\varepsilon)\geq \mu_p(\oA),
\end{equation}
for all  $\varepsilon \in (0,\varepsilon_0)$. Moreover, by definition of $\mu_p$ and the continuity of \\
$x \mapsto (A_\varepsilon(x), b_\varepsilon(x), c_\varepsilon(x), V_\varepsilon(x))$, we get $\Gamma_p^{\oA_\varepsilon}(x,\xi) \geq \mu_p(\oA_\varepsilon) (|\xi|^2 + V_\varepsilon(x))$ for all $x \in \Omega$, $\xi \in \C^d$ and $\varepsilon >0$. The limit in \eqref{i: inv p-cond after reg} now follows from \eqref{i : conv coeff} and \eqref{e : A12}.

Item \eqref{i: unif bound reg} follows by the definition of convolution and by the boundedness of $A,b,c,V$.
\end{proof}

\subsection{Truncations of first- and zero-order terms by cut-off functions}\label{s : appr by cut}
Let $\{K_n\}_{n \in \N_+}$ be a family of compact subsets in $\Omega$ such that $K_n \subset {\rm int}(K_{n+1})$ for all $n \in \N_+$ and $\bigcup_{n=1}^\infty {\rm int}(K_n)=\Omega$. Let $\{\psi_n\}_{n \in \N_+}$ a family of compactly supported $C^\infty$ functions in $\Omega$ such that $0 \leq \psi_n \leq 1$, $\psi_n = 1$ on a neighbourhood of $K_n$ and ${\rm supp} \, \psi_n \subset {\rm int}(K_{n+1}) $ for all $n \in \N_+$.  If $A \in \cA(\Omega)$, $b,c \in L^\infty(\Omega,\C^d)$ and $V \in L^\infty(\Omega,\R_+)$, we define 
\begin{equation}
\label{d : terms cut n}
\aligned
b_n &:= \psi_n b,\\
c_n &:= \psi_n c,\\
V_n &:= \psi_n V,\\
\oA_n&:= (A,b_n,c_n,V_n).
\endaligned
\end{equation}
\begin{lemma}
\label{l : inv after cut}
For every $\oA= (A,b,c,V) \in \cB_{\mu,M}(\Omega)$ and $\oA_n$  as above and $p \in (1,\infty)$ we have
\begin{enumerate}[{\rm (i)}]
\item 
\label{i : conv coeff cut}
 For a.e. $x \in \Omega$,
\begin{equation*}
\aligned
\lim_{n \rightarrow \infty} \| b_n(x) -b(x)\|_{\cB(\C^d,\C)}&= 0,\\
\lim_{n \rightarrow \infty} \| c_n(x) -c(x)\|_{\cB(\C,\C^d)}&= 0,\\
\lim_{n \rightarrow \infty} \| V_n(x) -V(x)\|_{\cB(\C)}&= 0.
\endaligned
\end{equation*}
\item
\label{i: inv 2-cond after cut}
$\oA_n \in \cB_{\mu,M}(\Omega)$ for all $n \in \N_+$;
\item
\label{i: inv p-cond after cut}
If $\oA \in \cS_p(\Omega)$, then for all $n \in \N_+$,
$$
\aligned
\oA_n &\in \cS_p(\Omega),\\
\mu_p(\oA_n) &\geq \mu_p(\oA).
\endaligned
$$
Moreover, $\lim_{n \rightarrow \infty} \mu_p(\oA_n) = \mu_p(\oA_n)$.

In particular, if $\oA \in \cB_p(\Omega)$ then $\oA_n \in \cB_p(\Omega)$ for all $n \in \N_+$;
\item
\label{i : 1 supp in 0 supp}
${\rm supp} \, b_n, \, {\rm supp} \, c_n \subset {\rm supp} \, V_n$ for all $n \in \N_+$;
\item
\label{i: unif bound cut}
There exists $C>0$ such that $\max\{\|b_n\|_\infty,\|c_n\|_\infty,\|V_n\|_\infty\} \leq C$ for all $n \in \N_+$.
\end{enumerate}
\end{lemma}
\begin{proof}
Items \eqref{i : conv coeff cut} and \eqref{i: unif bound cut} follow by the definition \eqref{d : terms cut n}.

Clearly, for all $\xi \in \C^d$ and all $r \in (1,\infty)$,
\begin{equation*}
\aligned
\Gamma_{r}^{\oA_n} (\cdot, \xi) &= (1-\psi_n) \Re\sk{A \xi}{\cJ_r \xi} + \psi_n \Gamma_r^\oA(\cdot, \xi), \\
|b_n-c_n|^2 &\leq \psi_n M^2\, V_n.
\endaligned
\end{equation*}
By taking $r =2$ and by combining  \eqref{e : mu e delta} with the fact that $0 \leq \psi_n \leq 1$, we get  \eqref{i: inv 2-cond after cut}. 

Similarly, by taking $r=p$ and by definition of $\mu_p$, we obtain that $\oA_n \in \cS_p(\Omega)$ and that $\mu_p(\oA_n) \geq \mu_p(\oA)$, for all $n \in \N_+$. Since $\oA_n \in \cS_p(\Omega)$, we have $\Gamma_p^{\oA_n}(x, \xi) \geq \mu_p(\oA_n) (|\xi|^2+V_n(x))$ for almost everywhere $x \in \Omega$ and all $\xi \in \C^d$ and $n \in \N_+$. The limit in \eqref{i: inv p-cond after cut} now follows by \eqref{i : conv coeff cut}.

Item \eqref{i : 1 supp in 0 supp} follows by the definition \eqref{d : terms cut n} and Proposition~\ref{p : property of B}\eqref{i : V dom Rec}.
\end{proof}

\subsection{Bilinear embedding with bounded potentials}\label{s : b.e. bound pot}
Let $\Omega \subset \R^d$ open, $p >1$, $q=p/(p-1)$ and $\oV$ and $\oW$ be two closed subspaces of $H^{1}(\Omega)$ of the type described in Section~\ref{s: boundary}. We prove now the bilinear embedding for arbitrary $A,B \in \cA(\Omega)$, $b,c,\beta,\gamma \in L^\infty(\Omega,\C^d)$ and $V,W \in L^\infty(\Omega,\R_+)$ such that $(A,b,c,V), (B,\beta, \gamma, W) \in \cB_p(\Omega)$. We will do it in two steps.
\begin{itemize}
\item
First, we will prove the bilinear embedding for all  $\oC=(C,d,e,U), \oD=(D,l,m,Z) \in \cB_p(\Omega)$ satisfying the assumptions of Section~\ref{s : appr by conv}, that is, 
$$
\aligned
{\rm supp}\, d, \, {\rm supp}\, e &\subset {\rm supp} \,U,\\
{\rm supp}\, l, \, {\rm supp}\, m &\subset {\rm supp} \,Z,
\endaligned
$$ 
and $d,e,l,m,U,Z$ are compactly supported. Fix $f,g \in (L^p \cap L^q)(\Omega)$. Let $\oC_\varepsilon, \oD_\varepsilon$ be the smooth approximations introduced in Section~\ref{s : appr by conv}. By Lemma~\ref{l : conv pot and grad sem}, Lemma~\ref{l : inv after reg} and a standard theorem in measure theory, there exists a sequence $(\varepsilon_n)_{n\in\N}$ such that $\nabla T_t^{\oC_{\varepsilon_n}}f \rightarrow \nabla T_t^{\oC}f$ and $U_{\varepsilon_n}^{1/2} T_t^{\oC_{\varepsilon_n}}f \rightarrow U^{1/2} T_t^{\oC}f$ as $n \rightarrow \infty$ almost everywhere on $\Omega$ and for all $t>0$; the same for $(\oD,g)$. Recall that we have already established in Section~\ref{s : proof bil emb} the bilinear embedding for smooth matrices and smooth and compactly supported first- and zero-order coefficients. Consequently, Fatou's lemma and Lemma~\ref{l : inv after reg}\eqref{i: inv 2-cond after reg},\eqref{i: inv p-cond after reg} give
$$
\aligned
\displaystyle\int^{\infty}_{0}\!&\int_{\Omega}\sqrt{\mod{\nabla T^{\oC}_{t}f}^2 + U \mod{T^{\oC}_{t}f}^2}\sqrt{\mod{\nabla T^{\oD}_{t}g}^2 +  Z \mod{T^{\oD}_{t}g}^2}  \\
&\leq \liminf_{n\rightarrow \infty} \displaystyle\int^{\infty}_{0}\!\int_{\Omega}\sqrt{\mod{\nabla T^{\oC_{\varepsilon_n}}_{t}f}^2 + U_{\varepsilon_n} \mod{T^{\oC_{\varepsilon_n}}_{t}f}^2}\sqrt{\mod{\nabla T^{\oD_{\varepsilon_n}}_{t}g}^2 +  Z_{\varepsilon_n} \mod{T^{\oD_{\varepsilon_n}}_{t}g}^2}\\
   &\leq \liminf_{n\rightarrow \infty}C_n \norm{f}{p}\norm{g}{q}, 
\endaligned
$$
where $C_n$ continuously depends on $p,\Lambda,\mu,M,\mu_p(\oC_{\varepsilon_n}),\mu_q(\oD_{\varepsilon_n})$. Therefore, \\Lemma~\ref{l : inv after reg}\eqref{i: inv p-cond after reg} implies that
\begin{equation}
\label{eq: N bil eps}
\aligned
\displaystyle\int^{\infty}_{0}\!\int_{\Omega}\sqrt{\mod{\nabla T^{\oC}_{t}f}^2 + U \mod{T^{\oC}_{t}f}^2}\sqrt{\mod{\nabla T^{\oD}_{t}g}^2 +  Z \mod{T^{\oD}_{t}g}^2} \leq C \norm{f}{p}\norm{g}{q},
\endaligned
\end{equation}
where $C$ continuously depends on $p,\Lambda,\mu,M,\mu_p(\oC),\mu_q(\oD)$.
%
\item
Let $\oA=(A,b,c,V), \oB=(B,\beta,\gamma,W)$ as above. Fix $f,g \in (L^p \cap L^q)(\Omega)$. Let $\oA_n, \oB_n$ as in \eqref{d : terms cut n}. By arguing just as in the previous case, only replacing Lemma~\ref{l : inv after reg} by Lemma~\ref{l : inv after cut}, we deduce that there exists a sequence $(n_j)_{j \in \N}$ such that $\nabla T_t^{\oA_{n_j}}f \rightarrow \nabla T_t^{\oA}f$ and $V_{n_j}^{1/2} T_t^{\oA_{n_j}}f \rightarrow V^{1/2} T_t^{\oA}f$ as $j \rightarrow \infty$ almost everywhere on $\Omega$ and for all $t>0$; the same for $(\oB,g)$. Consequently, Fatou's lemma, \eqref{eq: N bil eps} applied with $\oC=\oA_{n_j}$ and $\oD=\oB_{n_j}$, and Lemma~\ref{l : inv after cut} give
$$
\aligned
\displaystyle\int^{\infty}_{0}\!&\int_{\Omega}\sqrt{\mod{\nabla T^{\oA}_{t}f}^2 + V \mod{T^{\oA}_{t}f}^2}\sqrt{\mod{\nabla T^{\oB}_{t}g}^2 +  W \mod{T^{\oB}_{t}g}^2}  \\
&\leq \liminf_{n\rightarrow \infty} \displaystyle\int^{\infty}_{0}\!\int_{\Omega}\sqrt{\mod{\nabla T^{\oA_{n_j}}_{t}f}^2 + V_{n_j} \mod{T^{\oA_{n_j}}_{t}f}^2}\sqrt{\mod{\nabla T^{\oB_{n_j}}_{t}g}^2 +  W_{n_j} \mod{T^{\oB_{n_j}}_{t}g}^2}\\
   &\leq C \norm{f}{p}\norm{g}{q}.
\endaligned
$$
%
%
\end{itemize}

\section{Interior elliptic regularity}
In Section~\ref{s: Neumann introduction} we showed that we can associate a contractive and analytic semigroup $(T_t)_{t>0}$ on $L^2(\Omega)$ with the sesquilinear form defined in \eqref{e : form sesq}, provided that its coefficients satisfy conditions \eqref{eq: sect form below} and \eqref{eq: sect form above}. In this section we will see that if we assume that the coefficients are also smooth, then $T_t f$ is also smooth for all $t>0$ and $f \in L^2(\Omega)$.
\medskip

 For every function $f$ on $\R^d$ and $y \in \R^d \setminus\{0\}$, we introduce the functions $\tau_{y}f$, $\delta_{y}f$ on $\R^d$ by the  following rule:
$$
\aligned
\tau_{y} f(x) &= f(x-y), \\
\delta_{y}f(x) &= \frac{\tau_{-y}f(x) - f(x)}{|y|},
\endaligned
$$
for every $x \in \R^d$. 

We will use the following lemmas to prove Lemma~\ref{l : weak sol in H2}. See, for example, \cite[Proposition~4.8]{Giaq} for their proofs.
\begin{lemma}
\label{l : l1}
Let $f \in H^{1}_{\rm loc}(\R^d)$, $\nu \in \BS^{d-1}$, $R>0$ and $x_0 \in \R^d$. Then
$$
\int_{B(x_0,R)} |\delta_{h\nu} f|^2 \leq \int_{B(x_0, R+h)} |\partial_\nu f|^2, \quad \forall h >0.
$$
\end{lemma}
\begin{lemma}
\label{l : l2}
Let $f \in L^2_{\rm loc}(\R^d)$. Then $f \in H^1_{\rm loc}(\R^d)$ if and only if 
$$
\limsup_{h \rightarrow 0} \| \delta_{h\nu} f \|_{L^2_{\rm loc}} < \infty, \quad \forall \nu \in \BS^{d-1}.
$$
In this case,
$$
\|\partial_\nu f \|_{L^2(B)} \leq \limsup_{h \rightarrow 0} \|\delta_{h\nu}f\|_ {L^2(B)},
$$ 
for every compact set $B \subset \R^d$.
\end{lemma}

Let $\Omega \subseteq \R^d$ open, $A$ a complex $d \times d$ matrix-valued function on $\Omega$, $b,c$ two complex $d$-dimensional vector-valued functions on $\Omega$  and $V$ a nonnegative function on $\Omega$. Set $\oA=(A,b,c,V)$. We will write that
$$
\gota_\oA(u,v) = \displaystyle \int_{\Omega}\left[\sk{A\nabla u}{\nabla v}_{\C^{d}} + \sk{\nabla u}{b}_{\C^d}\overline{v} + u\sk{c}{\nabla v}_{\C^d} + V u \overline{v}\right]
$$
for every $u, v \in H^1_{\rm loc}(\Omega)$ for which the integral on the right-hand side is finite. Note that now we are not assuming that the domain of the form is of the type described in Section~\ref{s: boundary}.
\begin{lemma}
\label{l : weak sol in H22}
Let $\Omega \subset \R^d$ open, $A \in  L^\infty(\Omega, \C^{d \times d})$, $b,c \in L^\infty(\Omega, \C^{d})$ and \\ $V \in  L^\infty(\Omega, \C)$. Take $u \in H^1_{\rm loc}(\Omega)$. Suppose that there exists $f \in L^2_{\rm loc}(\Omega)$ such that
\begin{equation}
\label{eq: weak sol 2}
\gota_\oA(u, \varphi) = \int_\Omega f \, \overline{\varphi}, \quad \forall \varphi \in C_c^\infty(\Omega).
\end{equation}
Then 
$$
\gota_\oA(u, \varphi) = \int_\Omega f \, \overline{\varphi}, \quad \forall \varphi \in H^1_c(\Omega),
$$
where $H^1_c(\Omega)$ consists of all functions in $H^1(\Omega)$ having compact support in $\Omega$.
\end{lemma}
\begin{proof}
Let $u \in H^1_{\rm loc}(\Omega)$ and $f \in L^2_{\rm loc}(\Omega)$ for which \eqref{eq: weak sol 2} holds. Let $\varphi \in H^1(\Omega)$ with compact support in $\Omega$ and $\Omega^\prime \Subset \Omega$ such that ${ \rm supp}\, \varphi \Subset \Omega^\prime$. By \cite[Lemma~2.18(b) and Lemma~3.15]{AdFour} there exists a sequence $(\varphi_n)_{n \in \N}$ in $C_c^\infty(\Omega^\prime)$ such that 
\begin{equation}
\label{eq : apx approx}
\varphi_n \rightarrow \varphi \,\, \text{ in } H^1(\Omega^\prime).
\end{equation}
Therefore, by using also the assumption on $u$ and the boundedness of $A,b,c$ and $V$ (thanks to which we have, for instance, that $Vu \in L^2(\Omega^\prime)$), we get
\begin{equation}
\label{eq: apx first eq}
\aligned
\gota_\oA(u, \varphi) &= \displaystyle \int_{\Omega^\prime}\left[\sk{A\nabla u}{\nabla \varphi}_{\C^{d}} + \sk{\nabla u}{b}_{\C^d}\overline{\varphi} + u\sk{c}{\nabla \varphi}_{\C^d} + V u \overline{\varphi}\right]\\
&= \lim_{n \rightarrow \infty} \displaystyle \int_{\Omega^\prime}\left[\sk{A\nabla u}{\nabla \varphi_n}_{\C^{d}} + \sk{\nabla u}{b}_{\C^d}\overline{\varphi_n} + u\sk{c}{\nabla \varphi_n}_{\C^d} + V u \overline{\varphi_n}\right]\\
& = \lim_{n \rightarrow \infty} \gota_\oA(u, \varphi_n).
\endaligned
\end{equation}
Since $\varphi_n \in C^\infty_c(\Omega^\prime)$ for all $n \in \N$ and $\Omega^\prime \Subset \Omega$, by \eqref{eq: weak sol 2} we have
$$
 \gota_\oA(u, \varphi_n) =  \int_\Omega f \, \overline{\varphi_n} = \int_{\Omega^\prime} f \, \overline{\varphi_n},
$$
for all $n \in \N$.
On the other hand, $f \in L^2(\Omega^\prime)$. Hence by \eqref{eq : apx approx} we obtain
\begin{equation}
\label{eq: apx second eq}
\hskip-37,5 pt \lim_{n \rightarrow \infty} \gota_\oA(u, \varphi_n) = \int_{\Omega^\prime} f \, \overline{\varphi} = \int_\Omega f \, \overline{\varphi}.
\end{equation}
We conclude by combining \eqref{eq: apx first eq} and \eqref{eq: apx second eq}.
\end{proof}

The following lemma was proved in \cite[Theorem~4.9]{Giaq} when $b,c,V=0$. We adapt its proof to the general case when $b,c$ and $V$ are not null.
\begin{lemma}
\label{l : weak sol in H2}
Let $\Omega \subset \R^d$ open, $A \in \left({\rm Lip}_{\rm loc} \cap L^\infty\right)(\Omega, \C^{d \times d})$, $b,c \in \left({\rm Lip}_{\rm loc} \cap L^\infty\right)(\Omega, \C^{d})$ and $V \in \left({\rm Lip}_{\rm loc} \cap L^\infty\right)(\Omega, \C)$ such that $\oA=(A,b,c,\Re V) \in \cB(\Omega)$. Take $u \in H^1_{\rm loc}(\Omega)$. Suppose that there exists $f \in L^2_{\rm loc}(\Omega)$ such that
\begin{equation}
\label{eq: weak sol}
\gota_\oA(u, \varphi) = \int_\Omega f \, \overline{\varphi}, \quad \forall \varphi \in C_c^\infty(\Omega).
\end{equation}
Then $u \in H^2_{\rm loc}(\Omega)$.
\end{lemma}
\begin{proof}
Fix $j \in \{1, \dots, d\}$ and set 
$$
\aligned
\tau_h &=\tau_{he_j}, \\
\delta_h  &= \delta_{he_j}, \\
\oA_h&= (\tau_{-h}A,\tau_{-h}b,\tau_{-h}c,\tau_{-h}V) = \tau_{-h}\oA, \\
\delta_h(\oA) &=( \delta_h A, \delta_h b, \delta_h c, \delta_h V),
\endaligned
$$
where $e_j \in \R^d$ is the unit vector $(0,\dots,0,1,0,\dots,0)$ with $1$ in the $j$-th position.

Let $u \in H^1_{\rm loc}(\Omega)$ and $f \in L^2_{\rm loc}(\Omega)$ for which \eqref{eq: weak sol} holds.  Let $\varphi \in H^1(\Omega)$ with compact support in $\Omega$. If $h$ is small enough, $\tau_h \varphi$ also belongs to $H^1(\Omega)$ and has compact support in $\Omega$. Therefore, Lemma~\ref{l : weak sol in H22} implies that 
\begin{equation}
\label{eq: eq C1}
\gota_\oA(u, \tau_h \varphi) = \int_{\Omega} f \, \tau_h \overline{\varphi},
\end{equation}
for $|h|\ll 1$.

On the other hand, by $\nabla\tau_h= \tau_h\nabla$, a change of variable and \eqref{eq: eq C1}, we find that
\begin{equation}
\label{eq: eq C2}
\gota_{ \oA_h}(\tau_{-h} u, \varphi) =\gota_\oA(u, \tau_h \varphi) =\int_{\Omega} f \, \tau_h \overline{\varphi}.
\end{equation}
Subtracting \eqref{eq: weak sol} from equation \eqref{eq: eq C2}, we get
$$
\gota_{\oA_h}( \tau_{-h}u, \varphi)- \gota_\oA(u,  \varphi)= \int_\Omega f (\tau_h \overline{\varphi}- \overline{\varphi}).
$$
Observe that $\gota_{\oA_h}( \tau_{-h}u, \varphi)$ is well-defined for small $|h|$, even though $\oA$ is defined on $\Omega$. This is due to the fact that $\varphi$ has compact support. 

By writing $\tau_{-h}u=\tau_{-h} u -u+u$ and by dividing by $h$ both terms of the previous identity, we have
\begin{equation}
\label{eq: eq C3}
\gota_{\oA_h}(\delta_h u, \varphi)= - \gota_{\delta_h(\oA)}(u,\varphi) + \int_\Omega f \, \overline{\delta_{-h} \varphi}.
\end{equation}
Now we will choose an adequate $\varphi$. Let $x_0 \in \Omega$ and $R>0$ such that $\overline{B_{2R}}:=\overline{B(x_0, 2R)} \subset \Omega$. Fix $\eta \in C_c^\infty(\Omega)$ such that $0 \leq \eta \leq 1$, ${\rm supp }\,\eta \subseteq B_R$ and $\eta =1$ on $B_{R/2}$.
Define
\begin{equation*}
\aligned
v &:= \delta_h u, \\
\varphi &:= \eta^2 v
\endaligned
\end{equation*}
for $|h| \ll 1$.
By plugging this particular $\varphi$ in \eqref{eq: eq C3}, we obtain
\begin{equation}
\label{eq: eq C4}
\gota_{\oA_h}(v, \eta^2 v) = - \gota_{\delta_h(\oA)}(u,\eta^2 v) + \int_\Omega f \, \overline{\delta_{-h} (\eta^2 v)}.
\end{equation}
By writing $\nabla(\eta^2 v) = \eta^2 \nabla v  + 2 \eta v \nabla\eta$ in the definition of $\gota_{\oA_h}$ and $\gota_{\delta_h(\oA)}$, \eqref{eq: eq C4} yields
$$
\aligned
\displaystyle \int_{B_R}\eta^2&\biggl[\sk{\tau_{-h}A\nabla v}{\nabla v}+\sk{\nabla v}{\tau_{-h}b}\overline{v} + v\sk{\tau_{-h}c}{\nabla v} + \tau_{-h}V|v|^2\biggr] \\
=& -2 \displaystyle \int_{B_R} \eta \left[ \sk{\tau_{-h}A\nabla v}{\nabla \eta}\overline{v}  + |v|^2\sk{\tau_{-h}c}{\nabla \eta} \right]  \\
&- \displaystyle \int_{B_R}\eta^2\biggl[\sk{\delta_h(A)\nabla u}{\nabla v}+\sk{\nabla u}{\delta_h(b)}\overline{v} + u\sk{\delta_h(c)}{\nabla v} + \delta_h(V)u \overline{v}\biggr] \\
&- 2 \displaystyle \int_{B_R} \eta\overline{v} \left[ \sk{\delta_h(A)\nabla u}{\nabla \eta}  +u\sk{\delta_h(c)}{\nabla \eta}\right]\\
&+ \int_{B_{R+h}} f \, \overline{\delta_{-h} (\eta^2 v)}.
\endaligned
$$
Here we also used that ${\rm supp}\, \eta \subseteq B_R$.
Therefore, since $\oA \in \cB(\Omega)$, there exists $\mu \in (0,1)$ such that
$$
\aligned
\mu \int_{B_R} \eta^2 |\nabla v|^2 \leq& -2 \, \Re\displaystyle \int_{B_R} \eta \left[ \sk{\tau_{-h}A\nabla v}{\nabla \eta}\overline{v}  + |v|^2\sk{\tau_{-h}c}{\nabla \eta} \right]  \\
& -\Re \displaystyle \int_{B_R}\eta^2\biggl[\sk{\delta_h(A)\nabla u}{\nabla v}+\sk{\nabla u}{\delta_h(b)}\overline{v} + u\sk{\delta_h(c)}{\nabla v} + \delta_h(V)u \overline{v}\biggr] \\
&- 2 \, \Re\displaystyle \int_{B_R} \eta\overline{v} \left[ \sk{\delta_h(A)\nabla u}{\nabla \eta}  +u\sk{\delta_h(c)}{\nabla \eta} \right] \\
&+\Re \int_{B_{R+h}} f \, \overline{\delta_{-h} (\eta^2 v)}.
\endaligned
$$
The boundedness of $A$ and $c$ implies that $\tau_{-h}A$ and $\tau_{-h}c$ are bounded on $B_R$ uniformly in $|h|\ll 1$. Moreover, as $A$ is locally Lipschitz, $\delta_h(A)$ is also bounded on $B_R$ uniformly in $|h|\ll 1$. The same for $\delta_h(b),\delta_h(c)$ and $\delta_h(V)$. Therefore, by using also that  $\|\nabla \eta\|_\infty \leqsim_R \, 1$ and $0 \leq \eta \leq1$, we get 
\begin{equation}
\label{eq: eq CC}
\aligned
\int_{B_R} \eta^2 |\nabla v|^2 \leqsim&_{\mu,R} \int_{B_R} \left(\eta |\nabla v| |v| +  \eta |v|^2\right) \\
&+ \int_{B_R} \eta^2 \left(|\nabla u| +  | u|\right) |\nabla v| \\
&+ \int_{B_R} \eta \left(|\nabla u|+ | u| \right) | v| \\
& + \int_{B_{R+h}} |f| |\delta_{-h} (\eta^2 v)|\\
=:& \, I_1 +I_2 +I_3 +I_4.
\endaligned
\end{equation}
By repeatedly applying the inequality $2 ab \leq \varepsilon a + \varepsilon^{-1}b$ for all $a,b \in \R$, $\varepsilon >0$ and the fact that $0 \leq \eta \leq1$, we obtain the following estimates
$$
\aligned
I_1 \,&\leqsim \, \varepsilon \int_{B_R} \eta^2 |\nabla v|^2 + \varepsilon^{-1}  \int_{B_R} |v|^2 + \int_{B_R} |v|^2,\\
I_2 \,&\leqsim \, \varepsilon \int_{B_R} \eta^2 |\nabla v|^2 + \varepsilon^{-1} \int_{B_R} \left(|\nabla u|^2 + |u|^2 \right), \\
I_3\, &\leqsim  \,\int_{B_R} \left(|\nabla u|^2 + |u|^2  +|v|^2\right), \\
I_4 \,& \leqsim \,\varepsilon \int_{B_{R+h}}  |\delta_{-h}(\eta^2 v)|^2 + \varepsilon^{-1} \int_{B_{R+h}} |f|^2,
\endaligned
$$
for all $\varepsilon >0$. These estimates combined with \eqref{eq: eq CC} give
\begin{equation}
\label{eq: eq CCC}
\aligned
\int_{B_R} \eta^2 |\nabla v|^2 \leqsim& \, \varepsilon \int_{B_R} \eta^2 |\nabla v|^2 \\
& + \varepsilon^{-1} \int_{B_R} \left(|\nabla u|^2 + |u|^2 + |v|^2 \right) \\
& + \int_{B_R} \left(|\nabla u|^2 + |u|^2 + |v|^2 \right)  \\
& + \varepsilon \int_{B_{R+h}}  |\delta_{-h}(\eta^2 v)|^2 + \varepsilon^{-1} \int_{B_{R+h}} |f|^2,
\endaligned
\end{equation}
where the constant depends on $\mu$ and $R$, but not on $h$. On the other hand, by Lemma \ref{l : l1},
\begin{equation}
\label{eq: eq C6}
\aligned
\int_{B_{R+h}} |\delta_{-h}(\eta^2 v)|^2& \leq \int_{B_{R+2h}} |\nabla (\eta^2 v)|^2 \leqsim  \int_{B_{R+2h}} \eta^4 |\nabla  v|^2 + \eta^2 |v|^2 |\nabla \eta|^2 \\
&\leqsim \int_{B_R} \eta^2 |\nabla v|^2 + \int_{B_R} |v|^2,
\endaligned
\end{equation}
where in the last inequality we used that $\eta \leq 1$ and ${\rm supp}\, \eta \subseteq B_R$.
Hence \eqref{eq: eq CCC} and \eqref{eq: eq C6} give
$$
\aligned
\int_{B_R} \eta^2 |\nabla v|^2 \leqsim& \, \varepsilon \int_{B_R} \eta^2 |\nabla v|^2 \\
& + (\varepsilon^{-1} +1) \int_{B_R} \left(|\nabla u|^2 + |u|^2 + |v|^2 \right) \\
& + \varepsilon \int_{B_R} |v|^2 + \varepsilon^{-1} \int_{B_{R+h}} |f|^2.
\endaligned
$$
where the constant depends on $\mu$ and $R$, but not on $h$. Therefore, we may choose $\varepsilon$  independent of $h$ small enough such that the term with $\|\eta \nabla v\|_{L^2(B_R)}^2$ can be absorbed in the left-hand side of the inequality, obtaining 
\begin{equation*}
\aligned
\int_{B_R} \eta^2 |\nabla v|^2 \leqsim& \int_{B_R}  |\nabla u|^2 + \int_{B_R}  |u|^2 + \int_{B_R}  |v|^2 + \int_{B_{R+h}} |f|^2,
\endaligned
\end{equation*}
which implies that
\begin{equation}
\label{eq: eq C7}
\aligned
\int_{B_{R/2}}|\nabla v|^2 \leqsim\int_{B_R}  |\nabla u|^2 + \int_{B_R}  |u|^2 + \int_{B_R}  |v|^2 + \int_{B_{R+h}} |f|^2.
\endaligned
\end{equation}
By Lemma~\ref{l : l1} again,
$$
\int_{B_R}|v|^2 = \int_{B_R} |\delta_hu|^2 \leq \int_{B_{R+h}}|\nabla u|^2.
$$
Hence \eqref{eq: eq C7} yields
$$
\aligned
\int_{B_{R/2}}|\delta_h \nabla u|^2=\int_{B_{R/2}}|\nabla v|^2 \, \leqsim\int_{B_{R+h}}  |\nabla u|^2 + \int_{B_R}  |u|^2  + \int_{B_{R+h}} |f|^2.
\endaligned
$$
Since the constant in the previous estimate does not depend on $h$, we get
$$
\limsup_{h \rightarrow 0} \int_{B_{R/2}}|\delta_h \nabla u|^2 \leqsim \int_{B_R}  |\nabla u|^2 + \int_{B_R}  |u|^2  + \int_{B_R} |f|^2.
$$
From Lemma~\ref{l : l2} we infer that $u \in H^2(B_{R/2})$ and that
\begin{alignat*}{2}
\int_{B_{R/2}} |D^2 u|^2 \leqsim  \int_{B_R}  |\nabla u|^2 + \int_{B_R}  |u|^2  + \int_{B_R} |f|^2.
\tag*{\qedhere}
\end{alignat*}
\end{proof}
\begin{lemma}
Let $\Omega, A, b, c, V, \oV$ satisfy the standard assumptions of Section \ref{s : stand ass}. Suppose that $A \in C^\infty(\Omega, \C^{d \times d})$, $b,c \in C^\infty(\Omega, \C^d)$, $V \in C^\infty(\Omega, \R_+)\cap L^\infty(\Omega, \R_+)$ and that $\oA=(A,b,c,V) \in \cB(\Omega)$. Let $\oL^\oA=\oL^{\oA,\oV}$ be the operator on $L^2(\Omega)$ associated with the sesquilinear form $\gota_{\oA,\oV}$ defined in \eqref{e : form sesq}. Then
\begin{equation}
\label{eq: dom in H}
\Dom\left((\oL^\oA)^k\right) \subseteq H^{2k}_{\rm loc}(\Omega)
\end{equation}
for all $k \geq 1$. In particular, 
\begin{equation}
\label{eq: dom in smooth}
\bigcap_{k \geq 1} \Dom\left((\oL^\oA)^k\right) \subseteq C^\infty(\Omega).
\end{equation}
\end{lemma}
\begin{proof}
When $k=1$, \eqref{eq: dom in H} holds by Lemma~\ref{l : weak sol in H2}.

Assume that \eqref{eq: dom in H} holds for $k \in \N_+$. Then, for all $u \in \Dom\left((\oL^\oA)^{k+1}\right)$ we have 
$$
\aligned
u &\in \Dom\left((\oL^\oA)^k\right) \subseteq H^{2k}_{\rm loc}(\Omega), \\
\oL^\oA u &\in D\left((\oL^\oA)^k\right) \subseteq H^{2k}_{\rm loc}(\Omega).
\endaligned
$$
We will first prove that $u \in H^{2k+1}_{\rm loc}(\Omega)$ and then that $u \in H^{2k+2}_{\rm loc}(\Omega)$, thus establishing \eqref{eq: dom in H} for all $k \in \N$.

Let $|\alpha| \leq 2k-1$. For all $\varphi \in C^\infty_c(\Omega)$ we have 
$$
\aligned
\sk{\partial^\alpha \oL^\oA u}{\varphi}=&\,(-1)^{|\alpha|} \sk{\oL^\oA u}{\partial^\alpha \varphi} \\
=&\, (-1)^{|\alpha|} \int_\Omega\left[ \sk{A\nabla u}{ \partial^\alpha\nabla \varphi}+\sk{\nabla u}{b}\overline{\partial^\alpha \varphi} + u\sk{c}{ \partial^\alpha \nabla\varphi} + V u \overline{\partial^\alpha \varphi} \right].
\endaligned
$$
By assumptions $A,b,c,V$ are smooth. So, by integrating by parts and by applying the product rule, we get
$$
\aligned
\sk{\partial^\alpha \oL^\oA u}{\varphi}=&  \int_\Omega \sk{A \nabla \partial^\alpha u}{\nabla \varphi} - \sum_{i,j=1}^d \int_\Omega \partial_i \left(\sum_{0 \leq \beta < \alpha} \binom{\alpha}{\beta} \left(\partial^{\alpha-\beta}a_{ij}\right) \left(\partial^{\beta}\partial_j u\right) \right)\overline{\varphi} \\
&+ \int_\Omega \sk{\nabla \partial ^\alpha u}{b} \overline{\varphi} + \sum_{i=1}^d \int_\Omega \left( \sum_{0 \leq \beta < \alpha} \binom{\alpha}{\beta} \left(\partial^{\alpha-\beta} \overline{b_i}\right) \left(\partial^{\beta} \partial_i u\right) \right)\overline{ \varphi} \\
&+  \int_\Omega \partial^\alpha u \sk{ c}{\nabla \varphi} - \sum_{i=1}^d \int_\Omega \partial_i \left(\sum_{0 \leq \beta < \alpha} \binom{\alpha}{\beta}  \left(\partial^{\alpha-\beta}c_{i}\right) \left(\partial^{\beta} u\right) \right)\overline{\varphi} \\
&+ \int_\Omega V \partial^\alpha u \overline{\varphi} + \int_\Omega \left( \sum_{0 \leq \beta < \alpha} \binom{\alpha}{\beta} \left(\partial^{\alpha-\beta} V \right) \left(\partial^{\beta} u\right) \right)\overline{ \varphi} \\
=&  \int_\Omega \left[ \sk{A \nabla \partial^\alpha u}{\nabla \varphi} + \int_\Omega \sk{\nabla \partial ^\alpha u}{b} \overline{\varphi} +  \int_\Omega \partial^\alpha u \sk{ c}{\nabla \varphi} + \int_\Omega V \partial^\alpha u \overline{\varphi} \right] \\
& + \int_\Omega P(\partial)u \overline{\varphi},
\endaligned
$$
where $P$ is a polynomial with smooth coefficients and ${\rm deg}P \leq 2k$. Therefore, by using that $u \in H^{2k}_{\rm loc}(\Omega)$, we obtain that $P(\partial) u \in L^2_{\rm loc}(\Omega)$ and $\partial^\alpha u \in  H^1_{\rm loc}(\Omega)$. Hence we may write
$$
\gota_\oA(\partial^\alpha u,\varphi) = \int_\Omega  \left( \partial^\alpha \oL^\oA u -P(\partial)u \right)\overline{\varphi}.
$$
On the other hand, $\partial^\alpha \oL^\oA u \in H^1_{\rm loc}(\Omega)$, since $\oL^\oA u \in  H^{2k}_{\rm loc}(\Omega)$. 
This shows that $\partial^\alpha \oL^\oA u \in L^2_{\rm loc}(\Omega)$ and thus $ \partial^\alpha\oL^\oA u - P(\partial)u \in L^2_{\rm loc}(\Omega)$.
Now we may apply Lemma~\ref{l : weak sol in H2} which gives $\partial^\alpha u \in H^{2}_{\rm loc}(\Omega)$ for all $|\alpha| \leq 2k-1$. Thus we proved $u \in H^{2k+1}_{\rm loc}(\Omega)$.

 Let now $|\beta| \leq 2k$. As before, we get
$$
\aligned
\hskip -40pt \int_\Omega \sk{A \nabla \partial^\beta u}{\nabla \varphi} &+ \int_\Omega \sk{\nabla \partial ^\beta u}{b} \overline{\varphi} +  \int_\Omega \partial^\beta u \sk{ c}{\nabla \varphi} + \int_\Omega V \partial^\beta u \overline{\varphi} \\
&=  \int_\Omega \left( \partial^\beta \oL^\oA u -Q(\partial)u \right)\overline{\varphi}, \quad \forall \varphi \in C^\infty_c(\Omega),
\endaligned
$$
where $Q$ is a polynomial with smooth coefficients and ${\rm deg} Q \leq 2k+1$. We proved before that $u \in H^{2k+1}_{\rm loc}(\Omega)$, hence we have $\partial^\beta u \in H^{1}_{\rm loc}(\Omega)$ and $Q(\partial)u \in L^2_{\rm loc}(\Omega)$. Therefore, repeating the same procedure as before, we obtain that $u \in H^{2k+2}_{\rm loc}(\Omega)$. By induction we conclude the proof of \eqref{eq: dom in H}.

Prove now \eqref{eq: dom in smooth}. Consider $R>0$ and $x_0 \in \Omega$ such that $\overline{B_R}=\overline{B(x_0,R)} \subseteq \Omega$. Fix $\eta \in C_c^\infty(\Omega)$ such that $0 \leq \eta \leq 1$, ${\rm supp} \, \eta \subseteq B_{R}$ and $\eta =1$ on $B_{R/2}$. Let $ u \in \bigcap_{k \geq 1} \Dom\left((\oL^\oA)^k\right)$. By \eqref{eq: dom in H}, $u \in H_{\rm loc}^{2k}(\Omega)$, for all $k \geq 1$. Therefore, $\eta u \in H^{2k}(\R^d)$, for all $k \geq 1$. By the Sobolev embedding theorem, for example see \cite[Theorem~4.12, Part~I, Case~A]{AdFour}, $\eta u \in C^\infty(\R^d)$, hence $u \in C^\infty(B(x_0,R/2))$. Since $x_0 \in \Omega$ was arbitrary, we conclude that $u \in C^\infty(\Omega)$.
\end{proof}
\begin{lemma}
\label{l : Ttf smooth}
Let $\Omega, A, b, c, V, \oV$ satisfy the standard assumptions of Section \ref{s : stand ass}. Suppose that $A \in C^\infty(\Omega, \C^{d \times d})$, $b,c \in C^\infty(\Omega, \C^d)$, $V \in C^\infty(\Omega, \R_+)\cap L^\infty(\Omega, \R_+)$ and that $\oA=(A,b,c,V) \in \cB(\Omega)$. Let $(T_t^\oA)_{t>0}$ be the semigroup on $L^2(\Omega)$ associated with the sesquilinear form $\gota_{\oA,\oV}$. Then  $T_t^{\oA}f \in C^\infty(\Omega)$ for all $t>0$ and all $f \in L^2(\Omega)$.
\end{lemma}
\begin{proof}
Since $\oA \in \cB(\Omega)$, we know that $( T_t^\oA)_{t>0}$ is analytic and is generated by  $-\oL^\oA$, the operator on $L^2(\Omega)$ associated with the sesquilinear form $\gota_{\oA,\oV}$; see Section~\ref{s: Neumann introduction}.

Set $\oL =\oL^A$ and $T_t= T_t^\oA$ for all $t>0$.  By induction we deduce that 
\begin{equation}
\label{eq : rgTt in DL}
T_t f \in \Dom\left( \oL^k\right), \qquad\forall t>0, \quad \forall f \in L^2(\Omega)
\end{equation} 
for all $k \geq 1$. In fact, for $k=1$ this follows from the analyticity of $(T_t)_{t>0}$ and \cite[Chapter II, Theorem~4.6(c)]{EN}. Assume that \eqref{eq : rgTt in DL} holds for $k \in \N_+$. We want to prove that $\oL^k T_t f \in \Dom(\oL)$ for all $t>0$ and all $f \in L^2(\Omega)$. Fix $t>0$, $f \in L^2(\Omega)$ and choose $\varepsilon \in (0,t)$. By using the fact the generator $\oL$ commutes with the semigroup $(T_t)_{t>0}$  on $\Dom(\oL)$ \cite[Chapter II, Lemma~1.3(ii)]{EN}, the inductive hypothesis on $T_\varepsilon $ and \eqref{eq : rgTt in DL} applied for $T_{t-\varepsilon}$ with $k=1$, we get
$$
\oL^k T_t f = \oL^k T_{t-\varepsilon}T_\varepsilon f =T_{t-\varepsilon} \oL^k T_\varepsilon f \in \Dom(\oL).
$$
We conclude by invoking \eqref{eq: dom in smooth}.
\end{proof}

\section{Unbounded potentials}
In order to treat the general case with unbounded potentials, we will follow the argument used by Carbonaro and Dragi\v{c}evi\'c in \cite[Section~3.4]{CD-Potentials} when they proved \cite[Theorem~1.4]{CD-Potentials}. Like in their case, Theorem \ref{t: N bil} will follow from the special case of bounded potentials already proved in Section~\ref{s : b.e. bound pot}, once we prove the following approximation result.

Let $V \in L^{1}_{\rm loc}(\Omega)$ be a nonnegative unbounded function. For each $m \in \N_+$ define
\begin{equation}
\label{e : Vm}
V_m :=   \min\{V,m\}.
\end{equation}
 We also set $V_{\infty}=V$.
\begin{lemma}
\label{l : unif sett cut pot}
For every $\oA =(A,b,c,V) \in \cB_{\mu,M}(\Omega)$, $\oA_m=(A,b,c,V_m)$ and $p \in (1,\infty)$ we have
\begin{enumerate}[{\rm (i)}]
\item
\label{i: inv 2-cond after cut pot}
For every $\varepsilon \in (0,\mu)$ there exists $m_\varepsilon \geq0$ such that $\oA_m \in \cB_{\mu-\varepsilon,M}$ for all $m \geq m_\varepsilon$;
\item
\label{i: inv p-cond after cut pot}
If $\oA \in \cS_p(\Omega)$, then for every $\varepsilon \in (0,\mu_p(\oA))$ there exists $m_{\varepsilon,p}\geq0$ such that
$$
\aligned
\oA_m &\in \cS_p(\Omega), \\
\mu_p(\oA_m) &\geq \mu_p(\oA)-\varepsilon,
\endaligned
$$
 for all $m \geq m_{\varepsilon,p}$. Moreover, $\lim_{m \rightarrow \infty} \mu_p(\oA_m)=\mu_p(\oA)$. 

In particular, if $\cA \in \cB_p(\Omega)$ then $\cA_m \in \cB_p(\Omega)$ for all $m \geq \max\{m_{\varepsilon,p},m_{\varepsilon,q}\}$, where $q=p/(p-1)$.
\end{enumerate}
\begin{proof}
Let $r \in (1,\infty)$ and $\oA \in \cS_r(\Omega)$. Clearly, for almost all $x \in \Omega$,
$$
|b(x)-c(x)| \leq  M \sqrt{m}, \quad \forall m \geq \|b-c\|^2_\infty/M^2.
$$
On the other hand, note that for almost $x \in \Omega$, 
$$
\aligned
\Gamma_r^{A,b,c,m}(x,\xi) \geq  \frac{r}{2}\Delta_r(A) |\xi|^2 +\underset{x \in \Omega}{{\rm ess}\inf} \, \underset{|\sigma| =1}{\min} \,\Re \sk{b(x)+\cJ_rc(x)}{\sigma} |\xi| +m,
\endaligned
$$
for all $\xi \in \C^d$ and all $m \in \N$. Therefore, it suffices to show that for every $\varepsilon \in (0, \mu_r(\oA))$ there exists $\widetilde{m}_{\varepsilon,r}$ such that
$$
  \frac{r}{2}\Delta_r(A) |\xi|^2 +\underset{x \in \Omega}{{\rm ess}\inf} \, \underset{|\sigma| =1}{\min} \,\Re \sk{b(x)+\cJ_rc(x)}{\sigma} |\xi| +m  \geq (\mu_r(\oA)-\varepsilon)(|\xi|^2+m),
$$
for all $\xi \in \C^d$ and all $m \geq \widetilde{m}_{\varepsilon,r}$. This property holds by taking
$$
\widetilde{m}_{\varepsilon,r} := \frac{\left(\underset{x \in \Omega}{{\rm ess}\inf} \, \underset{|\sigma| =1}{\min} \,\Re^2\sk{b(x)+\cJ_rc(x)}{\sigma}\right)^2}{4( (r/2)\Delta_r(A)-\mu_r(\oA)+\varepsilon)\cdot(1-\mu_r(\oA)+\varepsilon)},
$$
which is a nonnegative finite constant by the boundedness of $b,c$ and  \eqref{e : mu e delta}.
%
Therefore, by taking $r=2$, we prove \eqref{i: inv 2-cond after reg} with $m_\varepsilon := \max\{\widetilde{m}_{\varepsilon,2},\|b-c\|^2_\infty/M^2\}$. By taking $r=p$, we prove that $\cA_m \in \cS_p(\Omega)$ and that
\begin{equation}
\label{e : dis mum e mu}
\mu_p(\oA_m) \geq \mu_p(\oA)-\varepsilon,
\end{equation}
for all $\varepsilon \in (0,\mu_p(\oA))$ and all $m \geq m_{\varepsilon,p} := \max\{\widetilde{m}_{\varepsilon,p},\|b-c\|^2_\infty/M^2\}$. Since $\oA_m \in \cS_p(\Omega)$, we have $\Gamma_p^{\oA_m}(x, \xi) \geq \mu_p(\oA_m) (|\xi|^2+V_m(x))$ for almost everywhere $x \in \Omega$, all $\xi \in \C^d$, all $\varepsilon  \in (0,\mu_p(\oA))$ and all $m \geq m_{\varepsilon,p}$. Therefore, since $V_m(x) \rightarrow V(x)$ for almost all $x \in \Omega$, as $m \rightarrow \infty$, we get
\begin{equation*}
\mu_p(A) \geq \limsup_{m \rightarrow \infty} \mu_p(\oA_m),
\end{equation*} 
which, combined with \eqref{e : dis mum e mu}, implies that
$$
\varepsilon + \liminf_{m \rightarrow \infty} \mu_p(\oA_m) \geq \mu_p(\oA) \geq \limsup _{m \rightarrow \infty} \mu_p(\oA_m),
$$
for all $\varepsilon \in (0,\mu_p(\oA))$.  By sending $\varepsilon \rightarrow 0$, we conclude.
\end{proof}
\end{lemma}
\begin{lemma}
\label{l: conv sem unbound pot}
For all $f \in L^2(\Omega)$, $(A,b,c, V) \in \cB_{\mu,M}(\Omega)$ and all $t >0$ we have 
\begin{equation*}
\aligned
\nabla T_t^{\oA_m}f &\rightarrow \nabla T_t^{\oA} f \quad \quad \,\text{in } L^2(\Omega, \C^d), \\
 V_m^{1/2} T_t^{\oA_m}f &\rightarrow V^{1/2} T_t^{\oA} f \quad \text{in } L^2(\Omega)
\endaligned
\end{equation*}
as $m \rightarrow \infty$.
\end{lemma}
\begin{proof}
Fix $f \in L^2(\Omega)$. Denote by $\oL_m = \oL^{\oA_m}$ the operator associated with $\oA_m= (A,b,c,V_m)$ for all $m \in \N_+ \cup \{\infty\}$. By Lemma~\ref{l : unif sett cut pot}\eqref{i: inv 2-cond after cut pot}, there exists $\widetilde{\mu} \in(0,\mu)$ and $\widetilde{m} \geq 0$ such that each $(A,b,c,V_m) \in \cB_{\widetilde{\mu},M}(\Omega)$ for all $m \in \N_+ \cup\{\infty\}$ such that $m\geq \widetilde{m}$. Therefore, each $\oL_m$ is sectorial with sectoriality angle $\omega(\oL_m) \leq \widetilde{\vartheta_0}=\vartheta_0(\widetilde{\mu},M,\Lambda)$, for all $m \in \N_+ \cup\{\infty\}$ such that $m\geq \widetilde{m}$; see Section~\ref{s: Neumann introduction}. Using the standard representation of the analytic semigroup $(T_t^{\oA_m})_{t >0}$, $m \in \N_+ \cup \{\infty\}$, $m \geq \widetilde{m}$,  by means of a Cauchy integral, earlier applied in the proof of Lemma~\ref{l : conv pot and grad sem}, we get
$$
\aligned
\|\nabla T_t^{\oA_m} f -\nabla T_t^\oA f \| &\leqsim \int_\gamma e^{-t \Re\zeta} \| \nabla(\zeta - \oL_{m})^{-1} f  -\nabla(\zeta - \oL)^{-1} f \|_2 |\wrt \zeta|,\\
\|V_m^{1/2} T_t^{\oA_m} f -V^{1/2} T_t^\oA f \| &\leqsim \int_\gamma e^{-t \Re\zeta} \| V_m^{1/2}(\zeta - \oL_{m})^{-1} f  -V^{1/2}(\zeta - \oL)^{-1} f \|_2 |\wrt \zeta|,
\endaligned
$$
where $\gamma$ is the positively oriented boundary of $\bS_\vartheta \cup \{\zeta \in \C \, : \, |\zeta|< \delta \}$, with $\theta >\widetilde{\theta_0}$ and $\delta >0$. Therefore it suffices to prove that
\begin{equation}
\label{eq: unb conv sem}
\aligned
\nabla(\zeta - \oL_{m})^{-1} f &\rightarrow \nabla(\zeta - \oL)^{-1}f  \quad  \quad \, &\text{ in } L^2(\Omega,\C^d), \quad &\forall \zeta \in \C \setminus \bS_{\widetilde{\theta_0}},\\
V_m^{1/2}(\zeta - \oL_{m})^{-1} f &\rightarrow V^{1/2}(\zeta - \oL)^{-1}f  \quad &\text{ in } L^2(\Omega), \quad \quad\,\,\, &\forall \zeta \in \C \setminus \bS_{\widetilde{\theta_0}},
\endaligned
\end{equation}
as $m \rightarrow \infty$, and that for all $t>0$ there exist $C=C(t)>0$, not depend on $m$, and $F(t, \cdot) \in L^1(\gamma)$ such that
\begin{equation}
\label{eq: unb cdl est}
\aligned
e^{-t \Re\zeta} \| \nabla(\zeta - \oL_{m})^{-1} f  \|_2  &\leq C F(t,\zeta), \\
e^{-t\Re \zeta} \|V_m^{1/2}(\zeta - \oL_{m})^{-1} f \|_2 &\leq C F(t,\zeta),
\endaligned
\end{equation}
for all $m \in \N_+ \cup \{\infty\}$, $m \geq \widetilde{m}$, and all $\zeta \in \gamma$. In fact, we would complete the proof by applying the dominated convergence theorem in the two integrals above.

In order to prove \eqref{eq: unb conv sem} and \eqref{eq: unb cdl est} one can proceed as in the proof of \cite[Proposition~3.9]{CD-Potentials}, where $b$ and $c$ are zero. We leave it to the reader to fill in the details.
\end{proof}
By combining Lemma~\ref{l : unif sett cut pot} and Lemma~\ref{l: conv sem unbound pot} and by applying the same limit argument of Section~\ref{s : b.e. bound pot}, we obtain the bilinear embedding \eqref{eq: N bil} for unbounded potentials.

\subsection{Unbounded complex potentials}
In Section~\ref{s : compl pot} a bililinear embedding theorem was stated for particular complex potentials of the type $\varrho V$, with $V \in L^1_{\rm loc}(\Omega)$ nonnegative and $\varrho \in \C$ such that $\Re \varrho >0$.

If $V$ is bounded, then the approximation argument of Section~\ref{a: A} can be applied exactly in the same way.

Otherwise, let suppose that $V$ is unbounded. We can assume that $\varrho \in \C \setminus \R$, otherwise $\varrho V$ would be real. For every $m \in \N_+ \cup \{\infty\}$, let $V_m$ be defined as in \eqref{e : Vm} with the notation $V_\infty = V$. Set $\oA = (A,b,c,\varrho V)$ and $\oA_m =(A,b,c,\varrho V_m)$ for all $m \in \N_+$. Fix $f \in L^2(\Omega)$ and $t>0$. We would like to get the analogue of Lemma~\ref{l: conv sem unbound pot}, that is, 
\begin{equation*}
\aligned
\nabla T_t^{\oA_m}f &\rightarrow \nabla T_t^{\oA} f \quad \quad \,\text{in } L^2(\Omega, \C^d), \\
  V_m^{1/2} T_t^{\oA_m}f &\rightarrow  V^{1/2} T_t^{\oA} f \quad \text{in } L^2(\Omega),
\endaligned
\end{equation*}
as $m \rightarrow \infty$, provided that $(A, b,c, \Re(\varrho) V) \in \cB(\Omega)$. Again, we would like to proceed like in \cite[Section~3.4]{CD-Potentials} where the potentials are real. For the kind of potentials we are considering now, it is straightforward to adapt their method to our case. The main step where the differences come out lies in proving the analogue of \cite[Lemma~3.8]{CD-Potentials}. For real potentials, they did that by combining a monotone convergence theorem for sequences of symmetric sesquilinear forms \cite[Theorem 3.13a, p. 461]{Kat} and a vector-valued version of Vitali's theorem \cite[Theorem~A.5]{ABHN}.
We will do the same following the sketch of its proof and highlighting what we need to change in our case, that is, for complex potentials.
More precisely, we would like to prove that, for all $f \in L^2(\Omega)$ and all $s>0$,
\begin{equation}
\label{e: Vitali conv I}
(s + \oL^{\oA_m})^{-1} f \rightarrow (s + \oL^{\oA})^{-1} f \quad \text{in } L^2(\Omega), \quad \text{as } m \rightarrow \infty.
\end{equation}
Since $(A,b,c, \Re(\varrho) V) \in \cB(\Omega)$, by arguing as in Lemma~\ref{l : unif sett cut pot}, it can be shown that there exist $\mu \in (0,1)$, $M>0$ and $\widetilde{m} \in \N_+$ such that 
$$
\aligned
(A,b,c, \Re(\varrho) V) &\in \cB_{\mu, M}(\Omega), \\
(A,b,c, \Re(\varrho) V_m) &\in \cB_{\mu, M}(\Omega), \quad \forall m \geq \widetilde{m}.
\endaligned
$$
Therefore, there exists $\vartheta_0 \in (0,\pi/2)$ such that the sesquilinear forms $\gota := \gota_\oA$ and $\gota_m := \gota_{\oA_m}$ are both sectorial of angle $\vartheta_0$, for all $m \geq \widetilde{m}$. In particular, each operator $\oL_m =\oL^{\oA_m}$ is sectorial of angle $\omega(\oL_m) \leq \vartheta_0$, for all $m \in \N_+ \cup\{\infty\}$ such that $m\geq \widetilde{m}$; see Section~\ref{s: Neumann introduction} and Section~\ref{s : compl pot}. 
It can be shown that the sesquilinear forms $\gota_z$ and $\gota_{m,z}$, defined by
$$
\aligned
\gota_z &:= \Re \gota + z \Im \gota,\\
\gota_{m,z} &:= \Re \gota_m + z \Im \gota_m,
\endaligned
$$
are closed and sectorial for all $z \in \cO :=\{z \in \C \, : \, |\Re z| <\delta\}$ and all $m \in \N_+$, $m \geq \widetilde{m}$, where $\delta=\cot \vartheta_0$. The real and the imaginary part of a sesquilinear form are defined as in \cite[Chapter VI, \S 1.1]{Kat}. They are both symmetric sesquilinear forms. Note that $\gota = \gota_i$.

Let $\oL_z$ and $\oL_{m,z}$ be the operators associated with $\gota_z$ and $\gota_{m,z}$, respectively. As explained in the sketch of the proof of \cite[Lemma~3.8]{CD-Potentials}, the maps $z \mapsto (s+\oL_z)^{-1}$ and $z \mapsto (s+\oL_{m,z})^{-1}$ are both holomorphic from $\cO$ to the space of bounded linear operators on $L^2(\Omega)$, for every $s >0$ and $m \in \N$, $m \geq \widetilde{m}$.

Now, we need to slightly change their proof due to the presence of complex potentials. We need to distinguish two cases. First, let suppose that $\Im(\varrho) >0$. Hence $\gota_{\widetilde{m},z} \leq \gota_{\widetilde{m}+1,z} \leq \dots$ is a monotone nondecreasing sequence of closed and sectorial symmetric forms, for all $z \in [0, \delta)$. Therefore, \cite[Theorem 3.13a, p. 461]{Kat} gives that for every $s >0$, $z \in [0, \delta)$ and $f \in L^2(\Omega)$ we have
\begin{equation}
\label{e : Vitali conv}
(s+\oL_{m,z})^{-1} f \rightarrow (s+\oL_z)^{-1}f \quad \text{in } L^2(\Omega), \quad \text{as } m \rightarrow \infty.
\end{equation}
As in \cite{CD-Potentials}, we conclude by using a vector-valued version of Vitali's theorem \cite[Theorem~A.5]{ABHN}, which implies that $(s+\oL_{m,z})^{-1} f \rightarrow (s+\oL_z)^{-1}f$ for all $z \in \cO$. In particular, by taking $z =i$ we get \eqref{e: Vitali conv I}. 

The case $\Im(\varrho) <0$ is treated similarly. In this case, \eqref{e : Vitali conv} holds for all $z \in (-\delta,0]$.

In \cite{CD-Potentials}, since the potential is real, the analogue of \eqref{e : Vitali conv}, that is, \cite[(3.19)]{CD-Potentials}, holds for all $z \in (-\delta,\delta)$.

\subsection*{Acknowledgements}

I was partially supported by the ``National Group for Mathematical Analysis, Probability and their Applications'' (GNAMPA-INdAM).

I would like to thank Andrea Carbonaro and Oliver Dragi\v{c}evi\'c for their support and guidance during the writing of this paper.

\bibliographystyle{amsxport}
\bibliography{biblio_mixed}

\end{document}